\documentclass[a4paper,twoside,11pt]{article}

\usepackage{BeaCed_article}

\title{The $Z$-invariant massive Laplacian on isoradial graphs}
\author{C\'edric Boutillier\thanks{{\small
Laboratoire de Probabilit\'es et Mod\`eles Al\'eatoires, Universit\'e Pierre et Marie Curie, 4 place Jussieu, 
F-75005 Paris.} {\small Email: \texttt{cedric.boutillier@upmc.fr}}
}\and B\'eatrice de Tili\`ere\thanks{{\small
Laboratoire d'Analyse et de Math\'ematiques Appliqu\'ees, Universit\'e Paris-Est Cr\'eteil, 61, avenue du G\'en\'eral de Gaulle
F-94010 Cr\'eteil.}
{\small Email: \texttt{beatrice.taupinart-de-tiliere@upec.fr}}
} \and Kilian Raschel\thanks{{\small
CNRS, Laboratoire de Math\'ematiques et Physique Th\'eorique, Universit\'e de Tours, Parc de Grandmont, F-37200 Tours.}
{\small Email: \texttt{kilian.raschel@lmpt.univ-tours.fr}}
}}

\date{\today}

\begin{document}

\maketitle

\begin{abstract}
We introduce a one-parameter family of massive Laplacian operators $(\Delta^{m(k)})_{k\in[0,1)}$ defined on isoradial graphs,
involving elliptic functions. We prove an explicit formula for the inverse of $\Delta^{m(k)}$, the massive Green function, 
which has the remarkable property
of only depending on the local geometry of the graph, and compute its asymptotics. We study the corresponding statistical mechanics model of 
random rooted spanning forests. We prove an explicit local formula for an infinite volume Boltzmann measure, and for the free energy of the model. We show that the model
undergoes a second order phase transition at $k=0$, thus proving that spanning trees corresponding to the Laplacian introduced by Kenyon
\cite{Kenyon3}
are critical. We prove that the massive Laplacian operators $(\Delta^{m(k)})_{k\in(0,1)}$ provide a one-parameter family of $Z$-invariant 
rooted spanning forest models. When the isoradial graph is moreover $\ZZ^2$-periodic, we consider the spectral curve 
 of the characteristic polynomial
 of the massive Laplacian. We provide an explicit parametrization of the curve and prove that it is Harnack and has genus $1$. 
 We further show that every Harnack curve of genus~$1$ with $(z,w)\leftrightarrow(z^{-1},w^{-1})$ symmetry arises from such a massive Laplacian.
\end{abstract}

\section{Introduction}

An \emph{isoradial graph} $\Gs=(\Vs,\Es)$ is a planar embedded graph such that all 
faces are inscribable in a circle of radius $1$. In this paper we introduce a one-parameter family of \emph{massive Laplacian operators} 
$(\Delta^{m(k)})_{k\in[0,1)}$ defined on infinite isoradial graphs, and study its remarkable properties. 
The massive Laplacian operator $\Delta^{m(k)}:\CC^\Vs\rightarrow\CC^\Vs$ involves elliptic functions, it is defined by
\begin{equation}\label{def:massiveLap_intro}
(\Delta^{m(k)}f)(x)=\sum_{y \sim x}\rho(\theta_{xy}\vert
k)[f(x)-f(y)]+m^2(x\vert k)f(x),
\end{equation}
where $k\in[0,1)$ is the elliptic modulus, $\theta_{xy}=\overline{\theta}_{xy} \frac{2K}{\pi}$, the constant $K=K(k)$ is the complete elliptic integral of the first kind,
and $\overline{\theta}_{xy}$ is an angle naturally assigned to the edge $xy$ in the isoradial 
embedding of $\Gs$. The conductance $\rho(\theta_{xy}\vert k)$ and the 
mass $m^2(x|k)$ are given by 
\begin{align}
&\rho(\theta_{xy}\vert k)=\sc(\theta_{xy}\vert k),\label{equ:cond_intro}\\
&m^2(x\vert k)=\sum_{y \sim x} [\Arm(\theta_{xy}\vert k)-\rho(\theta_{xy}\vert k)]
\label{equ:mass_intro},
\end{align}
where $\sc$ is one of Jacobi's twelve elliptic functions and the function $\Arm$, related to integrals of squared Jacobi elliptic 
functions, is defined in Equation \eqref{def:Abis}. More details are to be found in Section~\ref{sec:EllipticFunctions}. 

The first of the main results is an explicit formula for the inverse operator, 
namely for the \emph{massive Green function}~$G^{m(k)}$, see also Theorem~\ref{thm:expression_Green} for a detailed version.
\begin{thm}\label{thm:expression_Green_intro}
Let $\Gs$ be an infinite isoradial graph, and let $k\in(0,1)$.
Then, for every pair of vertices $x,y$ of $\Gs$, the massive Green function $G^{m(k)}(x,y)$ has the following explicit expression:
  \begin{equation}
  \label{eq:green_intro}
    G^{m(k)}(x,y) =\frac{k'}{4i\pi} \int_{\Cs_{x,y}} \expo_{(x,y)}(u\vert k) \ud u,
  \end{equation}
  where $k'=\sqrt{1-k^2}$, $\expo_{(x,y)}(\cdot\vert k)$ is the discrete massive exponential function, 
  $\Cs_{x,y}$ is a vertical contour on the torus $\TT(k)= \CC /(4K\ZZ+4iK'\ZZ)$ whose direction is given by the angle of the ray 
  $\RR \overrightarrow{xy}$.
\end{thm}  
Before describing the context of Theorem~\ref{thm:expression_Green_intro}, let us give its main features.
\begin{itemize}
\item The discrete exponential function $\expo_{(x,y)}(\cdot\vert k)$ is defined in Section~\ref{sec:expofunction} using a path of the 
embedded graph from $x$ to $y$. 
This implies that the expression \eqref{eq:green_intro} for 
$G^{m(k)}(x,y)$ is \emph{local}, meaning that it remains unchanged if the isoradial graph $\Gs$
is modified away from a path from $x$ to $y$.
This is a remarkable fact since, when computing the inverse of a discrete operator, 
one expects the geometry of the whole graph to be involved. 
 \item There is no periodicity assumption on the graph $\Gs$.
  \item The discrete massive exponential function is explicit and has a product structure;
 it has identified poles, so that computations can be performed using the residue theorem, see Appendix~\ref{app:comput_green} 
 for some examples. 
  \item The explicit expression \eqref{eq:green_intro} is suitable for asymptotic analysis. Using a saddle-point 
  analysis, we prove explicit exponential decay of the Green function, see Theorem~\ref{thm:asymp_Green}. 
\end{itemize}

\paragraph{Context.}
Local formulas for inverse operators have first been proved in \cite{Kenyon3}. Kenyon considers
two operators on isoradial graphs: the Laplacian with conductances $(\tan(\overline{\theta}_e))_{e\in\Es}$ and 
the Kasteleyn operator arising from the bipartite dimer model with edge-weights $(2\sin(\overline{\theta}_e))$. In the same vein, the first two authors
of this paper proved a local formula for the inverse Kasteleyn operator of a non-bipartite dimer model corresponding to the critical
Ising model defined on isoradial graphs~\cite{BoutillierdeTiliere:iso_perio}. 

The two papers \cite{Kenyon3,BoutillierdeTiliere:iso_perio} have the common feature of considering \emph{critical} models:
the polynomial decay of the inverse Kasteleyn operator proves that the bipartite dimer model is indeed critical; Ising weights of 
\cite{BoutillierdeTiliere:iso_perio} have recently
been proved to be critical \cite{Li:critical,CimasoniDuminil,Lis}; Laplacian conductances are called critical (although it not so clear from 
\cite{Kenyon3} why they should be). This led to the belief that existence of a local formula
for an inverse operator is related to the geometric property of the embedded isoradial graph \emph{and} criticality of the 
underlying model. In this paper, we go further and prove a local formula for a one-parameter family of \emph{non-critical} models.
Indeed, underlying the massive Laplacian is the model of rooted spanning forests, which is \emph{not} critical, as explained in Section~\ref{sec:statmech}.

The idea of the proof of the local formula for the inverse of the Laplacian
operator $\Delta$ given in \cite{Kenyon3}
is the following: find a
one-parameter family of local complex-valued functions in the kernel of
$\Delta$, define its inverse $G$ as a contour integral of these
functions against a singular function, and choose the contour of integration in such a way that $\Delta
G=\Id$. 
The
problem is that this proof neither provides a way of choosing the weights of the
operator, nor a criterion for existence of a one-parameter family of local
functions, nor a way to find them, if they exist. This is why one of
the main contributions of this paper is to actually define a one-parameter
family of weights for the massive Laplacian, and to find local functions in its
kernel, which allow to prove a local formula for its inverse.

Note that when the parameter $k$ is equal to $0$, the mass \eqref{equ:mass_intro} is $0$, the elliptic function
$\sc(\theta)$ becomes $\tan(\theta)$, and we recover the Laplacian considered in~\cite{Kenyon3}. In this case, 
the discrete massive exponential function becomes the exponential function
introduced in~\cite{Mercat:exp} and used in 
the local formula for the Green function of~\cite{Kenyon3}.

\paragraph{Random rooted spanning forests.}

The massive Laplacian operator is naturally related to the statistical
   mechanics model of \emph{rooted spanning forests}. Indeed, 
when the graph $\Gs$ is finite, by Kirchhoff's matrix-tree theorem, the determinant of $\Delta^{m(k)}$ is the partition function 
$\Zforest^k(\Gs)$, \emph{i.e.}, the weighted sum of rooted
spanning forests of $\Gs$, whose weights depend on the
conductances~\eqref{equ:cond_intro} and masses~\eqref{equ:mass_intro}. 
In Section~\ref{sec:statmech}, we prove the following results.

$\bullet$ Theorem~\ref{thm:infinite_vol_meas} proves an explicit expression for an infinite volume 
 rooted spanning forest Boltzmann measure of the graph $\Gs$, involving the massive Laplacian matrix and the massive Green function
 of Theorem~\ref{thm:expression_Green_intro}.
 The proof follows the approach of~\cite{BurtonPemantle}. This measure inherits the \emph{locality} property of 
 Theorem~\ref{thm:expression_Green_intro}, \emph{i.e.}, the probability that a finite subset of edges/vertices belongs to a rooted spanning 
 forest is unchanged if the graph is modified away from this subset.
 
$\bullet$ Assume that the infinite isoradial graph $\Gs$ is $\ZZ^2$-periodic, and consider the natural exhaustion 
$(\Gs_n=\Gs/n\ZZ^2)_{n\geq 1}$ of $\Gs$
 by toroidal graphs. The \emph{free energy} $F_{\mathrm{forest}}^k$ of the spanning forest model is minus the exponential 
 growth rate of the partition function $\Zforest^k(\Gs_n)$, as $n$ tends to infinity. We prove an explicit formula for $F_{\mathrm{forest}}^k$, see
 also Theorem~\ref{thm:free_energy}. It has the property of not involving the combinatorics of the graph. Indeed it is a sum over edges of
 the graph $\Gs_1$ of quantities only depending on the angle $\theta_e$ assigned to the edge $e$ in the isoradial embedding.

\begin{thm}
\label{thm:free_energy_intro}
For every $k\in(0,1)$, the free energy $F^k_{\mathrm{forest}}$ of the rooted spanning forest model on the infinite, $\ZZ^2$-periodic, isoradial graph $\Gs$,
is equal to:
\begin{equation}\label{equ:free_energy_intro}
  F^k_{\mathrm{forest}}=
  \vert \Vs_1 \vert \int_0^K 4H'(2\theta)\log\sc(\theta)\ud\theta+
  \sum_{e\in\Es_1}
  \int_{0}^{\theta_e}\frac{2H(2\theta)\sc'(\theta)}{\sc(\theta)}\ud\theta,
\end{equation}
where $H$ is the function defined in Equation~\eqref{def:H}. 
\end{thm}

$\bullet$ When $k=0$, $F^0_{\mathrm{forest}}$ is equal to the normalized determinant of the Laplacian 
operator of~\cite{Kenyon3}; it is also the free energy of the corresponding spanning tree model. Performing 
an asymptotic expansion around $k=0$ of~\eqref{equ:free_energy_intro},
we prove in Theorem~\ref{thm:free_energy_k0} that the rooted spanning forest model has a \emph{second order phase transition at $k=0$}. 
In particular, this gives a proof that the spanning tree model corresponding to the Laplacian considered
in \cite{Kenyon3} is indeed \emph{critical}.  
Note that the non-analyticity of the free energy at $k=0$ does not come from that of the weights or masses. 
Indeed, the latter are analytic around the origin, see Lemma \ref{lem:analytic_weight}.

$\bullet$ 
Recall that the infinite volume rooted spanning forest Boltzmann
measure inherits the locality property of Theorem~\ref{thm:expression_Green_intro}. 
From the point of view of statistical mechanics, this specific feature is expected from models defined on isoradial graphs that are \emph{$Z$-invariant}. 
Although already present in the work of Kenelly~\cite{Kennelly} and
Onsager~\cite{Onsager},
the notion of $Z$-invariance has been extensively developed by Baxter, see~\cite{Baxter:8V,Baxter:Zinv} and also \cite{Perk:McCoy,Perk:YB,Kenyon6}.
$Z$-invariance imposes a strong locality constraint on the model: invariance of the partition function under star-triangle moves,
see Figure~\ref{fig:figStarTriangle1} and Section~\ref{sec:deftriangleetoile} for definition, or equivalently invariance of the probability measure 
under these moves. This suggests a locality property of the measure, but it does not provide a way of finding explicit local formulas.
Using 3-dimensional consistency of the massive Laplacian operator (Proposition~\ref{prop:3Dconsistency}), we prove the following, 
see also Theorem~\ref{thm:Z_invariant_Lap}.

\begin{thm}
For every $k\in[0,1)$, the statistical mechanics model of rooted spanning forests on isoradial graphs,
with conductances \eqref{equ:cond_intro} and masses \eqref{equ:mass_intro},
is $Z$-invariant.
\end{thm}

\paragraph{The case of periodic isoradial graphs, Harnack curves of genus $1$.}

Suppose further that the isoradial graph $\Gs$ is $\ZZ^2$-periodic. The \emph{massive Laplacian characteristic polynomial}, denoted 
$P_{\Delta^{m(k)}}(z,w)$, is the determinant of the matrix $\Delta^{m(k)}(z,w)$, which is the matrix of the massive 
Laplacian $\Delta^{m(k)}$
restricted to the graph $\Gs_1$ with extra weights $z^{\pm 1},w^{\pm 1}$ along non trivial cycles of the torus. Of particular
interest is the zero locus of this polynomial, otherwise known as the \emph{spectral curve of the massive Laplacian}:
$\Ccal^k=\{(z,w)\in\CC^2:P_{\Delta^{m(k)}}(z,w)=0\}$. We provide an explicit parametrization of this curve, and combining
Proposition~\ref{prop:geom_torus} and Theorem~\ref{thm:harnack}, we prove that this curve has remarkable properties.
\begin{thm}
For every $k\in(0,1)$, the spectral curve $\Ccal^k$ of the massive Laplacian $\Delta^{m(k)}$ is a Harnack curve of genus $1$.
\end{thm}
This is reminiscent of the rational parametrization of critical dimer spectral curves
on periodic, bipartite, isoradial graphs of~\cite{KO2}, corresponding to the genus 0 case.
We further prove the following result, see also Theorem~\ref{thm:Harnack2}.
\begin{thm}\label{thmHarnack2_intro}
Every Harnack curve with $ (z,w)\leftrightarrow(z^{-1},w^{-1})$ symmetry arises as the 
spectral curve of the characteristic polynomial of the massive Laplacian $\Delta^{m(k)}$ on some periodic isoradial graph, for some
$k\in(0,1)$. 
\end{thm}
This can be compared to the fact proved in~\cite{KO2} that any genus $0$ Harnack curve, whose amoeba contains the origin,
is the spectral curve of a critical
dimer model on a bipartite isoradial graph.

Since the spectral curve $\Ccal^k$ has genus $1$, the amoeba's complement has a single bounded component. 
In Proposition~\ref{prop:area_hole}, we prove that the area of the bounded component grows continuously from $0$ to $\infty$ as
$k$ grows from $0$ to $1$.

Using the Fourier approach, the massive Green function can be expressed using the characteristic polynomial.
This approach also works for other choices of weights, and one cannot see from the formula that the locality property is satisfied.
In Section~\ref{sec:recovLocal}, we relate the Fourier approach and Theorem~\ref{thm:expression_Green_intro} by proving
that our choice of weights allow for an astonishing change of variable. Note that this relation was not understood in
the papers~\cite{Kenyon3,BoutillierdeTiliere:iso_gen,deTiliere:quadri}.

\paragraph{Outline of the paper.}
\begin{itemize}
 \item \textbf{Section~\ref{sec:Generalities}: Generalities}. Review of main notions underlying the paper: isoradial graphs and elliptic functions.
 \item \textbf{Section~\ref{sec:massive_Lap}: Massive Laplacian on isoradial graphs}. Introduction of the one-parameter family of massive Laplacian 
 operators $(\Delta^{m(k)})$,
 depending on the elliptic modulus $k\in[0,1)$. Proof of $3$-dimensional consistency. Definition of the 
 discrete massive exponential function. Proof that it defines a family of massive harmonic functions.
 \item \textbf{Section~\ref{sec:massiveGreen}: Massive Green function on isoradial graphs}. Theorem~\ref{thm:expression_Green} proves the local formula for the massive Green function $G^{m(k)}$,
 and Theorem~\ref{thm:asymp_Green} proves asymptotic exponential decay.
 \item \textbf{Section~\ref{sec:periodic_case}: The case of $\ZZ^2$-periodic isoradial graphs}. Definition of the characteristic
 polynomial of the massive Laplacian operators, of the Newton polygon of the characteristic polynomial. 
 Proof of confinement results for the Newton polygon. Definition of the spectral curve $\Ccal^k$ and its amoeba. Explicit parametrization
 of the spectral curve and proof that it has geometric genus $1$. Theorem~\ref{thm:harnack} shows that the spectral curve $\Ccal^k$ 
 is Harnack and Theorem~\ref{thm:Harnack2} proves that every genus $1$, Harnack curve with $(z,w)\leftrightarrow(z^{-1},w^{-1})$ symmetry
 arises from such a massive Laplacian.
 Consequences of the Harnack property on the amoeba of the spectral curve. Proof of the growth of the area of the bounded component of the 
 amoeba's complement. Derivation of the local formula of Theorem~\ref{thm:expression_Green} from the Fourier approach. 
 Asymptotics of the Green function using the approach of~\cite{PeWi13}. 
 \item \textbf{Section~\ref{sec:statmech}: Random rooted spanning forests on isoradial graphs}. Definition of the statistical mechanics model of rooted spanning forests. 
 Theorem~\ref{thm:infinite_vol_meas} proves an explicit, \emph{local}
 expression for an infinite volume Boltzmann measure involving the Green function of Theorem~\ref{thm:expression_Green}. 
 Theorem~\ref{thm:free_energy} proves an explicit, \emph{local} expression for the free energy of the model, and 
 Theorem~\ref{thm:free_energy_k0} shows a second order phase transition at $k=0$ in the rooted spanning forest model. At $k=0$, one 
 recovers the Laplacian considered in~\cite{Kenyon3}. We thus provide a proof that the corresponding spanning tree model is critical. 
 Theorem~\ref{thm:Z_invariant_Lap} proves that our one-parameter family of massive Laplacian defines a one-parameter family of 
 $Z$-invariant spanning forest models.
\item \textbf{Sections~\ref{app:elliptic},~\ref{app:comput_green},~\ref{app:stt} and~\ref{app:stat_mec}}. Appendices for elliptic functions, explicit
computations of the massive Green function, $Z$-invariance, rooted spanning forests and random walks.
\end{itemize}

\medskip

\emph{Acknowledgments:} We warmly thank Erwan Brugall\'e for very helpful discussions on Harnack curves.  
We acknowledge support from the Agence Nationale de la Recherche (projet  MAC2: ANR-10-BLAN-0123) and from the 
R\'egion Centre-Val de Loire (projet MADACA). We thank the referee for her/his useful comments
and suggestions which led us to improve the presentation of this paper. We also thank her/him for the arguments allowing to greatly simplify the proof of 
Proposition~\ref{prop:neg_mass} and of Theorem~\ref{thm:Z_invariant_Lap}.

\section{Generalities}
\label{sec:Generalities}

In this section we review two of the main notions underlying this work: isoradial graphs and elliptic functions.

\subsection{Isoradial graphs}\label{sec:isoradial}

Isoradial graphs, whose name comes from the paper~\cite{Kenyon3}, see also~\cite{Duffin,Mercat:ising} are defined as follows. 
A planar graph $\Gs=(\Vs,\Es)$ is
\emph{isoradial}, if it can be embedded in the plane in such a way that all internal faces are
inscribable in a circle, with all circles having the same radius, and such that all circumcenters are
in the interior of the faces, see Figure~\ref{fig:Iso1} (top left).
From now on, when we speak of an isoradial graph $\Gs$, we mean an isoradial graph \emph{together} with an isoradial embedding
also denoted by $\Gs$. Given an infinite isoradial graph $\Gs$, an isoradial embedding of the dual graph $\Gs^*$ is obtained by taking as dual vertices
the circumcenters of the corresponding faces, see Figure~\ref{fig:Iso1} (bottom
left).

\begin{figure}[ht!]
  \begin{center}
    \begin{tabular}{cc}
      \includegraphics[angle=90,width=70mm]{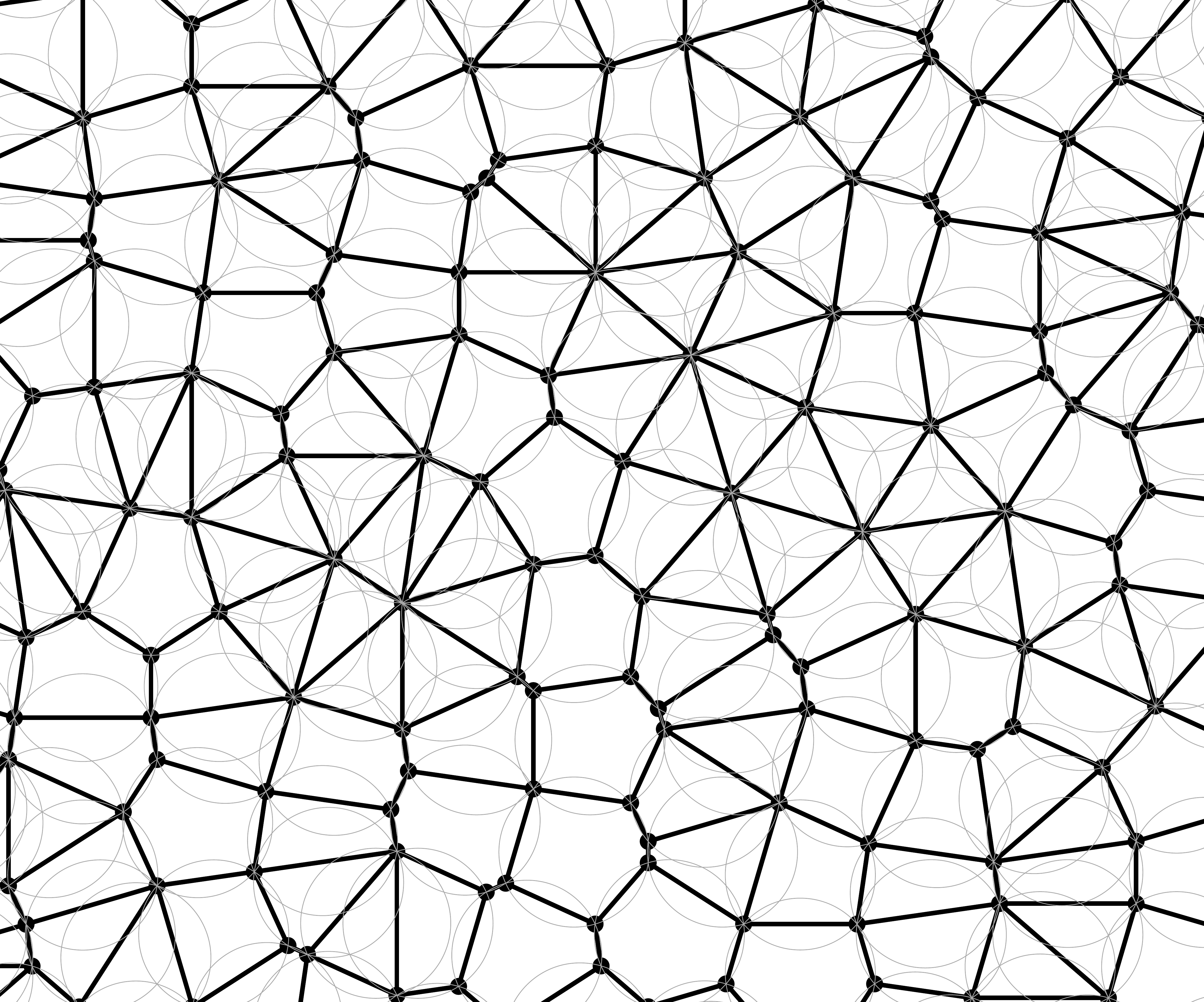} &
      \includegraphics[angle=90,width=70mm]{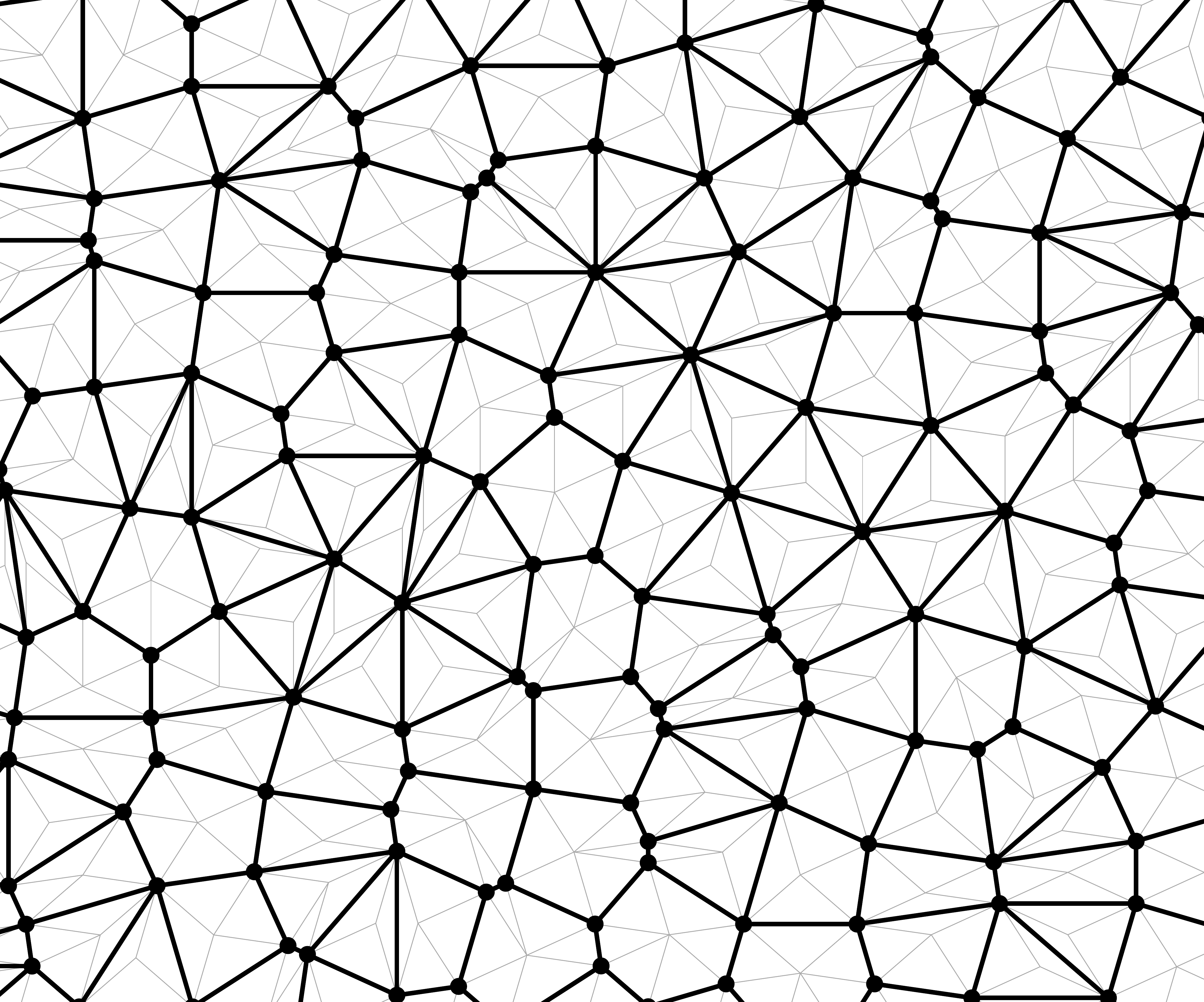} \\
      \includegraphics[angle=90,width=70mm]{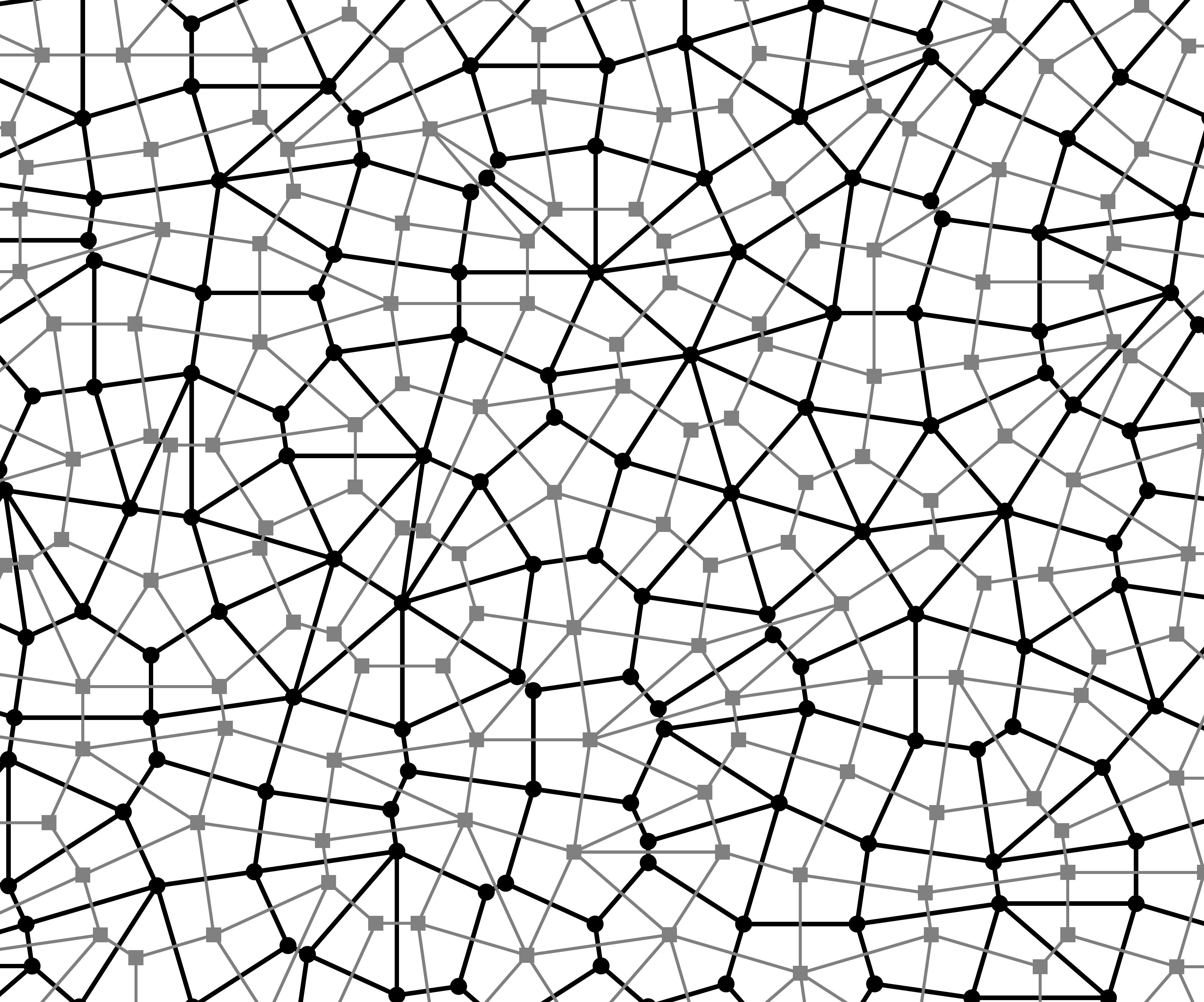} &
      \includegraphics[angle=90,width=70mm]{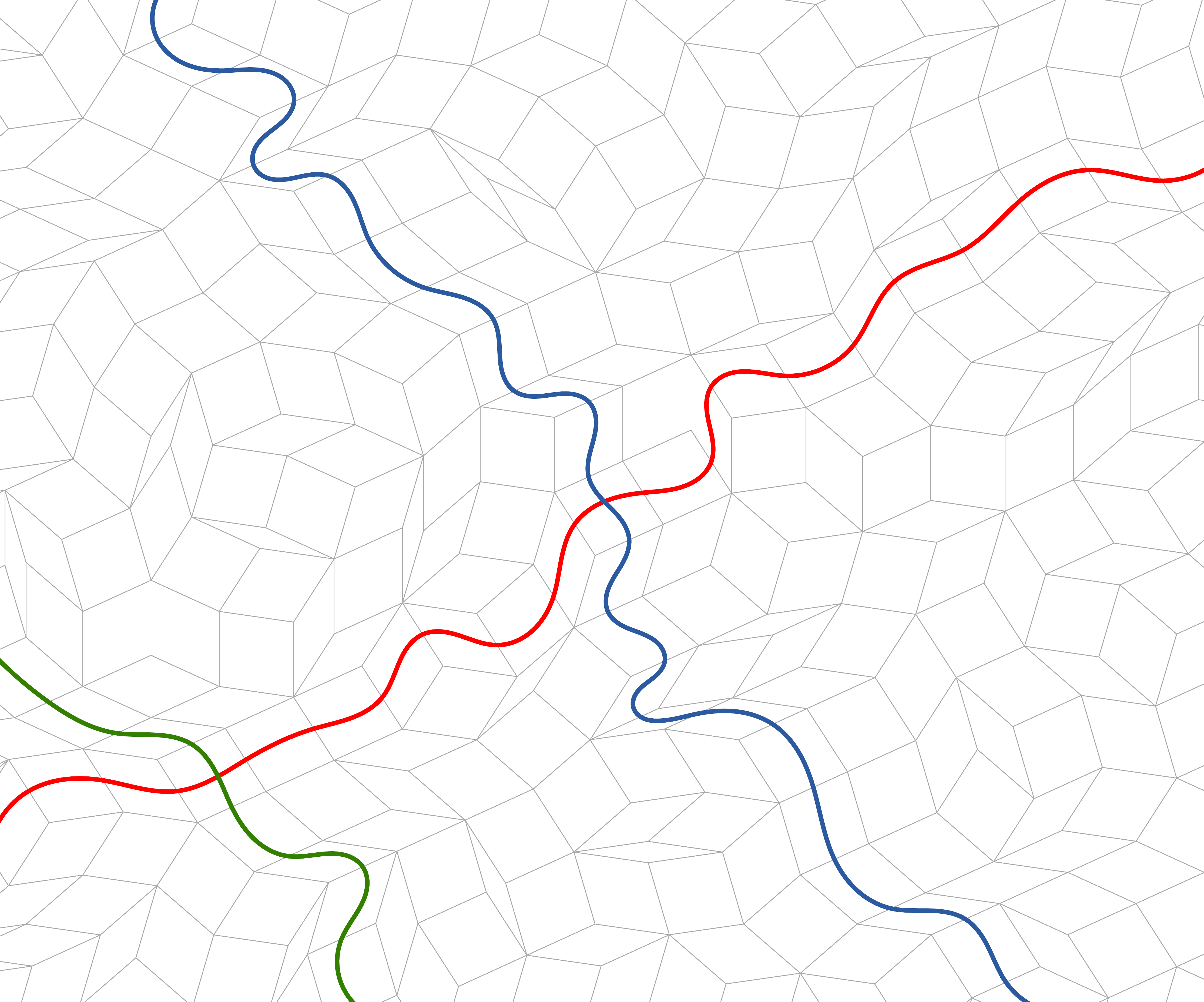} 
    \end{tabular}
    \caption{
     Top left: piece of an infinite isoradial graph $\Gs$ (in black) with the circumcircles of the faces.
      Top right: the same piece of infinite graph $\Gs$ 
      with its diamond graph~$\GR$. Bottom left: the isoradial graph
      superimposed with its dual graph, whose vertices are the centers of the
      circumcircles. Bottom right: the diamond
      graph with a few train-tracks pictured as paths of the dual graph of $\GR$.}
    \label{fig:Iso1}
  \end{center}
\end{figure}


\subsubsection{Diamond graph, angles and train-tracks}\label{sec:diamondetc}

The \emph{diamond graph}, denoted $\GR$, is constructed from
an isoradial graph $\Gs$ and its dual $\Gs^*$.
Vertices of $\GR$ are those of~$\Gs$ and those of $\Gs^*$. A dual vertex of $\Gs^*$ is joined to all primal
vertices on the boundary of the corresponding face, see Figure~\ref{fig:Iso1}
(top right). Since edges of the diamond graph $\GR$ are radii of circles, 
they all have length $1$, and can be assigned a direction $\pm e^{i\overline{\alpha}}$. Note that faces of $\GR$ are 
side-length $1$ rhombi.

Using the diamond graph, angles can naturally be assigned to edges of the graph~$\Gs$ as follows.
Every edge $e$ of $\Gs$ is the diagonal of exactly one
rhombus of $\GR$, and we let $\overline{\theta}_e$ be the half-angle at
the vertex it has in common with $e$, see Figure~\ref{fig:rhombus_angle}.
Note that we have $\overline{\theta}_e\in(0,\frac{\pi}{2})$, because circumcircles are assumed to be in the interior of the faces.
From now on, we actually ask more and suppose that there exists $\eps>0$, such that $\overline{\theta}_e\in(\eps,\frac{\pi}{2}-\eps)$.
We also assign two rhombus vectors to the edge $e$, denoted 
$e^{i\overline{\alpha}_e}$ and $e^{i\overline{\beta}_e}$, see Figure~\ref{fig:rhombus_angle}, and
we assume that $\overline{\alpha}_e,\overline{\beta}_e$ satisfy $\frac{\overline{\beta}_e-\overline{\alpha}_e}{2}=
\overline{\theta}_e$.

\begin{figure}[h]
\begin{center}
\resizebox{0.18\textwidth}{!}{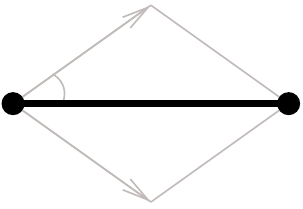}
\end{center}
\caption{An edge $e$ of $\Gs$ is the diagonal of a rhombus of $\GR$, defining the angle $\overline{\theta}_e$ and the rhombus vectors
$e^{i\overline{\alpha}_e}$ and $e^{i\overline{\beta}_e}$.}
\label{fig:rhombus_angle}
\end{figure}

A \emph{train-track} of an infinite isoradial graph~$\Gs$ is a 
bi-infinite chain of adjacent rhombi of $\GR$ which does not turn: on entering a face, it exits along the opposite 
edge~\cite{KeSchlenk}. 
As a consequence, each rhombus in a train-track has an edge parallel to a fixed direction 
$\pm e^{i\overline{\alpha}}$, known as the \emph{direction of the train-track}. 
Train-tracks are also known as \emph{de Bru{\ij}n lines} in the field of
non-periodic tilings~\cite{deBruijn1, deBruijn2}, or \emph{rapidity lines} in
integrable systems~\cite{Baxter:Zinv}; the terminology \emph{line} refers to the
representation of
train-tracks as paths of the dual graph of $\GR$, see Figure~\ref{fig:Iso1}
(bottom right).
In \cite{KeSchlenk}, they are used to give a necessary and sufficient condition for a 
planar graph to have an isoradial embedding.

A train-track is said to \emph{separate} two vertices $x$ and $y$ of $\GR$ if every path connecting $x$ and $y$ crosses this train-track. A path
from $x$ to $y$ in $\GR$ is said to be \emph{minimal} if all its edges cross train-tracks separating $x$ from $y$, and 
each such train-track is crossed exactly once.
An example of minimal path and non-minimal one is given in Figure~\ref{fig:Isoradial18}.

\subsubsection{Isoradial graphs as monotone surfaces of the hypercubic lattice}\label{def:quasi}

An isoradial graph $\Gs$ is said to be \emph{quasicrystalline} if the number
$\ell$ of possible directions $\pm e^{i\overline{\alpha}}$ assigned to edges of its
diamond graph $\GR$ is finite; $\ell$ is known as the \emph{dimension} of the
isoradial graph.
The degree of a vertex of $\Gs$ is at most $2\ell$, and
at a vertex of its diamond graph $\GR$, there can be edges with direction
$\pm e^{i\overline{\alpha}_1},\dotsc,\pm e^{i\overline{\alpha}_\ell}$. 
The graph $\GR$ can then be seen as the
projection of a monotone surface in $\mathbb{Z}^\ell$, see \cite{Thurston}
for $\ell=3$, and also for example~\cite{BobenkoMercatSuris,BodiniFerniqueRemila}, where the lattice $\mathbb{Z}^\ell$ is 
spanned by unit vectors $e_1,\dotsc, e_\ell$, \emph{i.e.}, the image by the linear map $e_j\mapsto e^{i\overline\alpha_j}$.
%
%
%
Rhombic faces of $\GR$ are images of square $2$-faces of $\mathbb{Z}^\ell$.
Since the surface is monotone, any path on the graph $\GR$ can be lifted to a
nearest-neighbor path in $\mathbb{Z}^\ell$.

\subsubsection{Natural operations on isoradial graphs}\label{sec:NaturalOperations}


\paragraph{Train-track tilting.}
Recall that a direction $\pm e^{i\overline{\alpha}}$ is assigned to every train-track of $\Gs$. 
If we slightly change the angle $\overline{\alpha}$, so that none of the
rhombi of the train-track becomes flat during the deformation, we get a new isoradial
embedding of the graph $\Gs$. The structure of the graph has not changed, however,
if quantities are defined through geometric characteristics of the isoradial
embedding (\emph{e.g.}, the angles of the rhombi as is the case in this
article), then this operation provides a continuous one-parameter family of
transformations for these quantities. This operation is called \emph{train-track
tilting}. It is introduced in \cite{Kenyon3} and used in the proof of Theorem~\ref{thm:free_energy} in
Section~\ref{sec:free_energy}.

\paragraph{Star-triangle transformation.}
If $\Gs$ has a \emph{star}, that is
a vertex of degree $3$, it can be replaced by a
\emph{triangle} by removing the vertex and connecting its three neighbors. The
graph obtained in this way is still isoradial: its diamond graph is obtained by
performing a \emph{cubic flip} in $\GR$,
see Figure~\ref{fig:figStarTriangle1}. This operation is involutive.

\begin{figure}[ht]
\begin{center}
\resizebox{0.5\textwidth}{!}{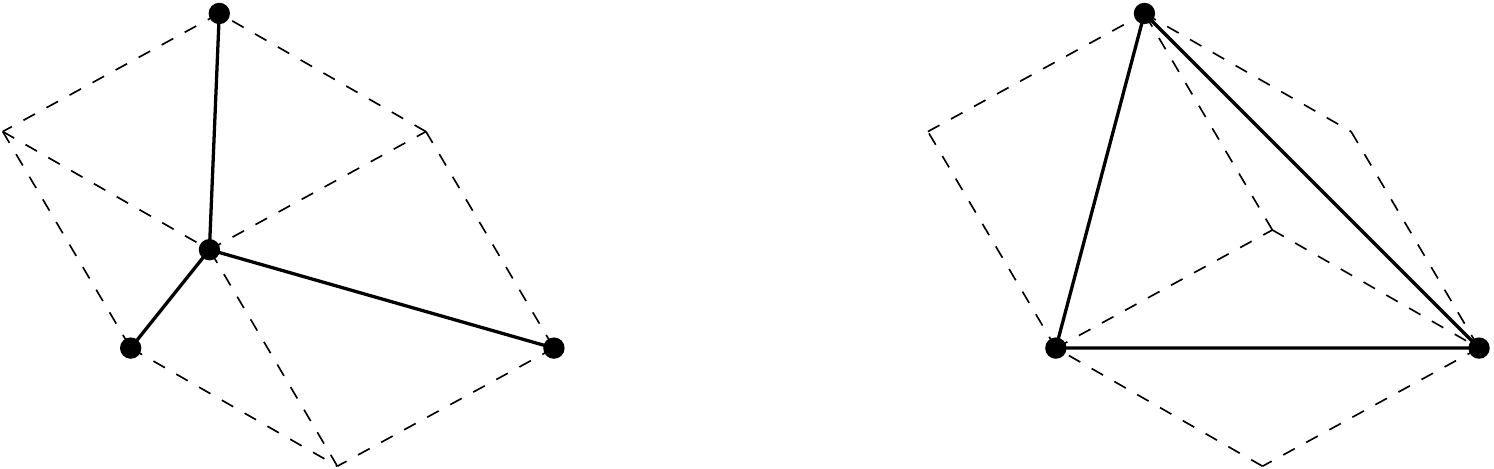}
\end{center}
\caption{Star-triangle transformation in an isoradial graph $\Gs$ and underlying cubic-flip in the diamond graph $\GR$
\label{fig:figStarTriangle1}}
\end{figure}


The star-triangle transformation is locally transitive in the following sense: if $B$ is a
bounded, connected domain obtained as the union of faces of $\GR$, then any
other tiling of $B$ with rhombi of the same edge-length can be obtained from the
initial one by a sequence of cubic flips~\cite{KenyonParall}. As a consequence,
two isoradial graphs coinciding outside of a bounded domain can be transformed
into one another by a sequence of star-triangle transformations.


This operations has a natural geometric interpretation from the monotone surface
point of view: a cubic flip corresponds to locally deforming the monotone surface
so that it uses different $2$-faces to go around the same $3$-cube of
$\ZZ^\ell$. 

If $\Gs$ has no location where such an operation can be performed, it means that
there is no triple of train-tracks intersecting eachother. However, if a pair
of train-tracks are going to infinity by staying at distance one in ${\GR}^*$,
then we can insert a rhombus ``at infinity'' to create a
location where to perform this transformation.

This operation, connected to the third Reidemeister move in knot theory, plays an
important role in integrable systems in two dimensions, and is closely related
to the Yang-Baxter equations~\cite{Perk:YB}.

\subsection{Elliptic functions}
\label{sec:EllipticFunctions}

This article strongly relies on Jacobi elliptic functions, which we now present.
Useful formulas are given in Appendix~\ref{app:elliptic},
our reference is the book of Lawden \cite{La89} and the one of Abramowitz and Stegun \cite{AS}. 

\paragraph{Elliptic modulus and quarter periods.}

Let $k\in[0,1]$, referred to as the \emph{elliptic modulus}, and let $k'=\sqrt{1-k^2}$ be the complementary elliptic modulus. 
The \emph{complete elliptic integral of the first kind}, denoted $K=K(k)$, and the 
\emph{complete elliptic integral of the second kind}, denoted $E=E(k)$, are defined by:
\begin{equation*}
     K=K(k)=\int_{0}^{{\pi}/{2}}\frac{1}{\sqrt{1-k^2 \sin^2 \tau}}\,\ud\tau,\quad 
     E=E(k)=\int_0^{{\pi}/{2}} \sqrt{1-k^2 \sin^2\tau}\,\ud\tau.
\end{equation*}
The complementary integrals are $K'=K'(k)=K(k')$ and $E'=E'(k)=E(k')$.
They satisfy Legendre's identity \cite[3.8.29]{La89}:
\begin{equation}
\label{eq:Legendre}
     EK'+E'K-KK'=\frac{\pi}{2}.
\end{equation}
\paragraph{Jacobi elliptic functions.}
There are twelve Jacobi elliptic functions, each of them corresponds to an arrow
drawn from one corner of a rectangle to another, see
Figure~\ref{Rectangle_elliptic}. The corners of the rectangle are labeled, by
convention, ${\rm s}$, ${\rm c}$, ${\rm d}$ and ${\rm n}$. These points
respectively correspond to the origin $0$, $K$ on the real axis, $K + iK'$, and
$iK'$ on the imaginary axis. The numbers $K$ and $iK'$ are called
the quarter periods. The twelve Jacobi elliptic functions are then $\pq(\cdot\vert k)$, where each of ${\rm p}$ and ${\rm q}$ is a different one of the
letters ${\rm s}$, ${\rm c}$, ${\rm d}$, ${\rm n}$. The Jacobi elliptic
functions are then the unique doubly periodic, meromorphic functions on $\mathbb
C$, satisfying the following properties \cite[Chapter 16]{AS}:
\begin{itemize}
  \item There is a simple zero at the corner ${\rm p}$, and a simple pole at the
    corner ${\rm q}$.
  \item The step from ${\rm p}$ to ${\rm q}$ is equal to half a period of the
    function $\pq(\cdot\vert k)$. The function $\pq(\cdot\vert k)$ is also periodic
    in the other two directions, with a period such that the distance from ${\rm
    p}$ to one of the other corners is a quarter period.
  \item The coefficient of the leading term in the expansion of $\pq(u\vert k)$ in ascending powers of $u$ about $u=0$ is $1$. In other words,
    the leading term is $u$, $1/u$ or $1$, according to whether $u=0$ is a zero, a pole
    or an ordinary point.
\end{itemize}

For instance, the function $\sc(\cdot\vert k)$ (which is the most important Jacobi
elliptic function here) has a simple zero at $0$ (with residue $1$), a simple
pole at $K$, and is doubly periodic with periods $2K$ and $4iK'$.


\unitlength=0.4cm
\begin{figure}[ht]
\vspace{25mm}
\begin{center}
\begin{tabular}{cccc}
\begin{picture}(0,0)(0,0)
\put(-6,0){\line(1,0){12}}
\put(-6,0){\line(0,1){3}}
\put(6,0){\line(0,1){3}}
\put(-2,6){\line(1,0){4}}
\put(2,6){\line(1,0){4}}
\put(-6,6){\line(1,0){4}}
\put(-6,3){\line(0,1){3}}
\put(6,3){\line(0,1){3}}
\put(-6.15,-0.15){{$\bullet$}}
\put(-6.5,-0.5){${\rm s}$}
\put(-5.8,0.5){$(0)$}
\put(4.2,0.5){$(K)$}
\put(-5.8,5){$(iK')$}
\put(1.7,5){$(K+iK')$}
\put(6.2,-0.5){${\rm c}$}
\put(6.2,6.3){${\rm d}$}
\put(-6.5,6.3){${\rm n}$}
\put(-6.15,5.85){{$\bullet$}}
\put(5.85,-0.15){{$\bullet$}}
\put(5.85,5.85){{$\bullet$}}
\end{picture}
\end{tabular}
\end{center}
\caption{The rectangle $[0,K]+[0,iK']$ and the four corners
labeled ${\rm s}, {\rm c}, {\rm d}, {\rm n}$.}
\label{Rectangle_elliptic}
\end{figure}
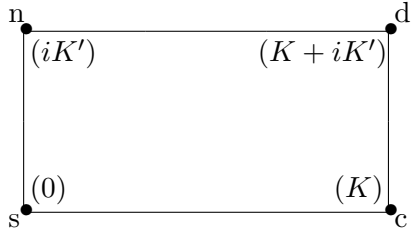

Jacobi functions $\pq(\cdot\vert k)$ also satisfy anti-periodicity relations: if $2L\in\{2K,2iK',2K+2iK'\}$ is not a period, 
then $\pq(\cdot+2L\vert k)=-\pq(\cdot\vert k)$, see \cite[16.2 and 16.8]{AS}.

\paragraph{Degenerate elliptic functions.}

Elliptic functions contain as limiting cases trigonometric functions ($k=0$)
and hyperbolic functions ($k=1$). For instance, $\sc$ degenerates for $k=0$ into
$\tan$ and for $k=1$ into $\sinh$; $\dn$ degenerates for $k=0$ to $1$, see \cite[16.6]{AS}.
Note that one of the
periods goes to infinity: for $k=0$ we have $K=\pi/2$ and
$K'=\infty$, while for $k=1$, $K=\infty$ and $K'=\pi/2$; explaining why the
limit functions are periodic in one direction only.

From now on, we suppose that the elliptic modulus $k$ is in $[0,1)$.

\paragraph{Integrals of squared Jacobi elliptic functions.} 
Following \cite[16.25.1]{AS}, we introduce
\begin{equation}
\label{eq:jacob_Dc}
    \forall\,u\in\CC,\quad\Dc(u\vert k)=\int_0^{u} \dc^2(v\vert k)\,\ud v.
\end{equation}
Since $\dc^2(\cdot\vert k)$ has no residue at its poles, the function $\Dc(\cdot\vert k)$ is meromorphic on $\mathbb C$. It is
related to Jacobi 
epsilon 
function \cite[16.26.7]{AS}.

The definition of the massive Laplacian of
Section~\ref{sec:massive_Lap}
involves the function $\Arm(\cdot\vert k)$, defined as
\begin{equation}\label{def:Abis}
    \forall\,u\in\CC,\quad\Arm(u\vert k)=\frac{1}{k'}\left(\Dc(u\vert k)+\frac{E-K}{K}u\right).
\end{equation}
The function $\Arm(\cdot\vert k)$ is periodic in the direction $2K$ and quasi-periodic in $2iK'$, 
see~\eqref{cor:Armbis:item1} and~\eqref{cor:Armbis:item1bis}.

The explicit expression of the Green function of Theorem~\ref{thm:expression_Green} involves the function $H(\cdot\vert k)$, 
defined from the function $\Arm(\cdot\vert k)$ by
\begin{equation}
\label{def:H}
     \forall\,u\in\CC,\quad H(u\vert k)=  \frac{-ik'K'}{\pi}\Arm\Bigl(\frac{iu}{2}\Big\vert k'\Bigr).
\end{equation}

Properties and identities satisfied by the functions $\Arm(\cdot\vert k)$ and $H(\cdot\vert k)$ are stated
in Lemmas~\ref{cor:Armbis}  and~\ref{lem:h} of Appendix~\ref{app:ellipticAH}.

\paragraph{One-parameter family of angles.}
Finally, we define a one-parameter family of angles, depending on the elliptic
modulus. For every $k\in[0,1)$ and every edge $e$ of $\Gs$,
\begin{equation*}
 \theta_e=\overline{\theta}_e \frac{2K}{\pi}\in(0,K),\quad\alpha_e=\overline{\alpha}_e \frac{2K}{\pi},\quad \beta_e=\overline{\beta}_e \frac{2K}{\pi}.
\end{equation*}

Since the elliptic modulus is fixed, the dependence in $k$ is not made explicit in the notation $\theta_e$, $\alpha_e$, $\beta_e$.

\section{Massive Laplacian on isoradial graphs}
\label{sec:massive_Lap}

In Section~\ref{sec:defLap}, we introduce a one-parameter family $(\Delta^{m(k)})_{k\in[0,1)}$ of 
massive Laplacian operators defined on an infinite isoradial graph $\Gs$, involving elliptic functions. 
We prove that the mass is non-negative and that the conductances and mass are analytic at $k=0$. 
Then, in Section~\ref{sec:integrability}, we show
that the equation $\Delta^{m(k)}f=0$ satisfies \emph{$3$-dimensional consistency}. Finally, in Section~\ref{sec:expofunction}, we introduce
the \emph{discrete $k$-massive exponential function}, which induces a family of massive harmonic functions. The latter play a key role
in the explicit formula for the massive Green function.

In the whole of this section, we let $\Gs$ be an infinite isoradial graph, and we fix an elliptic modulus $k\in[0,1)$.
Let us introduce some notation for edges and angles around a vertex $x$ of $\Gs$ of degree $n$:
denote by $e_1=x x_1,\dotsc,e_n=x x_n$ edges incident to $x$; for every edge $e_j$, denote by
$\overline{\theta}_j$ its rhombus half-angle and by $e^{i\overline{\alpha}_j},e^{i\overline{\alpha}_{j+1}}$ 
its two rhombus vectors, see Figure~\ref{fig:notation}.

\begin{figure}[h]
\begin{center}
\resizebox{0.35\textwidth}{!}{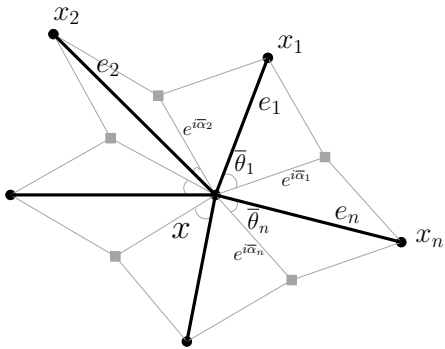}
\end{center}
\caption{Notation for edges and angles around a vertex $x$ of $\Gs$ of degree $n$.}
\label{fig:notation}
\end{figure}

\subsection{Definitions}
\label{sec:defLap}

\begin{defi}
Suppose that edges of the graph $\Gs$ are assigned positive \emph{conductances} $(\rho(e))_{e\in\Es}$ 
and that vertices are assigned (squared) \emph{masses} $(m^2(x))_{x\in\Vs}$. Then,
the \emph{massive Laplacian operator} $\Delta^{m}:\CC^\Vs\rightarrow\CC^\Vs$ 
is defined by:
\begin{align}
(\Delta^m\,f)(x)
&=\sum_{y\sim x} \rho(xy)[f(x)-f(y)]+m^2(x)f(x),\label{equ:operator_general} \\
&=d(x)f(x)-\sum_{y\sim x} \rho({xy})f(y),\nonumber
\end{align}
where 
$d(x)=m^2(x)+\sum_{y\sim x}
\rho(xy)$. The massive Laplacian operator is represented by an infinite matrix,
also denoted $\Delta^{m}$, whose
rows and columns are indexed by vertices of $\Gs$, and whose coefficients are given by:
\begin{equation*}
\forall\, x,y\in \Vs,\quad \Delta^{m}(x,y)=
\begin{cases}
  -\rho(xy)&\text{if $y\sim x$},\\
  d(x)&\text{if $y=x$},\\
  0 &\text{otherwise}.
\end{cases}
\end{equation*}
A function $f$ in $\CC^{\Vs}$ is \emph{massive harmonic} on $\Gs$ if $\Delta^{m} f =0$.
\end{defi}

We now introduce a one-parameter family of conductances and masses, indexed by the elliptic modulus $k\in[0,1)$.

\begin{defi}\label{def:massiveLap}
To every edge $e$ of $\Gs$, assign the \emph{conductance} $\rho(e)=\rke$, defined by:
\begin{equation}
\label{equ:def_conductances}
     \rke=\sc(\theta_e\vert k).
\end{equation}
To every vertex $x$ of degree $n$ of $\Gs$, assign the \emph{mass} $m^2(x)=m^2(x\vert k)$, defined by:
\begin{equation}
\label{eq:def_mass}
     m^2(x\vert k)=\sum_{j=1}^{n}[\Arm(\theta_j\vert k)-\rho(\theta_j\vert k)] \ \Leftrightarrow \ 
     d(x\vert k)=\sum_{j=1}^{n} \Arm(\theta_j\vert k),
\end{equation}
where $d(x\vert k)$ is the \emph{diagonal term} at the vertex $x$,
and $\Arm(\cdot\vert k)$ is given by Equation \eqref{def:Abis}. 

The main object studied in this paper is the 
corresponding \emph{$k$-massive Laplacian operator} denoted by $\Delta^{m(k)}$, defined by:
\begin{equation}
 (\Delta^{m(k)}\,f)(x)=
\sum_{j=1}^n [A(\theta_j\vert k)f(x)-\sc(\theta_j\vert k)f(x_j)].
\label{equ:operator}
\end{equation}

\end{defi}

\textbf{Notation.} From now on, to simplify notation, we only keep the dependence in $k$ in statements and omit it in
proofs, writing
$\Delta^m$, $\rho(\theta_e)=\sc(\theta_e)$, $m^2(x)$, $d(x)=\sum_{j=1}^n \Arm(\theta_j)$.

From Definition~\ref{def:massiveLap}, it is not clear that the mass is non-negative
and that the conductance and mass are analytic at $k=0$.
This is proved in the next two results.

\begin{prop}
\label{prop:neg_mass}
For every $k\in[0,1)$ and every vertex $x$ of $\Gs$, $m^2(x\vert k)\geq0$; it is equal to $0$ if and only if $k=0$.
\end{prop}

\begin{proof}
Returning to the definition of $m^2(x)$, see~\eqref{eq:def_mass}, it 
suffices to show that each term $\Arm(\theta_j)-\sc(\theta_j)$ is positive when $k>0$ and equal to 0 when $k=0$, for
$\theta_j\in(0,K)$.
Consider the function $f(u):=\Arm(u)-\sc(u)$ on $[0,K]$. Then $f(0)=0$ by definition~\eqref{def:Abis}. 
To prove that $f(K)=0$, we observe that as $u\to0$, using \eqref{id:scK} and \eqref{cor:Armbis:item2},
\begin{equation*}
     f(K-u) = \Arm(K-u)-\sc(K-u)=-A(u)+\frac{\dc(u)-\cn(u)}{k'\sn(u)}=O(u)+\frac{O(u^2)}{u+O(u^3)}\to 0.
\end{equation*}
Moreover, using formulas~\eqref{eq:derivatives_Jacobi_functions} and~\eqref{eq:deriv_Abis}, we have
\begin{equation*}
\frac{\text{d}^2f(u)}{\text{d}u^2}=\frac{\text{d}}{\text{d}u}\left(\frac{\dc^2(u)}{k'}-\frac{K-E}{k'K}-\dc(u)\nc(u)\right)=-\frac{\sn(u)}{\cn^3(u)}(k'-\dn(u))^2.
\end{equation*}
When $k>0$, the second derivative of $f$ is negative on $(0,K)$ implying that $f$ is strictly concave on $(0,K)$ and thus positive. When $k=0$, the first
derivative of $f$ is identically $0$ so that $f$ is constant and equal to $0$.
%
\end{proof}

\begin{lem}
\label{lem:analytic_weight}
For every edge $e$ and every vertex $x$ of $\Gs$, the conductance $\rho(\theta_e\vert k)$ and the mass 
$m^2(x\vert k)$ are analytic at $k=0$.
\end{lem}
\begin{proof}
We use the expansion \cite[16.23.9]{AS} of $\sc$ in terms of the nome $q=\exp(-\pi K'/K)$:
\begin{equation*}
     \rho(\theta_e\vert k)=\sc(\overline{\theta}_e\frac{2K(k)}{\pi}\vert k)=\frac{\pi}{2k'K(k)}\tan(\overline{\theta}_e)+
     \frac{2\pi}{k'K(k)}\sum_{n=1}^{\infty}(-1)^{n}\frac{q^{2n}}{1+q^{2n}}\sin(2n\overline{\theta}_e).
\end{equation*}
Since $1/k'$, $K$ and $q$ are analytic at $0$ \cite[17.3.11 and 17.3.21]{AS} for $K$ and $q$, respectively
we obtain the analyticity of the conductances.

The addition formula \eqref{cor:Armbis:item3} for $\Arm$ reduces the analyticity of the masses to those of the conductances, thereby concluding the proof.
\end{proof}

\paragraph{Example: $\Gs=\ZZ^2$.}
For every edge $e$, we have
$\overline{\theta}_e=\frac{\pi}{4}$, \emph{i.e.}, $\theta_e=\frac{K}{2}$, implying by \cite[2.4.10]{La89} that
$\rho(e)=\sc(\frac{K}{2})=\frac{1}{\sqrt{k'}}$. Moreover, using \eqref{eq:def_mass}
we have, for every vertex $x$,
\begin{equation*}
     m^2(x)=4\left(\Arm\Bigl(\frac{K}{2}\Bigr)-\sc\Bigl(\frac{K}{2}\Bigr)\right)=2\Bigl(1-\frac{1}{\sqrt{k'}}\Bigr)^2,
\end{equation*}
where to derive $\Arm(\frac{K}{2})$ we have used \eqref{cor:Armbis:item2} with $u=\frac{K}{2}$ and again \cite[2.4.10]{La89}.

In particular, the analyticity of the conductances and masses around $k=0$ ($k'=1$) proved in Lemma~\ref{lem:analytic_weight} is 
straightforward in this case.

\subsection{Massive harmonic functions and the star-triangle transformation}\label{sec:integrability}



Proposition~\ref{prop:3Dconsistency} below proves that the equation $\Delta^{m(k)} f=0$
satisfies \emph{$3$-dimensional consistency}~\cite{BobenkoSuris},
meaning that massive harmonic functions are compatible under star-triangle transformations of the underlying 
graph defined in Section~\ref{sec:NaturalOperations}.

Let us denote by $\Gsstar$ a finite or infinite isoradial graph containing a star, and by $\Gstriang$ the isoradial graph obtained from $\Gsstar$
by performing a star-triangle transformation. The vertex set
of $\Gsstar$ is the vertex set of $\Gstriang$ plus $x_0$, see Figure~\ref{fig:star_triangle}. 


\begin{prop}
  \label{prop:3Dconsistency}
\begin{itemize}
    \item Let $f$ be a function on $\Gsstar$. If
      $f$ is massive harmonic at $x_0$, then
      for all vertices $x$ of $\Gstriang$,
      $(\Deltastar^{m(k)} f)(x)=(\Deltatriang^{m(k)} f)(x)$.
    \item Conversely, let $f$ be a function on
      $\Gstriang$. Then there is
      a unique way of extending it to the vertex $x_0$ in such a way that
      $f$ is massive harmonic at $x_0$ and
      $(\Deltastar^{m(k)} f) (x) = (\Deltatriang^{m(k)} f)(x)$ for all vertices $x$ of
      $\Gstriang$.
  \end{itemize}
\end{prop}

\begin{figure}[ht]
\begin{center}
\resizebox{0.8\textwidth}{!}{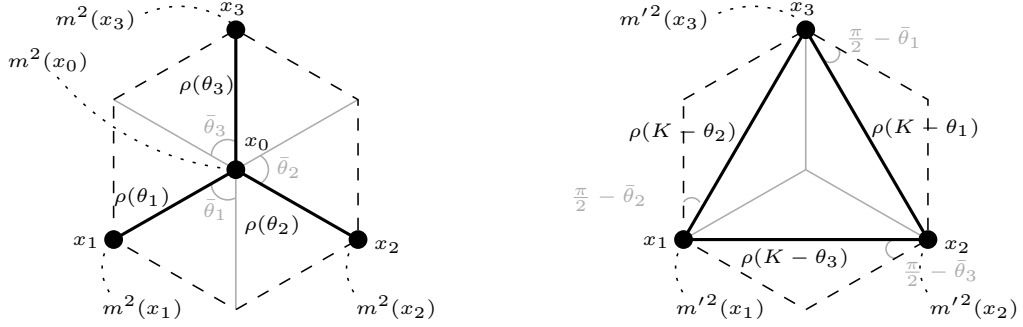}
\end{center}
\caption{\emph{Star-triangle transformation} and notation. If an isoradial graph $\Gsstar$
has a \emph{star} (left), \emph{i.e.}, a vertex $x_0$ of degree $3$, it can be transformed into a new
isoradial graph $\Gstriang$ having a \emph{triangle} (right) connecting the three neighbors $x_1,x_2,x_3$ of
$x_0$, by performing a cubic-flip on the underlying diamond graph $\GR$, and vice-versa.
\label{fig:star_triangle}}
\end{figure}

\begin{proof}
Refer to Figure~\ref{fig:star_triangle} for notation of vertices and weights of the star/triangle.
Consider a function $f$ on $\Gsstar$, and also denote by $f$ its restriction to $\Gstriang$.
Every vertex $x$ which is not one of $x_1,x_2,x_3,x_0$ has the same neighbors in $\Gsstar$ and $\Gstriang$, so that:
\begin{equation*}
  (\Deltastar^m f)(x)=(\Deltatriang^m f)(x).
\end{equation*}
Therefore, we only need to consider what happens at
vertices $x_1,x_2,x_3,x_0$. Suppose that we have proved the following:
\begin{equation}\label{equ:intermIntegr}
\forall\,i\in\{1,2,3\},\quad
\rho(\theta_j)\rho(\theta_k)(\Deltatriang^m f-\Deltastar^m f)(x_i)
=\Deltastar^m f(x_0),
\end{equation}
where $\{i,j,k\}=\{1,2,3\}$. Then, the first part of Proposition~\ref{prop:3Dconsistency} immediately follows.

For the second part,
consider a 
function $f$ on $\Gstriang$. Asking that its extension to $\Gsstar$ is massive harmonic
at $x_0$
requires that
\begin{equation*}
     \Deltastar^m f(x_0)=\bigl[m^2(x_0)+\sum_{\ell=1}^3\rho(\theta_\ell)\bigr]f(x_0)-\sum_{\ell=1}^3 \rho(\theta_\ell)f(x_\ell)=0,
\end{equation*}
which determines the value of $f$ at $x_0$. But then, by
Equation~\eqref{equ:intermIntegr}, the Laplacian on $\Gsstar$ of $f$ coincides
with the one of $f$ on $\Gstriang$ at the vertices
$x_1,x_2,x_3$, which concludes the proof of the second part.

We are thus left with proving Equation~\eqref{equ:intermIntegr}. Fix $i\in\{1,2,3\}$, and let $\mathscr{O}_i$ be the contribution to the massive Laplacian 
evaluated at $x_i$, coming from vertices outside of the 
triangle/star. It is common to both graphs, and returning to Expression~\eqref{equ:operator_general}, we have 
\begin{align*}
(\Deltastar^m f)(x_i)&=[m^2(x_i)+\rho(\theta_i)]f(x_i)-\rho(\theta_i)f(x_0)+\mathscr{O}_i,\\
(\Deltatriang^m f)(x_i)&=[{m'}^2(x_i)+\sum_{\ell\neq i}\rho(K-\theta_\ell)]f(x_i)-\rho(K-\theta_j)f(x_k)-\rho(K-\theta_k)f(x_j)+
\mathscr{O}_i.
\end{align*}
Using Equation~\eqref{equ:m1_bis} of Appendix~\ref{app:elliptic},
\begin{equation*}
{m'}^2(x_i)-m^2(x_i)=\rho(\theta_i)-\sum_{\ell\neq i}\rho(K-\theta_\ell)-k'\rho(K-\theta_j)\rho(K-\theta_k)\rho(\theta_i),
\end{equation*}
and taking the difference yields that $(\Deltatriang^m f)(x_i)-\Deltastar^m f$ is equal to
\begin{align*}
-\rho(K-\theta_j)f(x_k)-\rho(K-\theta_k)f(x_j)-k'\rho(K-\theta_j)\rho(K-\theta_k)\rho(\theta_i)f(x_i)+
\rho(\theta_i)f(x_0).
\end{align*}
Multiplying this equation by $k'\rho(\theta_j)\rho(\theta_k)$, using
the fact that $k'\rho(K-\theta_\ell)\rho(\theta_\ell)=1$ (see Identity~\eqref{id:scK}), and $k'\prod_{\ell=1}^3\rho(\theta_\ell)=m^2(x_0)+\sum_{\ell=1}^3\rho(\theta_\ell)$
(see Equation~\eqref{equ:m0}), we conclude:
\begin{equation*}
k'\rho(\theta_j)\rho(\theta_k)(\Deltatriang^m f-\Deltastar^m f)(x_i)=
[m^2(x_0)+\sum_{\ell=1}^3\rho(\theta_\ell)]f(x_0)-\sum_{\ell=1}^3 \rho(\theta_\ell)f(x_\ell)=
(\Deltastar f)(x_0).\qedhere
\end{equation*}
\end{proof}
When extending $f$ from $\Gstriang$ to $\Gsstar$, we have four equations which
could individually determine the value of $f(x_0)$: the massive harmonicity
condition at $x_0$, and the three equations from~\eqref{equ:intermIntegr}.
The remarkable fact, proved in Proposition~\ref{prop:3Dconsistency}, is that all these conditions give the same result; this is also
known as \emph{$3$-dimensional consistency} of the equation $\Delta^{m(k)} f=0$, because of the
geometric interpretation of the star-triangle transformation on quasicrystalline
isoradial graphs seen as monotone surfaces in $\ZZ^\ell$ \cite{BobenkoSuris}.
This condition is
then sufficient to ensure $\ell$-dimensional consistency, in the following sense:
let $(\Gs_{n})_{n}$ be a sequence of isoradial graphs where two successive
graphs differ by a star-triangle transformation, representing
a discrete sequence of monotone surfaces in $\mathbb{Z}^\ell$. Then, by Proposition~\ref{prop:3Dconsistency},
from a massive harmonic function $f_0$ on $\Gs_0$ one
can construct, in a consistent way, a harmonic function $f_n$ on
$\Gs_n$, for every $n$. In particular, if the sequence
$(\Gs_n)_n$ spans the whole $\ell$-dimensional lattice $\ZZ^\ell$ (namely, for
every vertex of $\mathbb{Z}^\ell$, there exists an $n$ such that this vertex is in
the monotone surface $\Gs_n$), then a massive harmonic function
on $\Gs_0$ can uniquely be extended to $\ZZ^\ell$, and its restriction to any
monotone surface, viewed as an isoradial graph, is again massive harmonic.

This property is in the spirit of integrable equations on quad-graphs discussed
in~\cite{AdlerBobenkoSuris,BobenkoSuris}. Our massive Laplacian satisfies a so-called
\emph{three-leg equation}, using the terminology of~\cite{AdlerBobenkoSuris},
as shown in the forthcoming Equation~\eqref{eq:threeleg}, but it does not fit in
their classification of \emph{three-leg integrable equations},
because it does not satisfy their symmetry requirement and
does not allow to define values on $\Gs^*$.

\subsection{The discrete $k$-massive exponential function}\label{sec:expofunction}

In this section we introduce the \emph{discrete $k$-massive
exponential function}. In Proposition~\ref{prop:exp_harmo}, we prove that it
defines a family of massive harmonic functions.  
This is one of the key facts needed to prove the local formula for the massive Green function of Theorem~\ref{thm:expression_Green}. 

\subsubsection{Definition}

\begin{defi}\label{def:massiveexpo}
The \emph{discrete $k$-massive exponential function} or simply \emph{massive exponential function}, denoted $\expo_{(\cdot,\cdot)}(\cdot\vert k)$, 
is a function from $\VR\times\VR\times\CC$ to $\CC$. Consider a pair of vertices $x,y$ of $\GR$, and an edge-path
$x=x_1,\dotsc,x_n=y$ of the diamond-graph $\GR$ from $x$ to $y$; let $e^{i\overline{\alpha}_j}$ be the vector corresponding to the edge
$x_jx_{j+1}$, see Figure~\ref{fig:Isoradial18}. Then $\expo_{(x,y)}(\cdot\vert k)$ is defined inductively 
along the edges of the path as follows. For every $u\in\CC$,
\begin{align}\label{eq:recursive_def_expo}
\expo_{(x_j,x_{j+1})}(u\vert k) = i \sqrt{k'}\,\sc(u_{\alpha_j}\vert k),\quad
\expo_{(x,y)}(u\vert k)         = \prod_{j=1}^{n-1} \expo_{(x_j,x_{j+1})}(u\vert k),
\end{align}
where $u_\alpha=\frac{u-\alpha}{2}$, and recall that $\alpha_j=\overline{\alpha}_j\frac{2K}{\pi}$.
\end{defi}

Note that when $k=0$, one recovers the discrete exponential function
of~\cite{Mercat:exp}, see also~\cite{Kenyon3} after the change of variable $z=e^{iu}$.

\begin{figure}[h]
  \centering
  \resizebox{0.9\textwidth}{!}{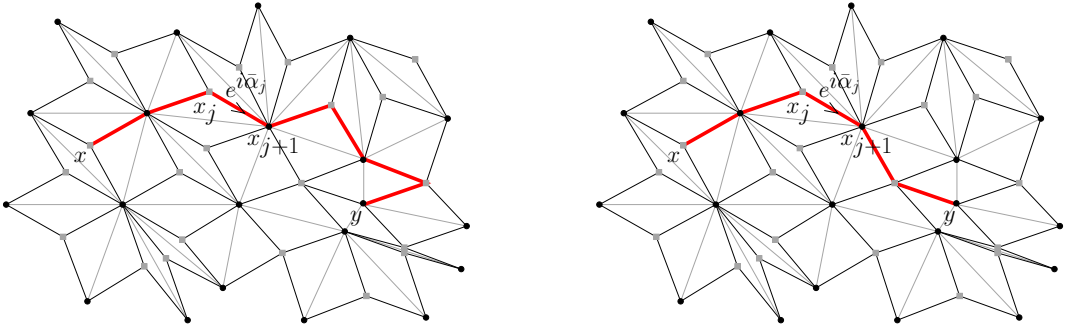}
  \caption{Examples of paths of $\GR$ from $x$ to $y$ used to compute the discrete massive
    exponential function $\expo_{(x,y)}(\cdot\vert k)$. The path on the right is minimal, whereas the one on the left is not.}
    \label{fig:Isoradial18}
\end{figure}

\begin{lem}\label{lem:welldefined}
  The discrete $k$-massive exponential function is well defined, that is, for every pair of vertices $x,y$ 
  of $\GR$, and for every $u\in\CC$, $\expo_{(x,y)}(u\vert k)$ is independent of the choice of edge-path
  from $x$ to $y$.
\end{lem}

\begin{proof}
  If $(x,y)$ is an edge of $\GR$ corresponding to a vector $e^{i\overline{\alpha}}$, then
  the edge $(y,x)$ corresponds to the vector $e^{i\overline{\alpha}+\pi}=e^{i\overline{\alpha+2K}}$. Observing that
  $u_{\alpha+2K}=u_\alpha-K$, we deduce by \eqref{id:scK}: 
    \begin{equation*}
    \expo_{(y,x)}(u\vert k) = i \sqrt{k'} \sc(u_{\alpha+2K}\vert k) =
    i\sqrt{k'}\times\frac{-1}{k'\sc(u_\alpha\vert k)}=\expo_{(x,y)}(u\vert k)^{-1}.
  \end{equation*}  
  This implies that the product of the local factors around any rhombus is equal to $1$. 
  Indeed, the contribution of a side of the rhombus
  comes with its inverse, which is the contribution of the opposite side.
  Therefore, the product of
  every closed path in $\GR$ is equal to $1$.
\end{proof}

\begin{rem}
\label{rem:encode}
A consequence of Definition~\ref{def:massiveexpo} and Lemma~\ref{lem:welldefined} is that the zeros (resp.\ poles) of $\expo_{(x,y)}(\cdot\vert k)$ are encoded by the steps of a minimal path from $x$ to $y$. Specifically, if the steps of a minimal path from $x$ to $y$ are $\{e^{i\overline{\alpha}_\ell}\}_\ell$, then the zeros (resp.\ poles) are $\{\alpha_\ell\}_\ell$ and $\{\alpha_\ell+4iK'\}_\ell$ (resp.\ $\{\alpha_\ell+2K\}_\ell$ and $\{\alpha_\ell+2K+4iK'\}_\ell$).
\end{rem}

The construction of a discrete massive harmonic function
from a starting
point, by successive multiplication by local factors along any path is called
a \emph{discrete zero curvature representation} of the solutions of the
equation~$\Delta^{m(k)}f=0$, see~\cite[Chapter~6]{BobenkoSuris} for analogous
constructions. This property, together with $3$-dimensional consistency proved in Proposition~\ref{prop:3Dconsistency},
means that the massive Laplacian $\Delta^{m(k)}$ is \emph{discrete integrable}.

\subsubsection{Restriction of the domain of definition}

Recall from Section~\ref{sec:EllipticFunctions} that the elliptic function $\sc(\cdot\vert k)$ is doubly-periodic with period $2K$ and $4iK'$.
Therefore the parameter $u$ of the massive exponential function
$\expo_{(x,y)}(u\vert k)$ defined in \eqref{eq:recursive_def_expo} can be seen as living on the torus $\CC / (4K\ZZ+8iK'\ZZ)$.
However, on this torus, the function $u\mapsto\sc(u_\alpha\vert k)$ satisfies
$\sc((u+4iK')_\alpha\vert k)=\sc(u_\alpha+2iK'\vert k)=-\sc(u_\alpha\vert k)$.

If both vertices $x$ and $y$ belong to $\Gs$,
the number of $\sc$ factors in the definition of
$\expo_{(x,y)}(u\vert k)$ is even, implying that  
$\expo_{(x,y)}(\cdot\vert k)$ is an elliptic function with period $4K$ and $4iK'$. 
In the following, when working with the massive exponential function restricted to pairs of vertices of $\Gs$,
we suppose that the parameter $u$
belongs to the torus $\TT(k):= \CC / (4K\ZZ+4iK'\ZZ)$.

\subsubsection{Massive exponential functions are massive harmonic functions}

The next proposition proves the key property of the discrete massive exponential function, \emph{i.e.}, that it defines a family of massive harmonic functions. 

\begin{prop}
  \label{prop:exp_harmo}
  For every $u\in \TT(k)$, the massive exponential function $\expo_{(x,y)}(u\vert k)$ is
  massive harmonic on $\Gs$ in each variable $x$ and $y$. Namely,
  \begin{equation*}
    \forall\,x\in\Vs,\quad
    \Delta^{m(k)} \expo_{(\cdot,x)}(u\vert k) = 
    \Delta^{m(k)} \expo_{(x,\cdot)}(u\vert k) =
    0.
  \end{equation*}
\end{prop}

\begin{proof}
Let $y$ be a vertex of $\Gs$.  Since $\Delta^m$ is symmetric and $\expo_{(x,y)}(u+2K)=\expo_{(y,x)}(u)$, it
is enough to prove that for every vertex $x$ of $\Gs$, 
$(\Delta^m \expo_{(\cdot,y)}(u))(x)=0$. Suppose that $x$ has degree $n$ and denote by $x_1,\dotsc,x_n$ the vertices incident to $x$, by
$e_1,\dotsc,e_n$ and $\theta_1,\dotsc,\theta_n$, the corresponding edges and rhombus angles, see Figure~\ref{fig:notation}. By definition of the massive
exponential function, we have $\expo_{(x_j,y)}(u)=\expo_{(x_j,x)}(u)\expo_{(x,y)}(u)$. As a consequence, using~\eqref{equ:operator},
\begin{align*}
(\Delta^m\,\expo_{(\cdot,y)}(u))(x)
=\Biggl(\sum_{j=1}^n [\Arm(\theta_j)-\sc(\theta_j)\expo_{(x_j,x)}(u)]\Biggr)\expo_{(x,y)}(u).
\end{align*}
It thus suffices to prove that the prefactor
\begin{equation}\label{equ:sumzero}
\sum_{j=1}^n [\Arm(\theta_j)-\sc(\theta_j)\expo_{(x_j,x)}(u)]=0.
\end{equation}
Replacing the exponential function by its definition, and referring to Figure~\ref{fig:notation} for the notation of the rhombus vectors, we have
for every $j$
\begin{align}
\Arm(\theta_j)-\sc(\theta_j)\expo_{(x_j,x)}(u)&=
\Arm\Bigl(\frac{\alpha_{j+1}-\alpha_{j}}{2}\Bigr)+
k'\sc(\theta_j)\sc(u_{\alpha_j+2K})\sc(u_{\alpha_{j+1}+2K})\nonumber\\
&=\Arm(u_{\alpha_j+2K})-\Arm(u_{\alpha_{j+1}+2K}),
\text{by \eqref{cor:Armbis:item3} of Lemma~\ref{cor:Armbis}}.
\label{eq:threeleg}
\end{align}
By Section~\ref{sec:diamondetc}, the angles $\alpha_j$, $\alpha_{j+1}$ are such that
$\frac{\alpha_{j+1}-\alpha_{j}}{2}=\theta_j$. This implies that
\[
    (\alpha_{n+1}+2K)-(\alpha_{1}+2K)
      = \sum_{j=1}^{n}\alpha_{j+1}-\alpha_j
      = 2\sum_{j=1}^{n}\theta_j 
      = 4 K.
\]
We thus have, $u_{\alpha_{1}+2K}=u_{\alpha_{n+1}+2K}+2K$. Summing over $j$ we obtain:
\begin{align*}
\sum\limits_{j=1}^n [\Arm(\theta_j)-\sc(\theta_j)\expo_{(x_j,x)}(u)]&=
\sum\limits_{j=1}^n[\Arm(u_{\alpha_j+2K})-\Arm(u_{\alpha_{j+1}+2K})]\\
&=\Arm(u_{\alpha_{n+1}+2K}+2K)-\Arm(u_{\alpha_{n+1}+2K})=0,
\end{align*}
where in the last equality we have used Equation~\eqref{cor:Armbis:item1} of Lemma~\ref{cor:Armbis} in
Appendix~\ref{app:ellipticAH}.
\end{proof}

\section{Massive Green function on isoradial graphs}\label{sec:massiveGreen}

In the whole of this section, we let $\Gs$ be an infinite isoradial graph, and
fix an elliptic modulus $k\in(0,1)$.
We consider the inverse $G^{m(k)}$ of the massive Laplacian operator
$\Delta^{m(k)}$, that is the \emph{massive Green function}, whose definition we recall
in Section~\ref{def:green}. In Theorem~\ref{thm:expression_Green} of
Section~\ref{sec:localformula}, we prove an \emph{explicit local formula} for the 
massive Green function.
Then, in Theorem~\ref{thm:asymp_Green} of
Section~\ref{sec:asymp_Green}, using a saddle-point analysis, we prove explicit asymptotic exponential decay of the Green function.
%

\subsection{Definition}\label{def:green}


The space of functions on $\Vs$ with finite support is endowed with a natural
scalar product: $\langle f, g\rangle = \sum_{x\in\Vs}\overline{f(x)} g(x),$
which can be completed into the Hilbert space $L^2(\Vs)$.

The operator $(\Delta^{m(k)})$, defines a symmetric
bilinear form on $L^2(\Vs)$, called the \emph{energy form} or \emph{Dirichlet
form} $\mathcal{E}(\cdot,\cdot\vert k)$:
\begin{equation*}
  \mathcal{E}(f,g\vert k)= \frac{1}{2} \langle f,(\Delta^{m(k)}) g\rangle =
  \frac{1}{2}\sum_{x\in\Vs} m^2(x\vert k) \overline{f(x)} g(x) +
  \sum_{y \sim x} \rho(\theta_{xy}\vert k)\overline{(f(x)-f(y))} (g(x)-g(y)).
\end{equation*}
Note that the condition imposed on rhombus half-angles, namely that they are 
in $(\eps,\frac{\pi}{2}-\eps)$ for some $\eps>0$, implies
that the degree of vertices is uniformly bounded, and that conductances $(\rho(\theta_{e}))$ are uniformly bounded away from 0 and 
infinity. Moreover, since $k>0$, masses $(m^2(x))$ are also uniformly bounded. Therefore, there exist two constants
$c, C>0$ such that for all $f\in L^2(\Vs)$,
\begin{equation*}
  c\langle f, f \rangle \leq 
  \langle f, (\Delta^{m(k)}) f\rangle \leq 
  C \langle f, f\rangle.
\end{equation*}


As a consequence, the inverse of $\Delta^{m(k)}$, called the \emph{massive Green function} and
denoted by $G^{m(k)}$, is
well defined, and can be expressed from the semigroup
$(e^{-t\Delta^{m(k)}})_{t\geq 0}$ as:
\begin{equation*}
  G^{m(k)}=\int_{0}^{\infty}e^{-t\Delta^{m(k)}}\ud t.
\end{equation*}

For every $f\in L^2(\Vs)$, $G^{m(k)}f\in L^2(\Vs)$, and $G^{m(k)}$ is uniquely
characterized by the fact that for any functions $f,g$ in
$L^2(\Vs)$, $\mathcal{E}(G^{m(k)}f,g\vert k)=\langle f, g\rangle.$

Like the massive Laplacian, the massive Green function can be seen as an infinite
symmetric matrix with rows and columns indexed by vertices of $\Gs$ as follows:
\begin{equation*}
  \forall\, x,y\in\Vs,\quad
  G^{m(k)}(x,y) = (G^{m(k)}\delta_y)(x).
\end{equation*}

Note that for any vertex $y$ of $\Gs$, $x\mapsto G^{m(k)}(x,y)$ belongs to
$L^2(\Vs)$. In particular,
\begin{equation*}
  \lim_{x\to\infty}G^{m(k)}(x,y)=0.
\end{equation*}


\subsection{Local formula for the massive Green function}\label{sec:localformula}

Theorem~\ref{thm:expression_Green} proves 
an explicit formula for the massive Green function
$G^{m(k)}$. Notable features of this theorem are explained in the introduction and briefly recalled in Remark~\ref{rem:critical_Green}.

\begin{thm}
\label{thm:expression_Green}
Let $\Gs$ be an infinite isoradial graph.
Then, for every pair of vertices $x$, $y$ of $\Gs$, the massive Green function
$G^{m(k)}(x,y)$ has the following explicit expression:
  \begin{equation}
  \label{eq:green}
  G^{m(k)}(x,y) =\frac{k'}{4i\pi} \int_{\Cs_{x,y}} \expo_{(x,y)}(u\vert k) \ud u,
  \end{equation}
  where the contour of integration $\Cs_{x,y}$ is the vertical closed
  path $\varphi_{x,y}+[0,4iK'(k)]$ on
  $\TT(k)$, winding once vertically and directed upwards, and
  $\overline{\varphi}_{x,y}=\frac{\pi}{2K}\varphi_{x,y}$ is the angle of the ray $\RR \overrightarrow{xy}$, see Figure~\ref{The_rectangle}.
  
  Alternatively, the massive Green function $G^{m(k)}(x,y)$ can be expressed as
\begin{equation}
\label{eq:greenbis}
    G^{m(k)}(x,y)=\frac{k'}{4i\pi} \oint_{\gamma_{x,y}} H(u|k) \expo_{(x,y)}(u\vert k) \ud u,
\end{equation}
where the function $H$ is defined in Equation \eqref{def:H}, $\gamma_{x,y}$ is a trivial contour on the torus, not crossing $\Cs_{x,y}$
and containing in its interior all the poles of $\expo_{(x,y)}(\cdot\vert k)$
and the pole of $H(\cdot\vert k)$, see Figure~\ref{The_rectangle}. 
\end{thm}


\unitlength=0.6cm
\begin{figure}[h]
\vspace{40mm}
\begin{center}
\begin{tabular}{cccc}
\begin{picture}(0,0)(0,0)
\put(-2,0){\line(1,0){4}}
\put(2,0){\line(1,0){4}}
\put(-6,0){\line(1,0){4}}
\put(-6,3){\line(1,0){12}}
\put(-6,0){\line(0,1){3}}
\put(6,0){\line(0,1){3}}
\put(-2,6){\line(1,0){4}}
\put(-3.6,0){\vector(0,1){1.5}}
\put(-3.6,1.5){\vector(0,1){1.5}}
\put(-3.6,3){\vector(0,1){1.5}}
\put(-3.6,4.5){\vector(0,1){1.5}}
\put(2,6){\line(1,0){4}}
\put(-6,6){\line(1,0){4}}
\put(-6,3){\line(0,1){3}}
\put(6,3){\line(0,1){3}}
\put(-6.15,-0.15){{$\bullet$}}
\put(-5.55,-0.15){{$\bullet$}}
\put(-5.55,5.85){{$\bullet$}}
\put(-6.15,2.85){{$\bullet$}}
\put(-0.2,-0.2){{$\blacksquare$}}
\put(-0.2,5.8){{$\blacksquare$}}
\put(-6.15,5.85){{$\bullet$}}
\put(5.85,-0.15){{$\bullet$}}
\put(5.85,5.85){{$\bullet$}}
\put(-7.7,2.9){{$-2K$}}
\put(-9.6,-0.8){{$-2K-2iK'$}}
\put(-5.6,-0.8){{$u_0$}}
\put(-5.6,6.4){{$u_0$}}
\put(-3.75,5.85){{$\bullet$}}
\put(-3.75,-0.15){{$\bullet$}}
\put(-4.0,-0.8){{$\varphi_{x,y}$}}
\put(-4.0,6.4){{$\varphi_{x,y}$}}
\put(-9.6,6.4){{$-2K+2iK'$}}
\put(6.1,6.4){{$2K+2iK'$}}
\put(6.1,-0.8){{$2K-2iK'$}}
\put(-1,-0.8){{$-2iK'$}}
\put(-5.23,4.5){\vector(-1,3){0}}
\put(-5.23,1.5){\vector(1,3){0}}
\put(-3.5,0.4){$\Cs_{x,y}$}
\put(-1.9,0.4){$\gamma_{x,y}$}
\put(-0.15,2.87){{\tiny$\square$}}
\put(-0.65,2.87){{\tiny$\square$}}
\put(-1.35,2.87){{\tiny$\square$}}
\put(0.15,2.87){{\tiny$\square$}}
\put(0.55,2.87){{\tiny$\square$}}
\put(1.05,2.87){{\tiny$\square$}}
\put(-0,4){\vector(-1,0){0}}
\qbezier(-2,3)(-2,4)(0,4)
\qbezier(0,4)(2,4)(2,3)
\put(-2,3){\vector(0,-1){1.5}}
\put(-2,1.5){\vector(0,-1){1.5}}
\put(2,0){\vector(0,1){1.5}}
\put(2,1.5){\vector(0,1){1.5}}
\qbezier(-2,6)(-2,5)(0,5)
\qbezier(2,6)(2,5)(0,5)
\qbezier(-4.8,3.12)(-5.4,4.5)(-5.4,6)
\qbezier(-4.8,2.88)(-5.4,1.5)(-5.4,0)
\put(-5.25,0.4){$\Cs'_{x,y}$}
\put(0,5){\vector(1,0){0}}
\put(5.85,2.85){$\circ$}
\put(5.35,2.85){{$\circ$}}
\put(4.65,2.85){{$\circ$}}
\put(-5.85,2.85){{$\circ$}}
\put(-5.45,2.85){{$\circ$}}
\put(-4.95,2.85){{$\circ$}}
\put(-2,6.5){\vector(1,0){4}}
\put(2,6.5){\vector(-1,0){4}}
\put(-1.8,6.8){\text{length $<2K$}}
\end{picture}
\end{tabular}
\end{center}
\vspace{5mm}
\caption{The fundamental rectangle $[-2K,2K]+i[-2K',2K']$ with a representation of the integration contours 
$\Cs_{x,y}$ and $\gamma_{x,y}$ of the Green function in~\eqref{eq:green} and~\eqref{eq:greenbis}, the poles of the 
exponential function $\expo_{(x,y)}(\cdot\vert k)$ (white squares), the zeros of $\expo_{(x,y)}(\cdot\vert k)$ (white bullets), 
the pole of the function $H(\cdot\vert k)$ in~\eqref{eq:greenbis} (black square), the saddle point $u_0$ and the steepest descent contour 
$\Cs'_{x,y}$ used in the proof of Theorem~\ref{thm:asymp_Green}.}
\label{The_rectangle}
\end{figure}
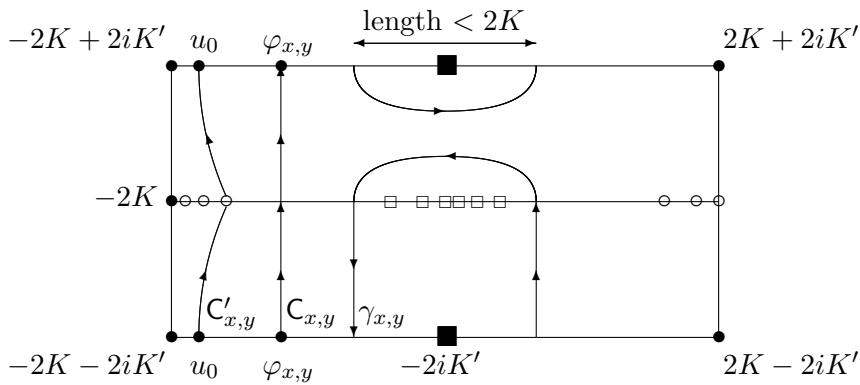

\begin{rem} $\,$ \label{rem:critical_Green}
\begin{itemize}
\item Formula \eqref{eq:green} has the remarkable feature of being \emph{local}, meaning that the 
Green function $G^{m(k)}(x,y)$ is computed using geometric information of a path from $x$ to $y$ only. This feature is inherited
from the massive exponential function, see Definition~\ref{def:massiveexpo}. Note also, that there is no periodicity 
assumption on the graph $\Gs$, and that explicit computations can be performed using the residue theorem, see Formulas \eqref{equ:explicit1},
\eqref{equ:explicit2} and \eqref{equ:explicit3}. More details and a description of the context, in particular the papers 
\cite{Kenyon3} and \cite{BoutillierdeTiliere:iso_gen}, can be found in the introduction.
\item
In the limiting case $k\to0$, the torus becomes an infinite cylinder ($K\to\pi/2$
and $K'\to\infty$), the contour of integration $\gamma_{x,y}$ becomes an infinite
(vertical) straight line, and one has $H(u)\to u/(2\pi)$ thanks to
Lemma~\ref{lem:h} of Appendix~\ref{app:ellipticAH}. In this way, we formally 
obtain the following expression for the massless Green function:
\begin{equation*}
     G^{m(0)}(x,y) = \frac{1}{8i\pi^2}\int_{\gamma_{x,y}} u \expo_{(x,y)}(u\vert 0) \ud u.
\end{equation*}
This expression, after the change of variable $z=-e^{iu}$, is exactly the one
given by Kenyon in~\cite[Theorem 7.1]{Kenyon3}.
Strictly speaking,
the limit of~\eqref{eq:green} when $k$ goes to $0$ is infinite, which can be
expected, since when $k=0$, the mass vanishes and the corresponding random walk
is recurrent.
However, if the diagonal is subtracted, one can take the limit, make sense
of the change of variable, and recover Kenyon's expression.

\item Note that
we can add to $H$ in \eqref{eq:greenbis} any elliptic function $f$ on $\TT(k)$ without changing the
result. Indeed, the sum of residues of $f\expo_{(x,y)}$ on $\TT(k)$ is equal to zero.
\end{itemize}
\end{rem}

\begin{proof}
Let us first prove the equality between expressions \eqref{eq:green} and \eqref{eq:greenbis}.
The function $H$ is multivalued because of a horizontal period.
By Lemma~\ref{lem:h} of Appendix~\ref{app:ellipticAH}, a determination of $H$ on $\TT(k)\setminus
\Cs_{x,y}$ is meromorphic on this domain, it has a single pole at $2iK'$,
and the jump across
$\Cs_{x,y}$ is constant and equal to $1$.
We start from expression \eqref{eq:greenbis}. On the torus
$\TT(k)$ deprived of $\Cs_{x,y}$ and from the poles of $H\expo_{(x,y)}$,
$\gamma_{x,y}$ is homologically equivalent to two vertical contours, one on each side
of $\Cs_{x,y}$, with different orientations. The sum of the integrals of
$H\expo_{(x,y)}$ on these two vertical contours is equal to the integral along
$\Cs_{x,y}$ of the jump of $H\expo_{(x,y)}$ across $\Cs_{x,y}$, which is
equal to $\expo_{(x,y)}$. We thus obtain expression \eqref{eq:green}.

The vertex $y$ is considered fixed. Denote by $f(x)$ the common value of the
right hand side of \eqref{eq:green} and \eqref{eq:greenbis}.
Using the idea of the argument of \cite{Kenyon3}, we now prove that $f(x)$ is the
Green function $G^m(x,y)$. 
We first show that $(\Delta^m f)(x)=\delta_y(x)$. The argument is separated
into two cases.

\paragraph{Case $x\neq y$.} Denote by $e^{i\overline{\alpha}_1},\dotsc,e^{i\overline{\alpha}_n}$ 
the unit vectors coding the edges of $\GR$ around $x$, and by $x_1,\dotsc,x_n$ the
neighbors of $x$ in $\Gs$ listed counterclockwise, such that
$x_j=x+e^{i\overline{\alpha}_j}+e^{i\overline{\alpha}_{j+1}}$. For definiteness, we chose
$e^{i\overline{\alpha}_1}$ to be the first vector when going counterclockwise around
$x$, starting from the segment $[y,x]$, and we have
$\overline{\alpha}_{j+1}=\overline{\alpha}_j+2\overline{\theta}_j$, where
$\overline{\theta}_j\in(\eps,\pi/2-\eps)$ is the
the rhombus half-angle of the edge $xx_j$.
      
The poles of the function $\expo_{(x,y)}(u)$ are encoded by the steps of a
minimal path from $x$ to $y$. If the steps of the minimal path are
$\{e^{i\overline{\alpha}_\ell}\}$, the poles are $\{\alpha_\ell+2K\}$, see Remark~\ref{rem:encode} and Figure~\ref{The_rectangle}.
According to~\cite[Lemma~17]{BoutillierdeTiliere:iso_gen}, the steps are
contained in a sector of angle not larger than $\pi$, avoiding the half-line
$\mathbb{R}^+ \overrightarrow{yx}$. As a consequence, the poles $\{\alpha_\ell+2K\}$ can be chosen
in an interval of length not larger than $2K$ and not touching the contour
$\Cs_{x,y}$ used for the integration in \eqref{eq:green} (see
again Figure~\ref{The_rectangle}). 
The contour $\Cs_{x,y}$
can be moved to the
left or to the right as long as it does not cross any of these poles.
      
In the function $\expo_{(x_j,y)}(u)=\expo_{(x_j,x)}(u)\expo_{(x,y)}(u)$, we have (at most) a subset of the poles of
$\expo_{(x,y)}(u)$ (since one of
the poles of $\expo_{(x,y)}(u)$ can be canceled by a zero of $\expo_{(x_j,x)}(u)$), plus those associated to 
$\expo_{(x_j,x)}(u)$, which are $\alpha_j,\alpha_{j+1}$.
The whole set of poles $\{\alpha_\ell+2K,\alpha_j\}$ avoids a sector
to which we can move all the contours $\Cs_{x,y}$ and
$\Cs_{x_1,y},\dotsc,\Cs_{x_n,y}$, without crossing any pole, and thus use
the same contour of integration $\Cs$ for $f(x)$ and $f(x_j)$.
By linearity of the integral, we thus have
\begin{equation*}
  (\Delta^m f)(x)
  = \left(\frac{k'}{4i\pi}\Delta^m \int_{\Cs}\exp_{(\cdot,y)}(u)\ud
  u\right)(x) = \frac{k'}{4i\pi}\int_{\Cs}[\Delta^m
  \exp_{(\cdot,y)}(u)](x)\ud u.
\end{equation*}
By Proposition~\ref{prop:exp_harmo}, the term in square brackets on the
right-hand side is zero, and we conclude that $f$
is massive harmonic outside of $y$.
\paragraph{Case $x=y$.}
By definition of the massive Laplacian, we have
\begin{equation*}
  (\Delta^m f)(x)
  =d(x) f(x)-\sum_{j=1}^{n} \rho(\theta_j) f(x_j).
\end{equation*}
The values of $f$ at $x$ and its neighbors $x_j$ are obtained by a direct
computation of the integral defining $f$ with the residue theorem,
explicited in Lemma~\ref{lem:Gneighbor} of Appendix~\ref{app:comput_green}:
\begin{equation*}
  f(x_j) = H(\alpha_{j})-H(\alpha_{j+1}) +
  \frac{k' K'}{\pi}\expo_{(x_j,x)}(2iK'),
  \qquad 
  f(x) = \frac{k'K'}{\pi},
\end{equation*}
with the convention that $\alpha_{n+1} = \alpha_{1}+4K$.
By Equation~\eqref{equ:sumzero},
\begin{equation*}\frac{k' K'}{\pi}\Biggl(d(x)-\sum_{j=1}^{n}
  \rho(\theta_j)\expo_{(x_j,x)}(2iK')\Biggr)=0,
\end{equation*}
so that the remaining terms are
\begin{equation*}
(\Delta^m f)(x)=
\sum_{j=1}^n H(\alpha_{j+1})-H(\alpha_{j}) = 
H(\alpha_1 + 4K)-H(\alpha_1)=1,
\end{equation*}
where in the last equality, we used the first point of Lemma~\ref{lem:h}.

In the forthcoming Proposition~\ref{thm:exp_bounded_Green}, we prove that $f(x)$ decays
(exponentially fast) to zero. Since $G^m(x,y)$ also goes to zero when $x$ goes
to infinity, and has the same massive Laplacian as $f$, the difference
$G^{m}(\cdot,y)-f$ tends to zero at infinity and is harmonic: by the maximum
principle,
$f$ has to be equal to $G^{m}(\cdot,y)$.
\end{proof}

\paragraph{Examples.}
Formula~\eqref{eq:greenbis} of Theorem~\ref{thm:expression_Green} allows for
explicit computations using the residue theorem.
We now list a few special values. Details are given in Lemma~\ref{lem:Gneighbor} of
Appendix~\ref{app:comput_green}. 
\begin{itemize}
 \item For every vertex $x$ of $\Gs$,
\begin{equation}\label{equ:explicit1}
  G^{m(k)}(x,x) = \frac{k'K'}{\pi}.
\end{equation}
Note that this value does not depend on $x$, and is a function of $k$ only.
\item Let $x,y$ be two adjacent vertices of $\Gs$, and let $\theta$ be the rhombus half-angle of the edge $xy$, then
\begin{equation}\label{equ:explicit2}
  G^{m(k)}(x,y)=\frac{K'\dn(\theta)}{\pi}-\frac{H(2\theta)}{\sc(\theta)}.
\end{equation}
\item In the limit $k\to 0$, 
\begin{equation}\label{equ:explicit3}
  \lim_{k\to 0} \,(G^{m(k)}(x,x)-G^{m(k)}(x,y)) = \frac{\theta}{\pi\tan\theta},
\end{equation}
which is the value obtained by Kenyon~\cite[Section 7.2]{Kenyon3} in the critical case.
\end{itemize}

%
%

\subsection{Asymptotics of the Green function}\label{sec:asymp_Green}

In this section, we suppose that the graph $\Gs$ is quasicrystalline and compute asymptotics of the Green function
$G^{m(k)}(x,y)$ when the graph distance in $\GR$ between $x$ and $y$ is large.

Under the quasicrystalline assumption, the number of directions $\pm e^{i\overline{\alpha}}$
assigned to edges of the diamond graph $\GR$ is finite, and $\GR$ can be seen as the projection of a monotone surface in $\ZZ^\ell$, see Section~\ref{def:quasi}.
The distance between two vertices $x$ and $y$ of
$\Gs$, measured as the length of a minimal path, is thus the
graph distance between $x$ and $y$ seen as vertices of $\GR$. It is also the graph distance in $\ZZ^\ell$ between the
corresponding points on the monotone surface, and we denote it by $\vert x-y\vert $,
where $x-y\in\ZZ^\ell$ is the vector between the points on the surface.



In order to state Theorem~\ref{thm:asymp_Green}, we need the following notation.
By \cite[Lemma 17]{BoutillierdeTiliere:iso_gen},
the set $\{\alpha_1,\dotsc, \alpha_p\}$ of zeros of $\expo_{(x,y)}(u)$ is contained in an
interval of length $2K-2\eps$, for some $\eps>0$. Let us denote by
$\alpha$ the midpoint of this interval. We also need the function $\chi$ defined by
\begin{equation*}
\chi(u)=\frac{1}{\vert x-y\vert }\log\{\expo_{(x,y)}(u+2iK') \},
\end{equation*}
which is analytic in the cylinder $\mathbb R/(4K\mathbb Z)+(-2iK',2iK')$.

\begin{thm}
\label{thm:asymp_Green}
Let $\Gs$ be a quasicrystalline isoradial graph. When the distance $\vert x-y\vert $ between vertices $x$ and $y$ of $\Gs$ is large,
we have
\begin{equation}
\label{eq:alt_asymp_green}
G^{m(k)}(x,y)=\frac{k'}{2\sqrt{2\pi\vert x-y\vert \chi''(u_0\vert k)}}e^{\vert x-y\vert \chi(u_0\vert k)}\cdot(1+o(1)),
\end{equation}
where $u_0$ is the unique 
$u\in\alpha+(-K+\eps,K-\eps)$ such that $\chi'(u\vert k)=0$, and $\chi(u_0\vert k)<0$.
\end{thm}

For periodic isoradial graphs, a geometric 
interpretation of $u_0$ is provided in Section~\ref{sec:recovAsympt}.

The proof consists in applying the saddle-point method to the 
contour integral expression~\eqref{eq:green} given in
Theorem~\ref{thm:expression_Green}. 
Note that the approach is different from \cite{Kenyon3}, where the author
obtains asymptotics of the Green function by the Laplace method, as there
are no saddle-points in the critical case. 
The proof of Theorem~\ref{thm:asymp_Green} is split as follows. In Lemma~\ref{lem:characterization_saddle_point},
we first show that there is a unique $u\in\alpha+(-K+\eps,K-\eps)$
such that $\chi'(u\vert k)=0$. 
Then, in Lemma~\ref{lem:conv0} we prove that $\chi(u_0\vert k)<0$, implying exponential decay of the Green function. 
Finally, we conclude the proof of Theorem~\ref{thm:asymp_Green}. Note that Lemmas~\ref{lem:characterization_saddle_point}
and~\ref{lem:conv0} do not use the quasicrystalline assumption.

Let us introduce some notation. Denote by $N_j$ the number of times a step $e^{i\overline{\alpha}_j}$ is taken in
a minimal path from $x$ to $y$, so that $N_1+\dotsb +N_p=\vert x-y\vert $. 
Using Equation \eqref{eq:recursive_def_expo}, one has
\begin{equation*}
     \expo_{(x,y)}(u)=\left\{i\sqrt{k'}\sc\left(\frac{u-\alpha_1}{2}\right)\right\}^{N_1} \times \dotsb \times\left\{i\sqrt{k'}\sc \left(\frac{u-\alpha_p}{2}\right)\right\}^{N_p}.
\end{equation*}
 With
$n_j=N_j/\vert x-y\vert$, the function $\chi$ is equal to
\begin{equation}
     \chi(u)= n_1\log \left\{\sqrt{k'}\nd\left(\frac{u-\alpha_1}{2}\right)\right\}+
     \dotsb 
     + n_p \log\left\{\sqrt{k'}\nd\left(\frac{u-\alpha_p}{2}\right)\right\},
     \label{eq:def_chi}
\end{equation}
where we used \eqref{id:sciKp} to simplify
$\expo_{(x,y)}(u+2iK')$. Because of the logarithm, $\chi$ 
is not meromorphic on $\TT(k)$, but is meromorphic (and even analytic) 
in the cylinder $\mathbb R/(4K\mathbb Z)+(-2iK',2iK')$.

\begin{lem}
\label{lem:characterization_saddle_point}
There is a unique $u_0$ in
$\alpha+(-K+\eps,K-\eps)$ such that $\chi'(u\vert k)=0$.
\end{lem}


\begin{proof}
Rotating $\Gs$, we assume that $\alpha=0$. 
Using \cite[2.5.8]{La89}, the equation of Lemma~\ref{lem:characterization_saddle_point} is equivalent to:
\begin{equation}
\label{eq:sum_dnsncn}
     n_1 \frac{\sn \cdot \cn}{\dn}\left(\frac{u-\alpha_1}{2}\right)
     + \dotsb +
     n_p \frac{\sn \cdot \cn}{\dn}\left(\frac{u-\alpha_p}{2}\right)=0.
\end{equation}
The above function is meromorphic on the torus $\TT(k)$. By 
Landen's ascending transformation (see \eqref{eq:alt} of Appendix~\ref{app:elliptic}), 
Equation~\eqref{eq:sum_dnsncn} can be rewritten in a simpler way, as follows:
\begin{equation}
\label{eq:to_solve}
     n_1\sn (v-\gamma_1\vert \ell)+\dotsb +n_p\sn (v-\gamma_p\vert \ell)=0,
\end{equation}
where $\ell$ and $\mu$ are defined in \eqref{eq:alt}, and we have noted $v=\frac{(1+\mu)u}{2}$ and $\gamma_i=\frac{(1+\mu)\alpha_i}{2}$.

Using the relation~\eqref{eq:K_Kp_after_Landen} between $K,K'$ and $K(\ell),K'(\ell)$, and the identity $(1+\mu)(1+\ell)=2$, 
under the change of variable, the torus
$\TT(k)$ becomes $\widetilde{\TT}(\ell)=\mathbb C/(4K(\ell)\mathbb Z+2iK'(\ell)\mathbb Z)$,
and the $\gamma_i$'s are in $(-K(\ell)+\widetilde\eps,K(\ell)-\widetilde\eps\,)$. Hereafter we shall replace $\widetilde\eps$ by $\eps$.


Let $f$ be the function $f(v)=2(1+\mu)\chi'(\frac{2v}{1+\mu})$ 
defined on the left-hand side of Equation~\eqref{eq:to_solve}.
We now show that $f$ has a unique zero
on the interval $(-K(\ell)+\eps,K(\ell)-\eps)$.

First notice that in the degenerate case $\ell=0$ this is obvious: the
addition formula for the sine function gives the unique solution
$v=\arctan\Bigl(\frac{\sum_{j=1}^pn_j\sin(\gamma_j)}{\sum_{j=1}^pn_j\cos(\gamma_j)}\Bigr)\in(-\pi/2,\pi/2)$.
In the other degenerate case $\ell=1$, $\sn $ becomes the hyperbolic tangent
function, and~\eqref{eq:to_solve} is a sum of $p$ increasing functions on
$\mathbb R$, which obviously has a unique zero on $\mathbb R$. 
The situation is more complicated in the remaining cases $\ell\in(0,1)$, where we show that:
\begin{enumerate}
     \item\label{item:poles1} $f$ has $2p$ simple poles in $\widetilde{\TT}(\ell)$
       (and thus also $2p$ zeros in $\widetilde{\TT}(\ell)$, counted with~multiplicities);
     \item\label{item:poles2} $f$ has at least one zero in the interval
       $(-K(\ell)+\eps,K(\ell)-\eps)\subset \widetilde{\TT}(\ell)$, and at least one zero in $(K(\ell)+\eps,3K(\ell)-\eps)\subset \widetilde{\TT}(\ell)$;
     \item\label{item:poles3} $f$ has at least $2p-2$ zeros on
       $-iK'(\ell)+\mathbb R/(4K(\ell)\mathbb Z)$.
\end{enumerate}
From~\ref{item:poles1},~\ref{item:poles2} and~\ref{item:poles3} it immediately follows that the zero of~\eqref{eq:to_solve} 
on $(-K(\ell)+\eps,K(\ell)-\eps)$ is unique. 

Point~\ref{item:poles1} is clear: each function $n_i\sn(v-\gamma_i\vert \ell)$ has two simple poles, at points congruent to $\gamma_i-iK'(\ell)$ and $\gamma_i-iK'(\ell)+2K(\ell)$. 
The poles cannot compensate for different values of $i$, since the $\gamma_i$'s are in an interval whose length is strictly less than $2K(\ell)$.

The intermediate value theorem yields Point~\ref{item:poles2}. At $v=-K(\ell)+\eps$ (resp.\ $v=K(\ell)-\eps$),
each $\sn(v-\gamma_i\vert \ell)$ is negative (resp.\ positive). 
We thus have at least one solution in $(-K(\ell)+\eps,K(\ell)-\eps)$. 
Since $\sn(v+2K(\ell)\vert \ell)=-\sn(v\vert \ell)$, the same holds in the interval $(K(\ell)+\eps,3K(\ell)-\eps)$.

We now prove Point~\ref{item:poles3}. In an interval of the form $-iK'(\ell)
+ [\gamma_i,\gamma_j]$, where $\gamma_i$ and $\gamma_j$ are consecutive, the
function $f$ has at least one zero. Indeed, evaluating $f$ at $-iK'+v$ and using
the addition formula~\eqref{id:sniKp} for $\sn$ by a quarter-period, we obtain
\begin{equation*}
     \frac{1}{\ell}\left(
      \frac{n_1}{\sn(v-\gamma_1\vert \ell)}+\dotsb
     +\frac{n_p}{\sn(v-\gamma_p\vert \ell)}
     \right).
\end{equation*}
As $v\to\gamma_i+0$ (resp.\ $v\to\gamma_j-0$)
the above function goes to $+\infty$ (resp.\ $-\infty$).
We conclude with the intermediate value theorem. We thus obtain $p-1$ zeros. 
The same reasoning on $-iK'(\ell) + [\gamma_i+2K(\ell),\gamma_j+2K(\ell)]$ 
provides $p-1$ zeros. These zeros are mutually disjoint. 
\end{proof}


\begin{lem}
\label{lem:conv0}
One has the following inequality, implying exponential decay of the Green function,
\begin{equation*}
     \chi(u_0)=\min\{ \chi(u): {u\in\alpha+(-K+\eps,K-\eps)}\}\leq \log\{\sqrt{k'}\nd (\eps/2)\}<0.
\end{equation*}
\end{lem}

\begin{proof}
First recall from Lemma~\ref{lem:characterization_saddle_point} and its proof that, on the interval 
$\alpha+(-K+\eps,K-\eps)$, in the neighborhood of which $\chi$ is analytic,
$\chi'$ has a unique zero. It is negative at $\alpha-K+\eps$ and positive at $\alpha+K-\eps$. 
This implies that $\chi(u_0)$ is the minimum of $\chi$ on $\alpha+(-K+\eps,K-\eps)$. 


We now compute the value of $\chi(\alpha)$:
\begin{equation*}
     \chi(\alpha) = n_1\log \left\{\sqrt{k'}\nd \left(\frac{\alpha-\alpha_1}{2}\right)\right\}+ \dotsb +n_p\log \left\{\sqrt{k'}\nd \left(\frac{\alpha-\alpha_p}{2}\right)\right\}.
\end{equation*}
For $u\in(-K/2,K/2)$ one has $\nd (u)\in[1,1/\sqrt{k'})$, see~\cite[16.5.2]{AS}. Further, $\nd$ is decreasing (resp.\ increasing) on $(-K/2,0]$ (resp.\ $[0,K/2)$). This implies that each term above satisfies
\begin{equation*}
     n_j\log \left\{\sqrt{k'}\nd \left(\frac{\alpha-\alpha_j}{2}\right)\right\}\leq n_j\log\{\sqrt{k'}\nd (\eps/2)\}.\qedhere
\end{equation*}
\end{proof}

\begin{proof}[Proof of Theorem~\ref{thm:asymp_Green}]
  Starting from Equation~\eqref{eq:green} defining $G^m(x,y)$,
  performing the change of variable $u+ 2iK'\rightarrow u$  
  and using the definition of $\chi$, the Green
  function between $x$ and $y$ is rewritten as
  \begin{equation*}
    G^m(x,y)=\frac{k'}{4i\pi}\int_{\Cs_{x,y}} e^{\vert x-y\vert \chi(u)}\ud u,
  \end{equation*}
where $\Cs_{x,y}$ is the vertical closed loop defined in
Theorem~\ref{thm:expression_Green}, and is thus invariant by vertical
translation.
Our aim is to compute asymptotics of this integral
when $\vert x-y\vert $ is large. We use the saddle-point method
(for classical facts our
reference is~\cite[Chapter 8]{Cop}), with
some particularities coming from the fact the $n_j$'s involved in the function
$\chi$ depend on $\vert x-y\vert $ and do not necessarily converge 
as $\vert x-y\vert \to \infty$.
For this reason, we will typically have to apply the saddle-point method in a
uniform way. 

When the graph $\Gs$ is periodic, however, this is the classical saddle-point method.
Indeed, we can write
$\expo_{(x,y)}(u)=\expo_{(x,y')}(u)\expo_{(y',y)}(u)$, where $y'$ is
the point congruent to $y$ in the same fundamental domain as $x$. Then
the periodicity allows to write $\expo_{(y',y)}(u)$  as
$\expo_{(y',y'')}(u)^L$. We thus have to compute the asymptotics of
$\int_{\Cs_{x,y}}\expo_{(x,y')}(u)\expo_{(y',y'')}(u)^L\ud u$  for large values of
$L$, all functions in the integral being independent of $L$.
Although the periodic case could be considered more classically, and thus apart,
we choose to treat both the periodic and non-periodic cases at the same time. 
We move the contour $\Cs_{x,y}=\varphi_{x,y}+[-2iK',2iK']$ into a new one $\Cs'_{x,y}$, directed upwards, going through $u_0$ and satisfying some
further properties, to be specified now. See Figure \ref{The_rectangle} for a representation of the contours $\Cs_{x,y}$ and $\Cs'_{x,y}$.

\paragraph{Neighborhood of the saddle-point.}
Near $u_0$ we choose $\Cs'_{x,y}$ to be
$[u_0-i\eta,u_0+i\eta]$, where $\eta=\vert x-y\vert^{-\alpha}$, $\alpha>0$ being fixed later
on. Hereafter we write $\chi(u)=\chi(u_0)+\sum_{j=2}^\infty a_j(u-u_0)^j$, and define 
$
     F(u) = \chi(u)-\chi(u_0)-a_2(u-u_0)^2=\sum_{j=3}^\infty a_j(u-u_0)^j
$.
This function is analytic on a disc centered at $u_0$ and with some radius
$r$; $M$ denotes its maximum on the disc. Simple computations lead to the
upper bound (see~\cite[Equation (36.3)]{Cop} for full details)
\begin{equation}
\label{eq:ub_F} 
     \vert F(u)\vert \leq \frac{M\vert u-u_0\vert ^3}{r^2(r-\vert u-u_0\vert )}\leq \frac{2M}{r^3}\vert x-y\vert^{-3\alpha}\leq C\cdot \vert x-y\vert^{-3\alpha}.
\end{equation}
In the last inequality we can choose $C$ to be independent of $\vert x-y\vert$, thanks to the fact that $\chi$ depends continuously on the $n_j$'s.
With the above estimation one can write
\begin{align}
\label{eq:int_part1}
     \int_{[u_0-i\eta,u_0+i\eta]} e^{ \vert x-y\vert \chi(u)}\ud u
     =& e^{\vert x-y\vert\chi(u_0)}
     \int_{[u_0-i\eta,u_0+i\eta]} e^{\vert x-y\vert a_2(u-u_0)^2}\ud u 
     \cdot (1+O(\vert x-y\vert^{1-3\alpha}))\nonumber\\
     = & \frac{ie^{\vert x-y\vert\chi(u_0)}}{\sqrt{\vert x-y\vert a_2}}
     \int_{[-\sqrt{\vert x-y\vert a_2}\eta,\sqrt{\vert x-y\vert a_2}\eta]} e^{-t^2}  \ud t
     \cdot(1+O(\vert x-y\vert^{1-3\alpha}))\nonumber\\
     =&\frac{i\sqrt{\pi}e^{\vert x-y\vert\chi(u_0)}}{\sqrt{\vert x-y\vert a_2}}
     \cdot(1+O(\vert x-y\vert^{1-3\alpha}))
     \cdot(1+O(e^{-\vert x-y\vert a_2\eta^2})).
\end{align}
We now show that $a_2=\chi''(u_0)/2$ remains bounded away from $0$ independently 
of $\vert x-y\vert$. First, it comes from the analytic implicit function theorem that $u_0$ 
is an analytic function of the $n_j$'s. Accordingly, $a_2$ is positive and continuous  on $\{(n_1,\dotsc ,n_p): n_j\geq 0\text{ and } n_1+\dotsb +n_p=1\}$. Under the 
quasicrystalline hypothesis this set is compact, and thus $a_2$ can be bounded from  
below by its (positive) minimum. 
To conclude, we take any $1/3<\alpha<1/2$ to obtain that the
contribution of the neighborhood of $u_0$ to the integral gives
the result stated in Theorem~\ref{thm:asymp_Green}.

\medskip

{\bf Outside a neighborhood of the saddle-point.}
We prove that the rest of the integral does not contribute in the limit. 
For this, we show that $\Cs'_{x,y}$ can
be chosen such that: 
\begin{equation}
\label{eq:prop_cont}
     \forall u\in\Cs'_{x,y}\setminus [u_0-i\eta,u_0+i\eta],\quad \vert e^{\chi(u)}\vert \leq \vert e^{\chi(u_0\pm i\eta)}\vert .
\end{equation} 
Recall that the exponential function $\expo_{(x,y)}(u+2iK')$ has its zeros on the interval $\pm2iK'+\mathbb R/(4K\mathbb Z)$, see Figure~\ref{The_rectangle}. 
Consider the steepest
descent path starting from $u_0\pm i\eta$
 until it hits one of the zeros
of the exponential function
(note that obviously it cannot cross the line $[-2K, 2K] + 2iK'$ before).
The resulting contour $\Cs'_{x,y}$ is a deformation of $\Cs_{x,y}$ and is symmetric (the level lines and hence the steepest
descent paths of $e^{\chi(u)}$ are 
symmetric with respect to the horizontal axis, this comes from properties of the $\nd(\cdot)$ function, see~\cite[16.21.4]{AS}).


On $\Cs'_{x,y}\setminus [u_0-i\eta,u_0+i\eta]$, we have $\vert e^{\vert x-y\vert \chi(u)}\vert \leq e^{\vert x-y\vert\chi(u_0)}e^{-\vert x-y\vert a_2\eta^2}e^{C\vert x-y\vert ^{1-3\alpha}}$ by \eqref{eq:ub_F} and \eqref{eq:prop_cont}, which readily implies that
\begin{equation}
\label{eq:int_part2}
     \int_{\Cs'_{x,y}\setminus [u_0-i\eta,u_0+i\eta]} e^{\vert x-y\vert \chi(u)}\ud u=O(e^{\vert x-y\vert \chi(u_0)}e^{-\vert x-y\vert a_2\eta^2}e^{C\vert x-y\vert ^{1-3\alpha}}),
\end{equation}
since we can take the supremum of the lengths of $\Cs'_{x,y}$ bounded, because of the continuity of the level lines with respect to the parameters.

The integral~\eqref{eq:int_part2} is exponentially negligible with respect to
the integral~\eqref{eq:int_part1} on $[u_0-i\eta,u_0+i\eta]$. The proof of
Theorem~\ref{thm:asymp_Green} is complete.
\end{proof}

\begin{rem}
\label{rem:non_cryst} 
If there were an infinite number of directions $\alpha_j$ (non-quasicrystalline case), the Green function
would still exponentially decay to $0$, see
Proposition~\ref{thm:exp_bounded_Green} and Lemma~\ref{lem:conv0}. However, our
conjecture is that the asymptotic behavior is exactly the same as in
Theorem~\ref{thm:asymp_Green}. The technical issue is to prove that  
 $\chi''(u_0)$ remains bounded away from $0$ as $\vert x-y\vert$ becomes large. 
 From our analysis, we only know that the 
second derivative at $u_0$ is non-negative.
\end{rem}

\begin{prop}
\label{thm:exp_bounded_Green}
Let $\Gs$ be any infinite isoradial graph (not necessarily quasicrystalline).
When the distance $|x-y|$ between vertices $x$ and $y$ of $\Gs$ is large, we have
%
\begin{align*}
     G^{m(k)}(x,y)=O\bigl(e^{\vert x-y\vert \chi(u_0\vert k)}\bigr),
\end{align*}
where $u_0$ is the unique 
$u\in\alpha+(-K+\eps,K-\eps)$ such that $\chi'(u\vert k)=0$, and $\chi(u_0\vert k)<0$.
\end{prop}

\begin{proof}
The proof is the same as the one of Theorem~\ref{thm:asymp_Green}: there exists
a contour $\Cs'_{x,y}$ such that \eqref{eq:prop_cont} holds with $\eta=0$. In this
way, the upper bound of Proposition~\ref{thm:exp_bounded_Green} immediately
follows.
\end{proof}

\section{The case of periodic isoradial graphs}\label{sec:periodic_case}



In this section, we suppose that the isoradial graph $\Gs$ is
$\mathbb{Z}^2$-periodic, meaning that $\Gs$ is embedded in the plane so that it is invariant under 
translations of $\ZZ^2$. The massive Laplacian $\Delta^{m(k)}$ is a periodic
operator. It is described through its Fourier transform $\Delta^{m(k)}(z,w)$, which is
the massive Laplacian matrix of the toric graph $\Gs_1=\Gs/\ZZ^2$, with modified edge-weights on edges crossing a horizontal
and vertical cycle. Objects of interest are: the \emph{characteristic polynomial}, $P_{\Delta^{m(k)}}(z,w)$,
equal to the determinant of the matrix $\Delta^{m(k)}(z,w)$; the zero locus of this polynomial, known as the
spectral curve and denoted $\Ccal^k$, and its amoeba~$\A^k$.

In Section~\ref{sec:qperfunc}, we prove confinement results for the Newton polygon of the 
characteristic polynomial $P_{\Delta^{m(k)}}(z,w)$. In
Section~\ref{sec:spectral_amoeba}, we provide an explicit parametrization of the spectral curve $\Ccal^k$
by discrete massive exponential functions. This allows us to prove that $\Ccal^k$ is a curve of genus 1 
(Proposition~\ref{prop:geom_torus}),
and to recover the Newton polygon using the homology of the train-tracks only. In Theorem~\ref{thm:harnack}, we prove that
the curve $\Ccal^k$ is a Harnack curve. Furthermore, in Theorem~\ref{thm:Harnack2}, we prove that every genus~1, Harnack curve
with $(z,w)\leftrightarrow(z^{-1},w^{-1})$ symmetry arises from the massive Laplacian $\Delta^{m(k)}$ on some isoradial graph for some $k\in(0,1)$.

Using Fourier techniques, the massive Green function can be expressed using the characteristic polynomial. In Section~\ref{sec:perio_green},
we explain how to recover the local formula of Theorem~\ref{thm:expression_Green} for the massive Green function
(in the periodic case) from its Fourier expression. 
A priori, the two approaches
are completely different; it is an astonishing change of variable which works with our specific
choice of weights. Note that this relation was not understood in~\cite{Kenyon3,BoutillierdeTiliere:iso_gen,deTiliere:quadri}.
Then, we also explain how to recover asymptotics of the Green function obtained in Theorem~\ref{thm:asymp_Green}
from the double integral formula of Equation~\eqref{eq:periodic_Green}, using
analytic combinatorics techniques from~\cite{PeWi13}.


%


\subsection{Isoradial graphs on the torus and their train-tracks}\label{sec:traintrack_torus}

If $\Gs$ is a $\ZZ^2$-periodic isoradial graph, then the graph $\Gs_1=\Gs/\ZZ^2$
is an
isoradial graph embedded in the torus $\torus$. 
Let $\GR_1$ be the diamond graph of $\Gs_1$.
Properties of train-tracks of planar isoradial graphs discussed in
Section~\ref{sec:diamondetc} have to be adapted to the toroidal case, see also~\cite{KeSchlenk}.

We need the notion of \emph{intersection form} for closed paths on the torus $\torus$.
  Let $A$ and $B$ be two oriented closed paths on the torus $\torus$ having a finite number of intersections. Then 
  $A\wedge B$ denotes the algebraic number of intersections between $A$ and $B$, where
  an intersection is counted positively (resp.~negatively) if when following
  $A$, we see $B$ crossing from right to left (resp.~from left to right).
  In particular,
  \begin{equation*}
    \vert A\wedge B\vert  \leq \#\{\text{intersection points of $A$ and $B$}\}.
  \end{equation*}

  The quantity $A\wedge B$ only depends on the homology classes $[A]$ and $[B]$
  in $H_1(\torus,\mathbb{Z}^2)$. If $([U], [V])$ is a homology basis, and
  $[A]=h_A [U] +v_A [V]$, $[B] = h_B [U] + v_B [V]$, then 
  \begin{equation}\label{equ:intersection}
  A\wedge B = h_A v_B - h_B v_A.
  \end{equation}

Recall from Section~\ref{sec:diamondetc} that train-tracks can be seen as unoriented paths on the dual of the diamond graph, and that 
they are assigned the common edge
direction $\pm e^{i\overline{\alpha}}$ of the rhombi. 
A train-track $T$ can also be seen as an oriented path. In this case we associate
the angle $\alpha$ of
the unit vector $e^{i\overline{\alpha}}$,
with the convention that when
walking along $T$, the unit vector $e^{i\overline{\alpha}}$ crosses $T$ from
right to left. 
If $T$ is oriented in the other direction, it is associated the 
angle $\alpha+2K$ (modulo $4K$) of $e^{i\overline{\alpha+2K}}=e^{i(\overline{\alpha}+\pi)} =
-e^{i\overline{\alpha}}$. Seeing train-tracks as unoriented paths amounts to 
considering angles modulo $2K$.

\paragraph{Train-tracks on the torus.}
Train-tracks of $\Gs_1$ form non-trivial self-avoiding
cycles on ${\GR_1}^{*}$. Contrary to what happens in the planar
(either finite or infinite) case, two train-tracks $T$ and $T'$ can cross more than once, but
the number of intersections is minimal, and thus equal to $\vert T\wedge T'\vert$.

Any vertex of ${\GR_1}^*$ is at the intersection of two train-tracks, and edges of
${\GR_1}^*$ are in bijection with pieces of train-tracks between two successive
intersections.
As in the planar case, $\GR_1$ is bipartite, the two classes of
vertices corresponding to vertices of $\Gs_1$ and of $\Gs_1^*$, respectively. In
particular, any closed path on $\GR_1$ has even length.
Conversely, any graph on the torus constructed from a collection of
self-avoiding cycles with the
minimal number of intersections, and whose dual is bipartite, is the dual of the
diamond graph of an isoradial graph on the torus, which can then be lifted to a
$\ZZ^2$-periodic isoradial graph. An example is provided in Figure~\ref{fig:iso_perio} (left and middle).

\paragraph{Minimal closed paths.}
A closed path on $\GR_1$ is said to be \emph{minimal} if it does not cross a
train-track in two opposite directions. Paths in $\GR_1$ obtained by following
the boundary of a train-track are examples of minimal closed paths. 

Let $2p$ be the number of edges of $\GR_1$ used by a minimal closed
path $\gamma$. Then $p$ is the number of vertices of $\Gs_1$ (and also of
$\Gs_1^*$) visited by $\gamma$. The length $2p$ of $\gamma$ is a function of its homology class 
and of those of the train-tracks. It is equal to:
\begin{equation*}
    2p= \sum_{T\in\T} \vert T\wedge \gamma\vert,
  \end{equation*}
where $\T$ denotes the set of train-tracks of $\Gs_1$, picking for each of them a particular orientation. 

\paragraph{Choice of basis, ordering of train-tracks.}
Fix a representative of a basis of the first homology group of the torus $H_1(\torus,\ZZ^2)$,
by taking $\gamma_x$ and $\gamma_y$ to be two minimal oriented closed paths on $\GR_1$.
Define $2p_x$ (resp.~$2p_y$) to be the number of
edges of $\GR_1$ used by $\gamma_x$ (resp.~$\gamma_y$).
These numbers depend only on the homology classes of $\gamma_x$ and $\gamma_y$.

An (oriented) train-track $T$ has primitive homology class $[T]=
h_T[\gamma_x]+v_T[\gamma_y]$
in $H^1(\torus,\mathbb{Z})\simeq \mathbb{Z}^2$, \emph{i.e.},
the two integers $h_T$ and $v_T$ are coprime, since it is a non-trivial, self-avoiding cycle.
We can therefore cyclically order all train-tracks (oriented in the two possible
directions), following the cyclic order of coprime
numbers in $\ZZ^2$ around the origin. Angles of the train-tracks are also in the
same order in $\RR/4K\ZZ$, this being guaranteed by the fact that we can place a rhombus at each intersection of two train-tracks, 
with the correct orientation, see Figure~\ref{fig:iso_perio} (right).

\begin{figure}
  \centering
  \resizebox{1\textwidth}{!}{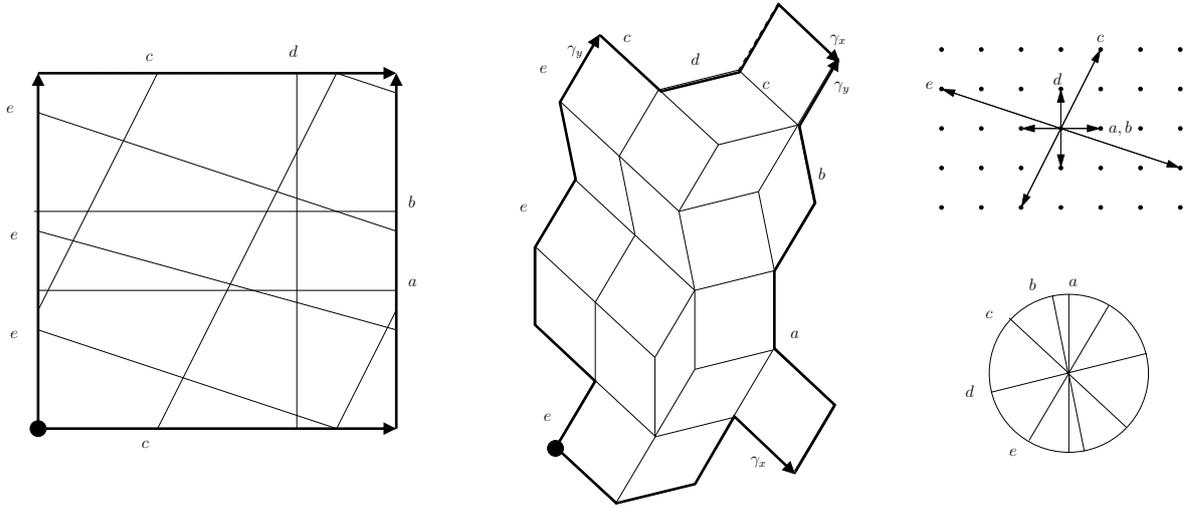}
  \caption{Left: a set of non-trivial cycles on the torus with the minimum
    number of intersections. Middle: the diamond graph of an isoradial graph on
    the torus, whose train-tracks have the same combinatorics as cycles on
    the left. Right: cyclic ordering of the homology class of the train-tracks and of the 
    corresponding angles, represented on the trigonometric circle.}
  \label{fig:iso_perio}
\end{figure}

\subsection{Quasiperiodic functions, characteristic polynomial}
\label{sec:qperfunc}


Define $\widetilde{\gamma}_x$ and $\widetilde{\gamma}_y$ to be closed paths on
$\Gs_1^*$, obtained from $\gamma_x$ and $\gamma_y$ as follows: replace any
sequence of steps $x^*\to y \to z^*$ of dual,
primal, dual vertices visited by $\gamma_x$ (resp.\ $\gamma_y$) by a sequence
$x^*=x^*_0 \to x^*_1 \to\cdots \to x^*_{n} = z^*$ of dual vertices around $y$,
``bouncing'' over
$y$ on top of $\gamma_x$ (resp.\ on the right of $\gamma_y$), and remove
backtracking steps if necessary. In other words,
$\widetilde{\gamma}_x$ goes around every vertex of $\Gs_1$ visited by $\gamma_x$ in the
clockwise order, see Figure~\ref{fig:paths_fund_domain}.

The cycles $\widetilde{\gamma}_x$ and $\widetilde{\gamma}_y$ delimit a \emph{fundamental domain} of $\Gs$.
To simplify notation, we will write $\widetilde{\gamma}_x$ and $\widetilde{\gamma}_y$ for the cycles and 
their lifts in $\Gs$, and write $\Gs_1=(\Vs_1,\Es_1)$ for the toroidal graph and the fundamental domain.

\begin{figure}[ht]
  \centering
  \resizebox{!}{6cm}{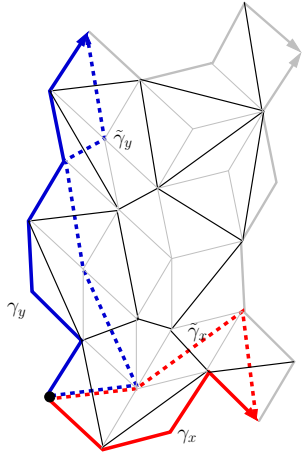}
  \caption{%
     The paths $\widetilde{\gamma}_x$ and $\widetilde{\gamma}_y$ (in dashed
lines) constructed from $\gamma_x$ and $\gamma_y$ (in plain lines)
  delimiting the fundamental domain~$\Gs_1$.
  Edges crossing the dashed path $\widetilde{\gamma}_x$ (resp.
  $\widetilde{\gamma}_y$) get an extra weight $z,z^{-1}$ (resp.\ $w,w^{-1}$) in 
$\Delta^m(z,w)$.
\label{fig:paths_fund_domain}}
\end{figure}

For $(m,n)$ in $\ZZ^2$, and $x$ a vertex of $\Gs$ (resp.\ $\GR$, resp.\ $\Gs^*$),
denote by $x+(m,n)$ the vertex $x+m\widetilde{\gamma}_x+n\widetilde{\gamma}_y$.
For $(z,w)\in \CC^2$, define $\CC^{\Vs}_{(z,w)}$ to be
the space of functions $f$ on vertices of $\Gs$
which are $(z,w)$-quasiperiodic:
\begin{equation*}
  \forall\, \text{$x\in\Vs$},\ \forall\, (m,n)\in\ZZ^2,\quad f(x+(m,n))=z^{-m} w^{-n} f(x).
\end{equation*}
The vector space $\CC^{\Vs}_{(z,w)}$ is finite dimensional, isomorphic to
$\CC^{\Vs_1}$, since quasiperiodic functions are completely
determined by their values in the fundamental domain $\Vs_1$.
For every vertex $x$ of $\Gs_1$, define
$\delta_x(z,w)$ to be the $(z,w)$-quasiperiodic function equal to zero on vertices
which are not translates of $x$, and equal to $1$ at $x$. Then the
collection $\{\delta_x(z,w)\}_{x\in\Vs_1}$ is a natural basis for $\CC^{\Vs}_{(z,w)}$.

Since $\Delta^m$ is periodic, the vector space $\mathbb{C}^{\Vs}_{(z,w)}$ is
invariant under the action of this operator. We denote by $\Delta^m(z,w)$ the matrix
of the restriction of $\Delta^m$ to the space $\CC^{\Vs}_{(z,w)}$ in the basis
$(\delta_x(z,w))_x$. The matrix $\Delta^m(z,w)$ can be seen as the matrix of the massive Laplacian on $\Gs_1$
with extra weight $z^{\pm 1}$ (resp.~$w^{\pm 1}$) for edges crossing\footnote{An
  oriented edge crossing $\widetilde\gamma_y$ gets the extra weight $z$ if it goes
  from a vertex in one fundamental domain to a vertex in the fundamental domain
  on the right of $\widetilde\gamma_y$, and $z^{-1}$ otherwise. An oriented
  edge crossing $\widetilde\gamma_x$ gets the extra weight $w$ if it goes from a
  vertex in one fundamental domain to a vertex in the fundamental domain above,
  \emph{i.e.}, on the left of $\widetilde\gamma_x$.
}
$\widetilde\gamma_x$ (resp.~$\widetilde\gamma_y$), the sign of the exponent depending on
the orientation of the edge with respect to $\widetilde\gamma_x$ or
$\widetilde\gamma_y$. By construction, these edges with extra weight $z$
(resp.\ $w$) are connected to vertices of $\Gs_1$ visited by $\gamma_y$ (resp.\ $\gamma_x$).

The \emph{characteristic polynomial} of the massive Laplacian on $\Gs$ is the
  bivariate Laurent polynomial $P_{\Delta^m}(z,w)$ equal to the determinant
  of the matrix $\Delta^m(z,w)$.
The \emph{Newton polygon} of $P_{\Delta^m}$ is the convex hull of the exponents $(i,j)\in\ZZ^2$ of 
the monomials $z^i w^j$ of $P_{\Delta^m}(z,w)$.
  
The characteristic polynomial plays an important role in understanding the massive Laplacian on periodic
isoradial graphs. We now study some of its properties.

\begin{lem}$\,$
  \label{lem:inclu_Newton_PDelta}
  \begin{itemize}
  \item The polynomial $P_{\Delta^m}$ is reciprocal: $
    \ \forall\, (z,w)\in {\CC}^2, \ 
    P_{\Delta^m}(z,w)=P_{\Delta^m}(z^{-1},w^{-1}).$
  \item The Newton polygon of $P_{\Delta^m}$ is contained in a rectangle
  $[-p_y,p_y]\times[-p_x,p_x]$, where $p_x$ (resp.
  $p_y$) is the number of
  vertices of $\Gs^*$ on $\gamma_x$ (resp.\ $\gamma_y$).
\end{itemize}
  \end{lem}

\begin{proof}
  The operator $\Delta^m$ is symmetric, thus
   $ {\Delta^{m}(z,w)}^T = \Delta^{m}(z^{-1},w^{-1})$, which implies the first part.
For the second part,
  let us first prove that the Newton polygon is contained in a vertical strip
  $[-p_y,p_y]\times \RR$.  Since $P_{\Delta^m}$ is reciprocal, it is enough to
  show that the degree of $z$ in any monomial of $P_{\Delta^m}$ cannot exceed
  $p_y$.
  
  The determinant of $\Delta^m(z,w)$ can be expanded as a sum over permutations
$\sigma$ of the vertices of $\Gs_1$. In this sum,
  the monomials with highest
  degree in $z$ come from bijections $\sigma$ 
  where as
  many vertices $v$ as possible are connected to $\sigma(v)$ with an edge
  crossing the path $\widetilde{\gamma}_y$, hence having an extra weight $z$ in
  $\Delta^m(z,w)$. However, there are at most $p_y$ vertices with this property,
  since they must be chosen among the $p_y$ vertices visited by $\gamma_y$.

  The fact that the Newton polygon is also contained in a horizontal strip
  $\RR\times[-p_x,p_x]$ follows from the same argument, by exchanging the role
  of $z$ and $w$.
\end{proof}

The confinement result for the Newton polygon in the previous lemma is highly
dependent on the homology class of the cycles $\gamma_x$ and $\gamma_y$.
If instead we use paths $\ubar{\gamma}_x,\ubar{\gamma}_y$
representing another basis of the
first homology group $H_1(\torus,\ZZ^2)$, then we obtain that the Newton polygon is included in
another parallelogram. More precisely, suppose that the paths $\ubar{\gamma}_x,\ubar{\gamma}_y$ are oriented so that
  $[\ubar{\gamma}_x] = a [\gamma_x] + b [\gamma_y]$, and
  $[\ubar{\gamma}_y] = c [\gamma_x] + d [\gamma_y]$,
where $M=\bigl(\begin{smallmatrix}a&b\\c&d\end{smallmatrix}\bigr)\in SL_2(\mathbb{Z})$.
Define also $\ubar{z}=z^a w^b$, and $\ubar{w}=z^c w^d$.
Then we can do the same construction as above with the variables $\ubar{z}$ and $\ubar{w}$
across the paths $\ubar{\gamma}_y$ and $\ubar{\gamma}_x$, to get a new
polynomial $\ubar{P}_{\Delta^m}(\ubar{z},\ubar{w})$. The polynomials $\ubar{P}_{\Delta^m}$ and $P_{\Delta^m}$ are related by the formula:
\begin{equation*}
  P_{\Delta^m}(z,w)=\ubar{P}_{\Delta^m}(z^{a} w^{b},z^{c}w^{d})=\ubar{P}_{\Delta^m}(\ubar{z},\ubar{w}).
\end{equation*}
The Newton polygon of $P_{\Delta^m}$ (in the $(z,w)$ variables) is obtained
as the image of that of $\ubar{P}_{\Delta^m}$ (in the $(\ubar{z},\ubar{w})$ variables) by the linear
map $M$.

We can apply the previous lemma to $\ubar{P}_{\Delta^m}$, and get that its
Newton polygon is included in a rectangle. The Newton polygon of
$P_{\Delta^m}$ is therefore included in the parallelogram, obtained as the image by $M$ of
that rectangle. In particular, it bounds the width
of the Newton polygon of $P_{\Delta^m}$ between two parallel lines with any rational
slope, this width being related to the
number of edges of the minimal paths with a certain homology.

Of particular interest is the case where $(\ubar{z},\ubar{w})=(z,w/z)$, \emph{i.e.}, where 
$M=\bigl(\begin{smallmatrix} 1 & -1 \\ 0 & \phantom{-}1\end{smallmatrix}\bigr)$.
Indeed, the horizontal width
of the Newton polygon of the polynomial $\ubar{P}_{\Delta^m}$ is directly related to
the degree of $P_{\Delta^m}$,
computed as the sum
of the degrees in $z$ and $w$.

\begin{cor}
  \label{cor:control_degree_P}
  Let $\gamma$ be a minimal closed path on $\Gs_1$ such that
  $[\gamma]=-[\gamma_x]+[\gamma_y]$, visiting $p$ vertices of $\Gs_1$ (and
  having $2p$ edges). Then, the Newton polygon of $P_{\Delta^m}$ is contained in
  a band delimited by the straight lines: $y+x\pm p=0$. In particular, the highest
  (resp.~lowest) degree of a monomial of $P_{\Delta^m}$ is not greater than $p$
  (resp.~not less than $-p$).
\end{cor}


\subsection{The spectral curve and the amoeba of the massive Laplacian}\label{sec:spectral_amoeba}

The zero set of the characteristic polynomial $P_{\Delta^{m(k)}}$ of 
the massive Laplacian defines a curve, known as the \emph{spectral curve}, denoted $\Ccal^k$:
\begin{equation*}
\Ccal^k=\{(z,w)\in \mathbb{C}^2:P_{\Delta^{m(k)}}(z,w)=0\}.
\end{equation*}

In Proposition~\ref{prop:geom_torus}, 
we show that the spectral curve has geometric genus~$1$
and in Theorem~\ref{thm:harnack}, we prove that it is \emph{Harnack}.  
In Theorem~\ref{thm:Harnack2}, we 
prove that every genus~$1$ Harnack curve with $(z,w)\leftrightarrow(z^{-1},w^{-1})$ symmetry
  is the spectral curve of the massive
  Laplacian of a periodic isoradial graph, for a certain value of $k\in(0,1)$.
These two points are reminiscent of what has been done in~\cite{KO2} for the correspondence
between genus $0$ Harnack curves and critical dimer spectral curves on isoradial graphs.


The \emph{real locus} of the spectral curve consists of
the set of points of $\Ccal^k$ that are invariant under complex conjugation:
  \begin{equation*}
\{(z,w)\in \Ccal^k\ \vert \ (\overline{z},\overline{w})=(z,w)\}=
\{(x,y)\in\RR^2\ \vert \ P_{\Delta^{m(k)}}(x,y)=0\},
\end{equation*} 
apart from isolated points. The latter are referred to as \emph{solitary nodes} of the curve.

The \emph{amoeba}~$\A^k$ of the curve $\Ccal^k$
is the image of $\Ccal^k$ under
$\Log:(z,w)\to(\log\vert z\vert ,\log\vert w\vert )$. 

General geometric features of the amoeba can be described from the Newton polygon
of the characteristic polynomial, see~\cite{Viro,GKZ} for an overview. 
It reaches infinity by several \emph{tentacles}, which are images by $\Log$ of neighborhoods of
the curve $\Ccal^k$ where $z$ and/or $w$ is $0$ or infinite. Each tentacle (counted with multiplicity) corresponds to a segment 
between two successive integer points on the boundary of the Newton polygon, and the direction of the asymptote is 
the outward normal to the segment. The amoeba's complement consists of components between the tentacles,
and bounded components. Components of the amoeba's complement are convex.  Bounded (resp.\ unbounded) components correspond to integer
points inside (resp.\ on the boundary of) the 
Newton polygon. The maximal number of bounded components is thus the number of inner integer points of the Newton polygon. 
Amoebas are unbounded, but their area is bounded by $\pi^2$ the area of the Newton polygon.

Using our explicit parametrization of the spectral curve $\Ccal^k$
allows us to prove properties of the amoeba $\A^k$, see Lemmas~\ref{lem:tentacle_amoeba}
and~\ref{lem:boundary_amoeba}. We show that its complement has a single bounded component
and prove in Proposition~\ref{prop:area_hole} that its area is increasing as a
function of the elliptic modulus $k$.

Figure~\ref{fig:amoeba} shows the Newton polygon and the amoeba of the spectral curve of the massive Laplacian of the graph
depicted in Figure~\ref{fig:iso_perio}, for $k^2=0.8$.

\begin{figure}[ht]
  \centering
  \begin{subfigure}{55mm}
  \includegraphics[width=5cm]{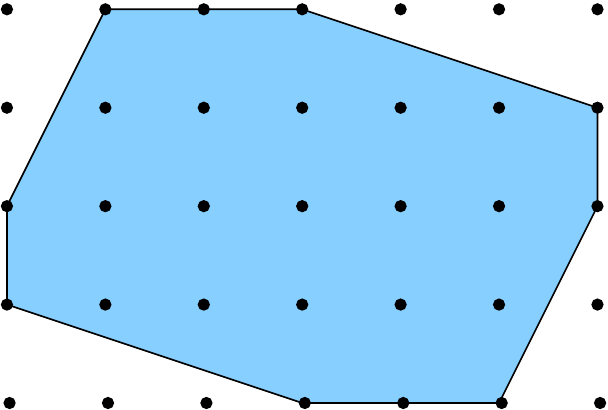}
\end{subfigure}
\qquad
  \begin{subfigure}{85mm}
  \includegraphics[width=8cm]{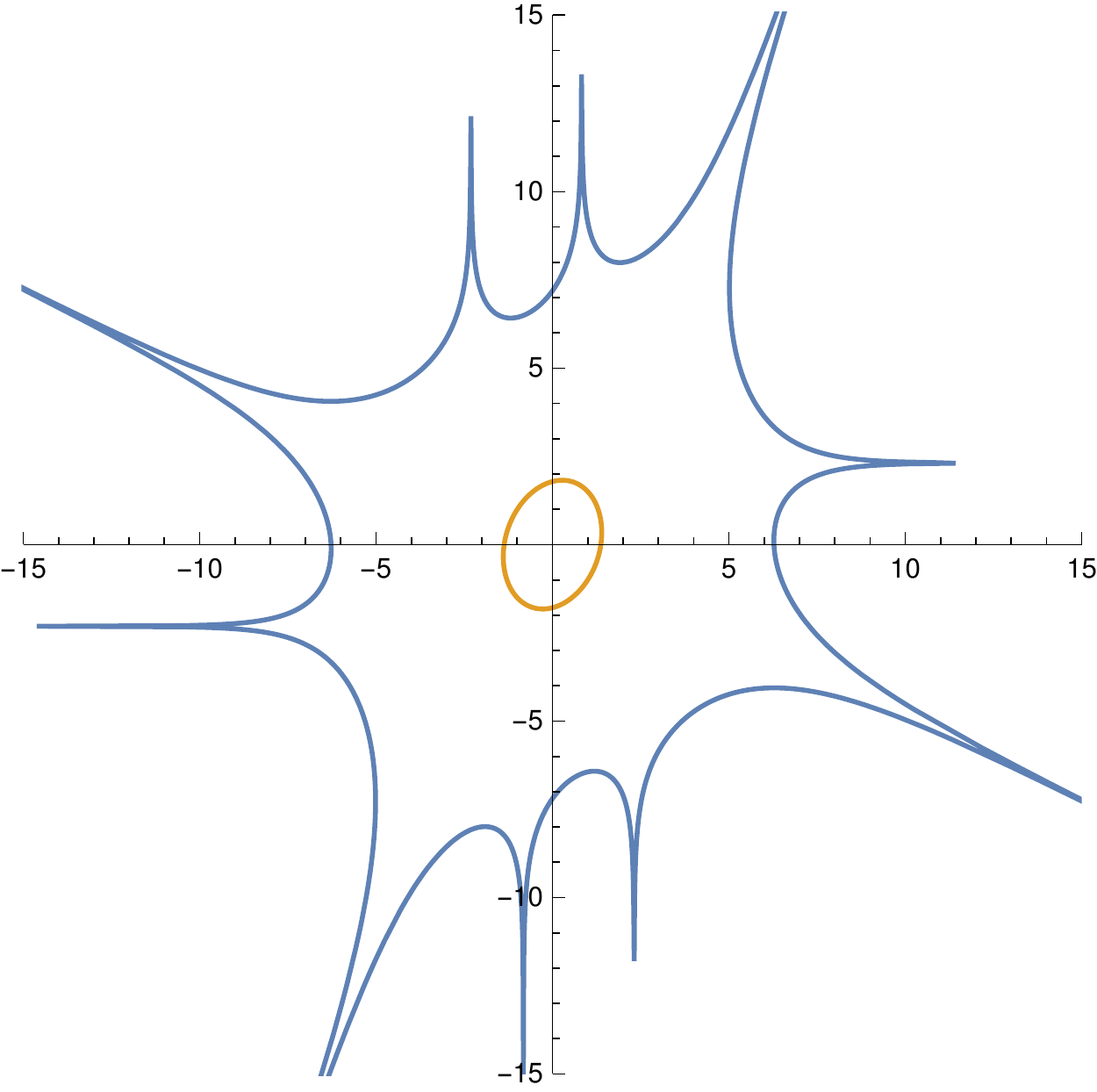}
\end{subfigure}
 \caption{Left: the Newton polygon of the massive Laplacian of 
 the graph pictured in Figure~\ref{fig:iso_perio}.
 Right: the amoeba $\A^k$ of its spectral curve $\Ccal^k$, when $k^2=0.8$.}
  \label{fig:amoeba}
\end{figure}

\subsubsection{Explicit parametrization of the spectral curve}\label{sec:paramSpectralCurve}

For $u$ in the torus $\TT(k)$, define
\begin{align}
  z(u|k) &= 
  \prod_{e^{i\overline\alpha} \in \gamma_x}
  \left(i\sqrt{k'}\sc(u_\alpha)\right)
  =
  \prod_{T\ \text{train-track in }\T}
  \left(i\sqrt{k'}\sc(u_{\alpha_T})\right)^{v_T},\nonumber\\
  w(u|k) &=
  \prod_{e^{i\overline\alpha} \in \gamma_y}
  \left(i\sqrt{k'}\sc(u_\alpha)\right)
  =
  \prod_{T\ \text{train-track in }\T}
  \left(i\sqrt{k'}\sc(u_{\alpha_T})\right)^{-h_T},
  \label{eq:def_zwofu}
\end{align}
where $\alpha_T$ is the angle associated to the oriented train-track $T$,
and $[T]=h_T[\gamma_x] + v_T[\gamma_y]$ is its homology class in $H_1(\torus,\ZZ^2)$.
There are $2p_x$ (resp.~$2p_y$) terms in the product defining $z(u|k)$
(resp.\ $w(u|k)$). Note that for every vertex $x$ of $\Gs_1$, we have
$z(u|k)=\expo_{(x,x+(1,0))}(u|k)$ and $w(u|k)=\expo_{(x,x+(0,1))}(u|k)$. Define $\psi(\cdot|k)$ to be the map:
\begin{equation*}
\begin{array}{llll}
\psi(\cdot|k):& \TT(k)&\rightarrow &\CC^2\\
&u&\mapsto &\psi(u|k)=(z(u|k),w(u|k)).
\end{array}
\end{equation*}

\begin{prop}
  \label{prop:geom_torus} 
The map $\psi$
  provides a complete parametrization of the spectral curve $\Ccal^k$ of the massive Laplacian. In
  particular, $\Ccal^k$ is an irreducible curve with geometric genus~$1$.
\end{prop}

\begin{proof}
  For every $u\in \TT(k)$, the function $\expo_{(\cdot,y)}(u)$ is 
  massive harmonic by Proposition~\ref{prop:exp_harmo}; it is
  $(z(u),w(u))$-quasiperiodic, since for every
  $(m,n)\in\ZZ^2$ and every vertex $x$ of $\Gs_1$,
  \begin{equation}
    \expo_{(x+(m,n),y)}(u)=\expo_{(x+(m,n),x)}(u)\expo_{(x,y)}(u)
    = z(u)^{-m} w(u)^{-n}\expo_{(x,y)}(u).
    \label{eq:expo_perio}
  \end{equation}
  As a consequence, for every $u$, the function $\expo_{(\cdot,y)}(u)$ belongs to the kernel 
  of $\Delta(z(u), w(u))$, and
  $
  P_{\Delta^m}(z(u), w(u))=0.
  $

  The image of the application $\psi$ is necessarily an irreducible component of
  the curve $\Ccal^k$, corresponding to the zeros of an
  irreducible factor
  $R$ of $P_{\Delta^m}$. But from the definition \eqref{eq:def_zwofu} of $z(u)$,
  we see that it has order $2p_x$: it takes the value $0$ (and thus any value) $2p_x$
  times. This means
  for example that the degree of the polynomial $R(1,w)$ is $2p_x$: indeed,
  if $u_1,\dotsc,u_{2p_x}$ are the distinct values of $u$ for which $z(u)=1$,
  then $w(u_1),\dotsc,w(u_{2p_x})$ are the roots of $R(1,w)$. But this degree is
not greater than the height of the Newton polygon of $R$. Applying the same
argument to $w(u)$, which has order $2p_y$, we get that
  the smallest rectangle containing the Newton polygon of $R$
  has height (resp.~width) $2p_x$ (resp.~$2p_y$).
  But if $R$ is not reduced to a monomial, then
  the Newton polygon of $P_{\Delta^m}$ has width strictly larger than the one of $R$ and doesn't
  fit in a $2p_y\times 2p_x$ rectangle, which is in contradiction with
  Lemma~\ref{lem:inclu_Newton_PDelta}. Therefore $R$ and $P_{\Delta^m}$ define
  the same curve in $\CC^2$ and $\psi$ parametrizes the whole spectral curve.

  The application $\psi$ is a birational map between $\TT(k)$ and the spectral curve $\Ccal^k$. The torus $\TT(k)$ is thus
  the normalization of $\Ccal^k$, and these curves have the same
  geometric genus, equal to $1$.
\end{proof}

The proof of Proposition~\ref{prop:geom_torus} shows that the bound on the width and height of the Newton
polygon obtained in Lemma~\ref{lem:inclu_Newton_PDelta} is tight, and
that the extension to other families of closed paths allows one to completely reconstruct
the Newton polygon of $P_{\Delta^{m(k)}}$, as the intersection of bands
contained between lines $ay-bx\pm p =0$.

The explicit parametrization $\psi$ of the spectral curve $\Ccal^k$ allows to
show that it is \emph{maximal}, meaning that its real locus has the largest
possible number of components, given by the geometric genus of the curve plus
$1$. In our case, it is $1+1=2$.

\begin{lem}\label{rem:real_locus}
The real locus of the spectral curve $\Ccal^k$ is the image by $\psi$ of $\RR/4K\ZZ+\{0,2iK'\}$; 
it thus has two components, and the spectral curve is maximal.
The connected component with ordinate $0$ is unbounded; the other one is bounded away from $0$ and infinity.
\end{lem}
\begin{proof}
Since the number of factors in the products defining $z(u)$ and $w(u)$ is even, and since $\sc(\overline{u})=\overline{\sc(u)}$, the map $\psi$
commutes with complex conjugation. As a consequence, 
the real locus of $\Ccal^k$ 
is the image by $\psi$ of the points 
of the torus $\TT(k)$ invariant by
complex conjugation: this is exactly $\RR/4K\ZZ+\{0,2iK'\}$. 
The connected component with ordinate $0$ is unbounded, since it contains the zeros and poles of
$z(u)$ and $w(u)$. On the other one, $z(u)$ and $w(u)$ are bounded away from $0$ and $\infty$.
\end{proof}


%

The parametrization $\psi$ also has consequences on the geometry of the amoeba $\A^k$. 

\begin{lem}\label{lem:tentacle_amoeba}
For every train-track $T$ of $\Gs_1$, the amoeba $\A^k$ has two tentacles, which
are symmetric with respect to the origin; their asymptote is orthogonal to the
vector of coordinates $(h_T, v_T)$ of the homology class $[T]$. Moreover,
every tentacle (counted with multiplicity) arises from a train-track $T$ of $\Gs_1$.
\end{lem}

\begin{proof}
  From the definition of the parametrization $\psi$, all the zeroes/poles of
  $z(u)$ and $w(u)$ correspond to parameters of the train-tracks.
  Let $T$ be a train-track. Choose an orientation, fixing the parameter
  $\alpha_T\in\RR/4K\ZZ$ and the sign of the homology $(h_T,v_T)$. When $u$ is
  close to $\alpha_T$, there are some non-zero constants $c_1$ and $c_2$ such
  that
  \begin{equation*}
    z(u) = c_1 (u-\alpha_T)^{-v_T}(1+o(1)),\qquad
    w(u) = c_2 (u-\alpha_T)^{h_T}(1+o(1)),
  \end{equation*}
  so that $\log\vert z(u)\vert $ and $\log\vert w(u)\vert $ go to $\pm\infty$ and 
  \begin{equation*}
    h_T \log\vert z(u)\vert + v_T \log\vert w(u)\vert  = h_T \log\vert c_1\vert +v_T\log\vert c_2\vert  +o(1),
  \end{equation*}
  which means exactly that for $u$ close to $\alpha_T$, the unbounded component
  of the boundary has an asymptote with a normal $(h_T, v_T)$.
\end{proof}

\begin{rem}\label{rem:newton_polygon}
  Using Lemma~\ref{lem:tentacle_amoeba} and the duality between the amoeba 
  and the Newton polygon mentioned in the beginning 
  of Section~\ref{sec:spectral_amoeba}, we know that the 
  Newton polygon of $P_{\Delta^{m(k)}}$ is the only convex polygon
  centered at the origin whose boundary consists of the lattice vectors
  representing the homology classes of all the oriented train-tracks of the graph
  $\Gs_1$, in cyclic order. In particular, every vector comes with its opposite,
  corresponding to the same train-track with reverse orientation. For example,
  the Newton polygon of Figure~\ref{fig:amoeba} (right) is obtained from the homology classes of the train-tracks pictured in 
  Figure~\ref{fig:iso_perio} (top right).

\end{rem}

\subsubsection{Spectral curve of the massive Laplacian and genus $1$ Harnack curves}

A curve defined as the zero set of a $2$-variables complex polynomial 
with real coefficients, is said to be \emph{Harnack} if it is maximal, and if
only one of its real connected components meets the coordinate axes (including
the line at infinity), and the intersections with different axes lie on disjoint
arcs of that component~\cite{Viro}.
We now prove that the spectral curve $\Ccal^k$ of the massive Laplacian
is Harnack and that every symmetric, genus $1$ Harnack curve arises in this way.

\begin{thm}
  \label{thm:harnack}
The spectral curve $\Ccal^k$ of the massive Laplacian $\Delta^{m(k)}$ is a Harnack curve.
\end{thm}

\begin{proof}
%
The curve $\Ccal^k$ is maximal, by Lemma~\ref{rem:real_locus}.
To prove that it is Harnack it suffices to check,
by Theorem 10 of~\cite{Brugalle},
  that when $u$ runs through $\RR/4K\ZZ$, $(z(u),w(u))$ visits the axes in the
  right order, namely that when $u$ increases, the slopes of the asymptotes of the
  tentacles of the amoeba $\A^k$ are also increasing in the counterclockwise order.
  But, by Lemma~\ref{lem:tentacle_amoeba},
  the slope of an asymptote at $u=\alpha$ is orthogonal to the homology class of
  a train-track with angle $\alpha$. Since the homology classes of the train-tracks
  and the angles associated to them are in the same cyclic order, the property
  is thus satisfied. 
\end{proof}

\begin{thm}\label{thm:Harnack2}
  Every genus $1$ Harnack curve with $(z,w)\leftrightarrow(z^{-1},w^{-1})$ symmetry 
  arises as the spectral curve of the characteristic polynomial of the massive
  Laplacian $\Delta^{m(k)}$ on some periodic isoradial graph for some $k\in(0,1)$.
\end{thm}

\begin{proof}
  Let $\Ccal$ be a Harnack curve with geometric genus $1$ and
  $(z,w)\leftrightarrow(z^{-1},w^{-1})$ symmetry.
  Since $\Ccal$ is a genus $1$ maximal real curve, it can be parametrized by a torus of
  pure imaginary modulus~\cite[p.~59]{Natanzon}. This torus, after maybe a
  dilation, is a $\TT(k)$, for
  some $k\in(0,1)$. Let $\psi$ be the birational map from $\TT(k)$ to $\Ccal$. The
  symmetry $(z,w)\leftrightarrow(z^{-1},w^{-1})$ preserves each of the two components of the
  real locus of $\Ccal$, with their orientation. It is thus conjugated by $\psi$ to
  a real translation $u\mapsto u+u_0$ on $\TT(k)$. But since it is a non-trivial
  involution, then $u_0$ is equal to $2K$, the horizontal half-period of the torus
  $\TT(k)$.

  Let us denote by $\alpha_1,\ldots,\alpha_{2\ell}$ the values of
  $u\in\RR/4K\ZZ$ corresponding to a pole or a zero of $z(u)$ or $w(u)$, ordered
  cyclically. For $j\in\{1,\dotsc,\ell\}$, denote by $a_j$ (resp.\ $b_j$) the order of
  $\alpha_j$ in $z(u)$ (resp.\ $w(u)$). Because of the symmetry of the curve,
  we have
  $
 \alpha_{j+\ell}=\alpha_{j},\, a_{j+\ell} = -a_j,b_{j+\ell}=-b_j.
  $
    Moreover, $\sum_{j=1}^{\ell}a_j$ and $\sum_{j=1}^{\ell} b_j$ are even.

  Knowing the zeros and poles of $z(u)$ is enough to reconstruct the whole
  function: $z(u)$ and
  $
    \prod_{j=1}^\ell \sc(u-\alpha_j)^{a_j}
    $
  are meromorphic functions on $\TT(k)$ and have the same zeros and poles, with the same
  multiplicities. Therefore they are equal up to a multiplicative constant.
  The constant is determined by the symmetry $z(u+2K)=z(u)^{-1}$ and the identity
  \eqref{id:scK} for $\sc$; we obtain
  \begin{equation*}
    z(u)=\prod_{i=1}^{\ell} \left(i\sqrt{k'}\sc(u-\alpha_j)\right)^{a_j}.
  \end{equation*}
  The same argument for $w(u)$ yields:
  \begin{equation*}
    w(u)=\prod_{i=1}^{\ell} \left(i\sqrt{k'}\sc(u-\alpha_j)\right)^{b_j}.
  \end{equation*}

  We now want to construct a periodic isoradial graph $\Gs$ (or equivalently an
  isoradial graph $\Gs_1$ on the torus), on which the spectral
  curve of the massive Laplacian is $\Ccal$. First we construct the graph of
  train-tracks ${\GR_1}^*$, as explained in Section~\ref{sec:traintrack_torus},
  by drawing on the torus for every
  $j\in\{1,\dotsc,\ell\}$ a self-avoiding cycle with homology class
  $(b_j,-a_j)$, such that the total number of intersections is minimal.
  The arrangement of the train-tracks is not unique, but because of
  $3$-dimensional consistency (Section~\ref{sec:integrability}), all of them
  should yield the same result.
  The graph ${\GR_1}^*$ determines the graph structure of $\Gs_1$ once we decide
  which is the primal and the dual graph. Now remains to determine the
  embedding, \emph{i.e.}, to attribute to every train-track a direction for the common
  sides of the rhombi on the train-track.

  Every value of $\alpha_j$ corresponds to a tentacle of the amoeba of $\Ccal$, with
  an asymptotic slope given by $(a_j,b_j)$. Since the curve $\Ccal$ is Harnack,
  the slopes of the tentacles are in the same cyclic order as the $\alpha_j$.
  This implies that if we associate to every oriented train-track $T_j$ with
  homology $(b_j,-a_j)$, the unit vector $e^{i\overline{\alpha_j}}$, we can place 
  a rhombus with the correct orientation at each intersection of two
  train-tracks, so that we get a proper isoradial embedding of the graph $\Gs$.
  According to Proposition~\ref{prop:geom_torus}, the spectral curve of the
  massive Laplacian on $\Gs$ for the value of $k$ chosen above, is also
  parametrized by $u\mapsto(z(u),w(u))$, and is therefore equal to $\Ccal$.
\end{proof}

\subsubsection{Consequence of the Harnack property on the amoeba}\label{sec:Amoeba}
 
The spectral curve $\Ccal^k$ has genus $1$ and is Harnack, so
the complement of the amoeba in $\mathbb{R}^2$ has a unique
bounded component, denoted by $D_{\A^k}$.
Since the characteristic polynomial is reciprocal, the amoeba $\A^k$
is invariant under central symmetry about the origin. Therefore,
the component $D_{\A^k}$ contains the origin, and corresponds to the integer point $(0,0)$ 
of the Newton polygon of the characteristic polynomial $P_{\Delta^{m(k)}}$, see Figure~\ref{fig:amoeba}.

The Harnack property also implies that the boundary of the amoeba
coincides with its real locus~\cite{Mikhalkin2}. Combining this with the explicit parametrization of the 
real locus of the spectral curve proved in Lemma~\ref{rem:real_locus} yields the following.

\begin{lem}\label{lem:boundary_amoeba}
The outer boundary of the amoeba is the image by $\Log\circ\psi$ of $\RR/4K\ZZ$.
The boundary of $D_{\A^k}$ is the image by $\Log\circ\,\psi$ of $\RR/4K\ZZ+\{0,2iK'\}$. 
\end{lem}

Since the spectral curve is Harnack, we know by~\cite{Mikhalkin2} that the area of the amoeba $\A^k$ is $\pi^2$
times the area of the Newton polygon of $P_{\Delta^{m(k)}}(z,w)$. It is thus independent of $k$ and 
only depends on the geometry of the isoradial graph. A quantity which does depend on the elliptic modulus $k$ is the 
area of the hole $D_{\A^k}$. We now prove the following.

\begin{prop}\label{prop:area_hole}
  As $k$ varies from $0$ to $1$, the area of $D_{\A^k}$ grows continuously from $0$ to
  $\infty$.
\end{prop}

\begin{proof}
According to Lemma~\ref{lem:boundary_amoeba}, the boundary of $D_{\A^k}$ is parametrized by
$(\log\vert z(u)\vert ,\log\vert w(u)\vert )$, for $u\in[0,4K]+2iK'$.

The area of $D_{\A^k}$ is computed by integrating the form $x\,\ud y$ along the
boundary of $D_{\A^k}$:

\begin{equation}
  \label{eq:area}
  \text{Area}(D_{\A^k}) = \int_{0}^{4K} \log \vert z(u+2iK')\vert 
  \frac{w'(u+2iK')}{w(u+2iK')}\ud u.
\end{equation}


Using the definition \eqref{eq:def_zwofu} of $z(u)$ and $w(u)$, and
the fact that $\sc(u+iK')=i\nd(u)$ \eqref{id:sciKp}, we have
\begin{equation*}
  \log\vert z(u+2iK')\vert = \sum_{S\ \text{train-track }\in\T} v_S
  \log\{\sqrt{k'}\nd(u_{\alpha_S})\}
\end{equation*}
and
\begin{equation*}
      \frac{w'(u+2iK')}{w(u+2iK')}
      = - \sum_{T\ \text{train-track }\in\T} h_T
      \frac{\nd'(u_{\alpha_T})}{\nd(u_{\alpha_T})}
      = - \sum_{T\ \text{train-track }\in\T} h_T k^2
      \frac{\sn \cdot \cn}{\dn}(u_{\alpha_T}).
\end{equation*}
Thus, Equation~\eqref{eq:area} can be rewritten as
\begin{equation*}
       \sum_{S,T\in\T} (-k^2 v_S h_T)
      \int_0^{4K}\log\{\sqrt{k'}\nd(u_{\alpha_S})\}
      \frac{\sn\cdot \cn}{\dn}(u_{\alpha_T}) \ud u.
\end{equation*}

First notice that terms in the sum for which $S=T$ do not contribute to the
sum, by antisymmetry under the change of variable $v=2\alpha_T-u$.

The contribution of the two terms corresponding to the same (unordered) pair of
train-tracks $\{S,T\}$ in the sum is:
%
\begin{equation*}
  k^2 (S\wedge T)
  \int_0^{4K}\log\{\sqrt{k'}\nd(u_{\alpha_T})\}
      \frac{\sn\cdot \cn}{\dn}(u_{\alpha_S}) \ud u,
\end{equation*}
where, by \eqref{equ:intersection}, $S\wedge T  = h_S v_T - v_S h_T$,
equals the number of intersections between $S$ and $T$, with a sign $+$
(resp.\ $-$) if $\alpha_T-\alpha_S\in (0,2K)$ (resp.\ in $(-2K,0)$).

Fix $\alpha$, $\beta$ and $\theta=\frac{\beta-\alpha}{2}$, and define
\begin{equation*}
      I(k)=\int_0^{4K}\log\{\sqrt{k'}\nd(u_{\alpha})\}\frac{\sn\cdot \cn}{\dn}(u_{\beta}) \text{d}u
      =
\frac{2K}{\pi}\int_{0}^{\pi}
      \log \Bigl\{\frac{\dn(\frac{K}{\pi}(-v+2\overline{\theta}))}{\dn(\frac{K}{\pi}(v+2\overline{\theta}))}\Bigr\}
      \frac{\sn\cdot \cn}{\dn}\Bigl(\frac{K}{\pi}v\Bigr) \ud v.
\end{equation*}
To prove that the area \eqref{eq:area} is an increasing function of
$k\in(0,1)$, it is sufficient to prove that if $\bar\theta\in(0,\frac{\pi}{2})$, the
integral $I(k)$
is an increasing function of $k\in(0,1)$.

Using the addition formula for the $\dn$ function \cite[16.17.3]{AS} one can write
\begin{equation*}
     \frac{\dn(\frac{K}{\pi}(-u+2\overline{\theta}))}{\dn(\frac{K}{\pi}(u+2\overline{\theta}))}
     =
     \frac
     {1+k^2\frac{\sn\cdot \cn}{\dn}(\frac{2K\overline{\theta}}{\pi})\frac{\sn\cdot \cn}{\dn}(\frac{Ku}{\pi})}
     {1-k^2\frac{\sn\cdot \cn}{\dn}(\frac{2K\overline{\theta}}{\pi})\frac{\sn\cdot \cn}{\dn}(\frac{Ku}{\pi})}.
\end{equation*}

The function $X\mapsto \log\bigl\{\frac{1+X}{1-X}\bigr\}$ is increasing.
Moreover, for a fixed $v\in[0,\pi]$, the quantity $\frac{\sn\cdot
\cn}{\dn}(\frac{K}{\pi}u\vert k)$ is non-negative and increasing in
$k\in(0,1)$, as can be checked by Landen transformation (see \eqref{eq:alt} in Appendix~\ref{app:elliptic}) and \cite[Figure
2.1]{La89}. Thus the non-negative functions $k\mapsto K(k)$,  
$k\mapsto \frac{\sn\cn}{\dn}(\frac{K u}{\pi})$ and 
$k\mapsto k^2 \frac{\sn\cdot\cn}{\dn}(\frac{K u}{\pi})
  \frac{\sn \cdot \cn}{\dn}(\frac{2 K \overline{\theta}}{\pi})$, are also
  increasing functions of $k$. Thus, so is $I$.

The limits of the area when $k\to 0$ and $k\to 1$ are obtained by noticing
that $I(0)=0$ since the integrand is zero (we recall that $\dn(\cdot\vert
0)=1$). Moreover, we have just proved that $I(k)/K(k)$ is positive and increasing on
$(0,1)$. Since $K(k)$ goes to infinity, when $k\to 1$, so does $I(k)$.
\end{proof}

\subsection{Further properties}
\label{subsec:further}

Since the spectral curve $\Ccal^k$ has geometric genus $1$, the space of
holomorphic differential $1$-forms on $\Ccal^k$ has dimension $1$. 
It turns out that we can explicitly compute one of these forms from the matrix $\Delta^{m(k)}(z,w)$. 
Before doing this, we need two lemmas about the dimension of the kernel of the matrix
$\Delta^{m(k)}(z,w)$ and the structure of the adjugate matrix, denoted $Q^k(z,w)$.

\begin{lem}\label{lem:kernel}
For every $(z,w)\in\Ccal^k$, the dimension of the kernel of the matrix $\Delta^{m(k)}(z,w)$ is:
\begin{equation*}
\mathrm{dim} \big[\ker \Delta^{m(k)}(z(u),w(u))\big]
\begin{cases}
=1&\text{ if $(z,w)$ is a simple point of $\Ccal^k$,}\\
\geq 2&\text{ if $(z,w)$ is a solitary node of $\Ccal^k$}.
\end{cases}
\end{equation*}
Moreover, when $(z,w)$ is a simple point, every $(z,w)$-quasiperiodic massive harmonic function is
  proportional to $\expo_{(\cdot,x_0)}(u)$, where $u\in\TT(k)$ is such that
  $(z(u),w(u))=(z,w)$.
\end{lem}
\begin{proof}
  Let $u\in \TT(k)$.
  We have seen in the proof of Proposition~\ref{prop:geom_torus} that the function $\expo_{(\cdot,x_0)}(u)$ is 
  a non-zero $(z(u),w(u))$-quasiperiodic and massive harmonic function.
  Therefore, it is in the kernel of $\Delta^m(z(u),w(u))$.

Suppose that $(z,w)$ is a simple point of $\Ccal$.
  The fact that the kernel of $P_{\Delta^m}(z,w)$ has dimension $1$ for a simple
  point $(z,w)$ follows from \cite{CookThomas}. Let us quickly recall the argument here.
The following identity holds for every $z'$ and $w'$):
  \begin{equation}
    \label{eq:comatrix_det}
    Q(z',w')\Delta^m(z',w')  = P_{\Delta^m}(z',w')\cdot \text{Id}.
  \end{equation}
Since the point $(z,w)$ is simple,
  $(\frac{\partial P}{\partial z}(z,w),\frac{\partial P}{\partial w}(z,w))\neq(0,0)$.
  Suppose we have $\frac{\partial P}{\partial z}(z,w)\neq 0$.
  Differentiating~\eqref{eq:comatrix_det} with respect to $z'$, and evaluating
  at $(z',w')=(z,w)$, we get:
  \begin{equation*}
    \frac{\partial Q(z,w)}{\partial z} \Delta^m(z,w) + 
    Q(z,w) \frac{\partial \Delta^{m}(z,w)}{\partial z} =
    \frac{\partial P_{\Delta^m}(z,w)}{\partial z}  \text{Id}.
  \end{equation*}
  If $\ker\Delta^m(z,w)$ had dimension strictly greater than $1$, the matrix $Q(z,w)$
  would be identically zero. But $\frac{\partial Q(z,w)}{\partial z} \Delta^m(z,w)$ cannot
  be equal to a non-zero multiple of the identity, because $(z,w)$ is on the
  curve $\Ccal^k$ and thus $\Delta^{m}(z,w)$ is non-invertible.
  Therefore $\mathrm{dim} \ker\Delta^m(z,w)=1$, and if $u\in \TT(k)$ is such
  that $(z(u), w(u))=(z,w)$, then by the remark above, the function $\expo_{(\cdot,x_0)}(u)$
  spans the kernel of $\Delta^m(z,w)$.
  
  
  Suppose now that $(z,w)$ is a solitary node of the curve. This point has
  two\footnote{Since the spectral curve is Harnack, $u$ and 
  $\overline{u}$ are the only two points of $\TT(k)$ giving the value $(z,w)$.}
  distinct conjugated preimages $u\neq\overline{u}$ by $\psi$ on $\TT(k)$.
  The two functions $\exp_{(\cdot,x_0)}(u)$ and
  $\exp_{(\cdot,x_0)}(\overline{u})$ are in the kernel of $\Delta^m(z,w)$, but
  are not proportional.
  The kernel of $\Delta^m(z,w)$ is thus at least two-dimensional. 
  \end{proof}

\begin{lem}
  \label{lem:def_g}
  There exists a meromorphic function $g^k$ on $\TT(k)$
  such that:
  \begin{equation*}
    \forall\, u\in\TT(k),\ 
    \forall\, x,y\text{ vertices of }\Gs_1,\quad
    Q_{x,y}^k(z(u),w(u)) = g^k(u) \expo_{(x,y)}(u).
  \end{equation*}
  In particular, $g^k(u)$ is the diagonal coefficient $Q_{x,x}^k(u)$ for every
  vertex $x$ of $\Gs_1$. When $u$ is such that $(z(u),w(u))$ is a solitary node, we have $g^k(u)=0$. 
\end{lem}
\begin{proof}Let $u\in\TT(k)$. 
  Suppose that $(z(u),w(u))$ is not a solitary node of $\Ccal^k$.
  Then, by Lemma~\ref{lem:kernel}, $\mathrm{dim}\ker\Delta^m(z(u),w(u))=1$, implying that $Q(z(u),w(u))$ has
  rank~$1$, and can be written
  \begin{equation*}
    Q(z(u),w(u))= V\cdot W^T,
  \end{equation*}
  with $V\in \ker\Delta^m(z(u),w(u))$ and $W\in
  \coker\Delta^m(z(u),w(u))=\ker\Delta(z(u)^{-1},w(u)^{-1})=\ker
  \Delta^m(z(u+2K),w(u+2K))$.
  So $V$ and $W$ are (non-zero) multiples of $\expo_{(\cdot,x_0)}(u)$ and 
  $\expo_{(\cdot,x_0)}(u+2K)=\expo_{(x_0,\cdot)}(u)$ respectively.
  Therefore there exists a non-zero coefficient $g(u)$ such that for any vertices 
  $x$ and $y$ of $\Gs_1$,
  \begin{equation*}
    Q_{x,y}(z(u),w(u)) = V_x W_y
    = g(u) \expo_{(x,x_0)}(u) \expo_{(x_0,y)}(u)
    = g(u) \expo_{(x,y)}(u).
  \end{equation*}
  For $x=y$, we get $g(u)=Q_{x,x}(z(u), w(u))$, which is
  meromorphic as the composition of a polynomial with meromorphic functions.
  In particular, $Q_{x,x}(z(u),w(u))$ does not depend on $x$.
  If $u$ corresponds to a solitary node, $Q(z(u),w(u))$ vanishes because the kernel of
  $\mathrm{dim}\ker\Delta^m(z,w)\geq 2$. Therefore, $g$
  extends analytically to $u$, and $g(u)=0$.
\end{proof}

Since the spectral curve has geometric genus $1$, the space of
holomorphic differential $1$-forms on $\Ccal^k$ has dimension $1$. 
The next proposition states that we can explicitly compute one of these forms
using the matrices $\Delta^{m(k)}(z,w)$ and $Q^k(z,w)$.
Any other holomorphic 1-form is a multiple of this one.

\begin{prop}\label{lem:holomorphic_form}
The differential form
$\displaystyle\frac{Q_{x,x}^k(z,w)}{\frac{\partial P_{\Delta^{m(k)}}}{\partial w}(z,w)w z}\ud z$
is a holomorphic $1$-form on $\Ccal^k$.
\end{prop}

\begin{proof}
  According to classical theory of algebraic curves~\cite{ArbaCornGrifHarr}, all holomorphic
  differential forms on $\Ccal^k$ are of the form
  \begin{equation*}
   \frac{R(z,w)}{\frac{\partial P_{\Delta^{m}}}{\partial w}(z,w)}\ud z ,
  \end{equation*}
  with $R$ a polynomial of degree not greater than $\deg P_{\Delta^m} -3$ and vanishing on
  solitary nodes of the curve $\Ccal^k$.

  Let us prove that the polynomial $R(z,w)=Q_{x,x}(z,w)/zw$ satisfies these two
  properties. The fact that $R$ vanishes on nodes is a consequence of
  Lemma~\ref{lem:def_g}.

  To control the degree of $Q_{x,x}$, we apply the argument of the proof of Lemma~\ref{lem:inclu_Newton_PDelta}
  to $Q_{x,x}$, for paths $(\gamma_x,\gamma)$ as in
  Corollary~\ref{cor:control_degree_P}. Recall that $Q_{x,x}(z,w)$ is
  computed as the determinant of the matrix $\Delta^m(z,w)$ from which the row and the column
  indexed by $x$ are removed, and recall that any choice for the vertex $x$ yields the same
  result. If $x$ is a vertex of $\Gs_1$ visited by $\gamma$, then the degree
  counting argument in the variable $z$ (along $\gamma$) in the expansion of the
  determinant of the minor of $\Delta^m(z,w)$ shows that the maximal degree of
  $Q_{x,x}(z,w)$ is strictly less than that of $P_{\Delta^m}(z,w)$. When dividing
  by $zw$, we get a polynomial of degree not higher than $\deg P_{\Delta^m}-3$.
\end{proof}

\subsection{Green function on periodic isoradial graphs}
\label{sec:perio_green}

When the graph $\Gs$ is periodic, the massive Green function $G^m$ can be expressed 
as a double integral involving the Fourier transform. Indeed, the matrix
$\Delta^m(z,w)$ is invertible for
generic values of $z$ and $w$, and the Green function is obtained as Fourier
coefficients of ${\Delta^m(z,w)}^{-1}$: if $x$ and $y$ are two vertices of $\Gs_1$
and $(m,n)\in\mathbb{Z}^2$,

\begin{align}
  \label{eq:periodic_Green}
  G^m(x+(m,n),y)&= \iint_{\vert z\vert =\vert w\vert =1} z^{-m}w^{-n}  ({\Delta^m(z,w)}^{-1})_{x,y}\,
    \frac{\ud z}{2i\pi z}
    \frac{\ud w}{2i\pi w}\nonumber\\
    &= \iint_{\vert z\vert =\vert w\vert =1} z^{-m} w^{-n} \frac{{Q_{x,y}(z,w)}}{P_{\Delta^m}(z,w)}
    \frac{\ud z}{2i\pi z}
    \frac{\ud w}{2i\pi w}.
\end{align}


In Section~\ref{sec:recovLocal}, we give an alternative proof of the local
formula \eqref{eq:green} of Theorem~\ref{thm:expression_Green}
for the massive Green function $G^m(x,y)$, starting
from the double integral formula of Equation~\eqref{eq:periodic_Green}. 

In Section~\ref{sec:recovAsympt}, we explain how to recover asymptotics of the Green function of Theorem~\ref{thm:asymp_Green}
from the double integral formula of Equation~\eqref{eq:periodic_Green}, using the approach of \cite{PeWi13}. This yields a
geometric interpretation of the exponential rate of decay in terms of the amoeba~$\A^k$.

  
\subsubsection{Recovering the local formula for the massive Green function}\label{sec:recovLocal}

We now give an alternative proof of the local formula \eqref{eq:green} for the
massive Green function, starting from the double integral formula 
\eqref{eq:periodic_Green}. We can assume that $x+(m,n)$ is ``below'' the vertex
$y$, \emph{i.e.}, that $n \leq -1$, after maybe having to change the boundary of
the fundamental domain, and exchanging the axes and their directions.

We first transform the integral~\eqref{eq:periodic_Green} defining $G^m(x+(m,n),y)$, by computing at
fixed $z$ the integral over $w$ by residues. For a generic value of $z$ on the
unit circle $\mathbb S^1$, the function $P_{\Delta^m}(z,\cdot)$ has $2d_x$ distinct
non-zero roots, which all have modulus different from $1$, because $(0,0)$ is not
in the amoeba of $\Ccal^k$. Since $P_{\Delta^m}$ is reciprocal, if $w(z)$ is a root, then $\overline{w(z)}^{-1}$ is also a root, 
meaning that $d_x$ of them are inside the unit disk: $w_1(z),\dots,w_{d_x}(z)$,
and $d_x$ of them outside: $w_{d_x+1}(z),\dotsc,w_{2d_x}(z)$.
Since $n\leq -1$, there is no pole at $0$, and by application of
the residue theorem,
\begin{equation*}
  \int_{\vert w\vert =1} \frac{Q_{x,y}(z,w)}{P_{\Delta^m}(z,w)}w^{-n-1}\frac{\ud w}{2i\pi} =
  \sum_{i=1}^{d_x} \frac{Q_{x,y}(z,w_{i}(z))}{\frac{\partial P_{\Delta^m}(z,w)}{\partial
  w}(z,w_i(z))} w_i(z)^{-n-1}.
\end{equation*}

In the remaining integral over $z$ of Equation~\eqref{eq:periodic_Green}, we
perform the change of variable from $z$ to
$u\in\TT(k)$. There are on the spectral curve $\Ccal^k$ two disjoint simple
paths on which the first coordinate $z$ is in the unit circle. They project onto
the amoeba to two vertical segments obtained as the intersection of the
amoeba and the vertical axis $x=0$, one of the two segments is below the horizontal axis, the other one is above.
The preimage by $\psi:u\mapsto(z(u),w(u))$ of those two segments are two ``vertical'' loops on $\TT(k)$, 
denoted by $\Gamma$ and $\Gamma'$, respectively.
The loops $\Gamma$ and $\Gamma'$ are assumed to be oriented in such a way that when $u$ moves in the
positive direction, $z(u)$ winds
counterclockwise around the unit circle.
The map $u\in\Gamma\mapsto z(u)\in \mathbb S^1$ has degree $d_x$:
along $\Gamma$, there are exactly $d_x$ values of
$u$ such that $z(u)=z$, the corresponding value of $w(u)$ being equal to one of
the $w_i(z)$, $i\in\{1,\dotsc,d_x\}$. We can therefore rewrite
\begin{equation*}
  \int_{\vert z\vert =1} \sum_{i=1}^{d_x} f(z,w_i(z))\ud z =
  \oint_{u\in\Gamma} f(z(u),w(u))z'(u) \ud u,
\end{equation*}
for any measurable function $f$. In particular, for
$f(z,w)=z^{-m-1}w^{-n-1}\frac{Q_{x,y}(z,w)}{\frac{\partial P_{\Delta^m}}{\partial
w}(z,w)}$, one gets the following expression for the massive Green function:
\begin{equation*}
  G^m(x+(m,n), y) = \oint_{u\in\Gamma} z(u)^{-m}w(u)^{-n}
  \frac{Q_{x,y}(z(u),w(u))}{z(u) w(u) \frac{\partial P_{\Delta^m}}{\partial w}(z(u),w(u))}
  z'(u)\frac{\ud u}{2i\pi}.
\end{equation*}
But according to Lemma~\ref{lem:def_g} and Equation~\eqref{eq:expo_perio},
\begin{equation*}
  z(u)^{-m}w(u)^{-n}Q_{x,y}(z(u),w(u)) = \expo_{(x+(m,n),y)}(u) g(u),
\end{equation*}
and the differential form
\begin{equation*}
  \frac{g(u)}{z(u)w(u) \frac{\partial P_{\Delta^m}}{\partial w}(z(u),w(u)} z'(u)\ud u
\end{equation*}
is the pullback by the biholomorphic map $\psi$ of the holomorphic $1$-form defined
in Lemma~\ref{lem:holomorphic_form}. Therefore it is
a holomorphic form on $\TT(k)$, and as such,
  equal to $\ud u$, up to a
multiplicative constant $A$ to be determined by other means:
\begin{equation*}
 G^m(x+(m,n), y) =
 A\times\oint_{\Gamma} \expo_{(x+(m,n),y)}(u)\ud u.
\end{equation*}
One then checks that the position of the contour $\Gamma$ with respect to the
poles of the exponential function is indeed the one described in
Theorem~\ref{thm:expression_Green}. In order to determine the numerical value of the constant $A$, one needs to
compute the Laplacian of the Green function $G^m(\cdot,y)$ at the vertex
$y$, as in the proof of Theorem~\ref{thm:expression_Green}.

\subsubsection{Recovering asymptotics of the massive Green function}
\label{sec:recovAsympt}

Let $x$ and $y$ be two vertices of the fundamental domain $\Gs_1$, and let
$(m,n)\in\ZZ^2$. We now explain how to recover the asymptotics formula of Theorem~\ref{thm:asymp_Green}.
In the periodic case, we can let the vertex $x+(m,n)$ tend to
  infinity with an asymptotic direction:
for $\mathbf{r}=(m,n)\in\ZZ^2\setminus\{0,0\}$, denote by $\widehat{\mathbf{r}}$ the
unit vector in the direction of $\mathbf{r}$, and $\vert \mathbf{r}\vert $ its
norm. The asymptotic regime we consider corresponds to $|\mathbf{r}|\to\infty$
and $\widehat{\mathbf{r}}\to\widehat{\mathbf{r}}^*$, where $\widehat{\mathbf{r}}^*$ is a
fixed direction.

The double integral formula~\eqref{eq:periodic_Green}
is the coefficient $a_{\mathbf{r}}=a_{m,n}$ of $z^m w^n$ in the
(multivariate) series expansion of the
rational fraction $\frac{Q_{x,y}}{P_{\Delta^m}}$ in a neighborhood of
$\vert z\vert =\vert w\vert =1$, and the domain of convergence of this expansion is exactly the set
\begin{equation*}
  \Log^{-1}(D_{\A^k})=\{(z,w):  (\log\vert z\vert ,\log\vert w\vert )\in D_{\A^k}\}.
\end{equation*}

In particular, the general term $a_{m,n} z^{m} w^{n}$ should go to zero for
$\Log(z,w)\in D_{\A^k}$, and should be unbounded if $\Log(z,w)$ is in the interior
of the amoeba.

%

Define the exponential rate of
the series coefficients $(a_{\mathbf{r}})$ in the direction
$\widehat{\mathbf{r}}_*$ as in~\cite{PeWi13}:
\begin{equation*}
  \overline{\beta}(\widehat{\mathbf{r}}_*) =
  \inf_{\mathcal{N}}\limsup_{
    \substack{\mathbf{r}\to\infty,\\
    \widehat{\mathbf{r}}\in \mathcal{N}}
  } \vert \mathbf{r}\vert ^{-1} \log \vert a_{\mathbf{r}}\vert ,
\end{equation*}
where $\mathcal{N}$ varies over a system of open neighborhoods of
$\widehat{\mathbf{r}}_*$ whose intersection is the singleton~$\{\widehat{\mathbf{r}}_*\}$.
Then, for every $\widehat{\mathbf{r}}$,
$
  \overline\beta(\widehat{\mathbf{r}})=
  \inf\{-  \widehat{\mathbf{r}}\cdot\mathbf{s}: \mathbf{s}\in D_{\A^k}\}$,
see~\cite[Chapter~8]{PeWi13}. The compact oval $D_{\A^k}$ is
strictly convex, because the spectral curve is Harnack. So the infimum is obtained at a single
point
on the boundary of the amoeba, corresponding to a unique value of the parameter
$u_0+2iK'\in 2iK'+\RR/4\ZZ$. This gives, in the periodic case, a geometric
interpretation of the parameter $u_0$ in Theorem~\ref{thm:asymp_Green} in terms
of the spectral curve.

Using a little further the formalism developed by Pemantle and
Wilson~\cite[Chapter~9]{PeWi13}, 
one can recover in the periodic case
the precise asymptotics we obtain in Theorem~\ref{thm:asymp_Green} in the
general quasicrystalline setting, with the exact prefactor.

\section{Random rooted spanning forests on isoradial graphs}
\label{sec:statmech}

In this section we study random rooted spanning forests on isoradial graphs. 
In Section~\ref{subsec:rsf_ust}
we define the statistical mechanics model of rooted spanning forests. Then, in Section~\ref{sec:infinite_vol_meas} we
prove an explicit, local expression for 
an infinite volume Boltzmann measure involving the Green function of Theorem~\ref{thm:expression_Green}. In Section \ref{sec:free_energy} we show an explicit, \emph{local} expression for the free energy of the model; 
we also show a second order phase transition at $k=0$ in the rooted spanning forest model. At $k=0$, one 
 recovers the Laplacian considered in~\cite{Kenyon3}, and we thus provide a proof that the corresponding spanning tree model is critical. 
In Section~\ref{sec:Zinvariance} we prove that our one-parameter family of massive Laplacian defines a one-parameter family of 
 $Z$-invariant spanning forest models.

\subsection{Rooted spanning forest model and related spanning trees}
\label{subsec:rsf_ust}


Let $\Gs=(\Vs,\Es)$ be a (not necessarily isoradial) graph.
A \emph{tree} of $\Gs$ is a connected subgraph of $\Gs$ containing no cycle. A
\emph{rooted tree} is a tree with a distinguished vertex, known as the
\emph{root}. The root of a generic tree $\Ts$ is denoted $x_\Ts$. 
A \emph{spanning tree} is a tree spanning all vertices of the graph. 

A \emph{rooted spanning forest} of $\Gs$ is a subgraph of $\Gs$, spanning all
vertices of the graph, such that every connected component is a rooted tree.
Let $\F(\Gs)$ denote the set of rooted spanning forests of the graph $\Gs$.

Assume that edges of the graph $\Gs$ are assigned positive \emph{conductances} $(\rho(e))_{e\in\Es}$ 
and that vertices are assigned positive \emph{masses} $(m^2(x))_{x\in\Vs}$.
This is equivalent to defining a massive Laplacian
$\Delta^m$ on $\Gs$, through Equation~\eqref{equ:operator_general}.

Suppose now that $\Gs$ is a finite. Then we can define a model of statistical mechanics, by constructing
the \emph{rooted spanning forest Boltzmann probability measure}, denoted
$\PPforest$, defined by:
\begin{equation*}
\forall\,\Fs\in \F(\Gs),\quad \PPforest(\Fs)=\frac{1}{\Zforest(\Gs,\rho,m)} \prod_{\Ts\in\Fs}\Bigl(m^2(x_\Ts)\prod_{e\in\Ts}\rho(e)
\Bigr),
\end{equation*}
where the normalizing constant
$\Zforest(\Gs,\rho,m)=\sum_{\Fs\in\F(\Gs)}\prod_{\Ts\in\Fs}\Bigl(m^2(x_\Ts)\prod_{e\in\Ts}\rho(e)
\Bigr)$ is the \emph{rooted spanning forest partition function}.

There is a direct and useful bijection between
weighted rooted spanning forests of $\Gs$ and weighted spanning trees of the
graph $\Gsr$, obtained from $\Gs$ by adding a vertex $\rs$, and by joining 
every vertex of $\Gs$ to $\rs$.
Given a spanning tree of $\Gsr$, removing every edge connecting a vertex of $\Gs$ to the vertex
$\rs$, and replacing it by a root, yields a rooted spanning forest of $\Gs$.
This bijection is weight preserving if edges of $\Gsr$ have conductances
$\rho^m$
defined by:
\begin{equation}\label{equ:defrhom}
\rho^m(e)=
\begin{cases}
\rho(e)&\text{ if $e$ is an edge of the graph $\Gs$,}\\
m^2(x)&\text{ if $e=xr$ is an edge of $\Gsr\setminus\Gs$ connecting the vertex $x$ to $\rs$.}
\end{cases}
\end{equation}

Let $\T(\Gsr)$ denote the set of spanning trees of the graph $\Gsr$. 
The \emph{spanning tree Boltzmann probability measure} on $\Gsr$, denoted $\PPtree$, is defined by:
\begin{equation*}
\forall\,\Ts\in \T(\Gsr),\quad \PPtree(\Ts)=\frac{1}{\Ztree(\Gsr,\rho^m)} \prod_{e\in\Ts}\rho^m(e),
\end{equation*}
where $\Ztree(\Gsr,\rho^m) = \sum_{\Ts\in\T(\Gsr)}\prod_{e\in\Ts}\rho^m(e)$
is the \emph{spanning tree partition function}.
$\PPtree$ is the image
measure of $\PPforest$ by the bijection above.

From the above bijection, we know that
$\Zforest(\Gs,\rho,m)=\Ztree(\Gsr,\rho^m)$, and also that $\PPtree$ and
$\PPforest$ are transported one into the other by the bijection.




By Kirchhoff's matrix-tree theorem~\cite{Kirchhoff}, there is an explicit
expression of the spanning tree partition function
as a determinant, and by the
work of Burton and Pemantle~\cite{BurtonPemantle} (see also~\cite{BLPS}),
under $\PPtree$ the edges of the random spanning tree on $\Gsr$ form a determinantal
process. Restating these results from the point of view of spanning forests on
$\Gs$ yields:

\begin{thm}[Matrix-Forest Theorem~\cite{Kirchhoff}]
  \label{thm:matrix_forest}
The rooted spanning forest partition function of the graph $\Gs$ is equal to:
$$
\Zforest(\Gs,\rho,m)=\det(\Delta^m).
$$
\end{thm}

With the same notation as before, we let $G^m$ be the massive Green
function on $\Gs$, \emph{i.e.}, the inverse of the massive
Laplacian~$\Delta^m$. Fix an arbitrary orientation of the edges of $\Gs$, so
that every edge $e=(e_-,e_+)$ is now oriented from one of its ends $e_-$ to the
other one $e_+$.

\begin{thm}[Transfer Impedance Theorem~\cite{BurtonPemantle}]
  \label{thm:transimp_forest}
  For any distinct edges $e_1,\dotsc,e_j$ and vertices
  $x_1,\dotsc,x_k$ of $\Gs$, the probability that these edges belong to a
  random rooted spanning forest and that these vertices are roots, is:
  \begin{equation*}
    \PPforest(\{e_1,\dotsc,e_k,x_1,\dotsc,x_j\})=\det
    \left(
    \begin{array}{c|c}
      \bigl(\Hs^k(e_i,e_\ell)\bigr)_{1\leq i,\ell\leq j} & 
      \bigl(\Hs^k(e_i,x_\ell)\bigr)_{1\leq i\leq j,1\leq \ell \leq k} \\
      \hline
      \bigl(\Hs^k(x_i,e_\ell)\bigr)_{1\leq i\leq k,1\leq \ell \leq j}  & 
      \bigl(\Hs^k(x_i,x_\ell)\bigr)_{1\leq i,\ell\leq k}
    \end{array}
%
      \right),
  \end{equation*}
  where
  \begin{align*}
    \Hs(e,e') & = 
    \rho(e')(G^{m}(e_-,e'_-)-G^m(e_+,e'_-)-G^m(e_-,e'_+)+G^m(e_+,e'_+)),\\
    \Hs(e,x) &= m^2(x) (G^m(e_-,x)-G^m(e_+,x)), \\
    \Hs(x,e) &= \rho(e')(G^m(x,e_-)-G^m(x,e_+),\\ 
    \Hs(x,x') &= m^2(x')G^m(x,x').
  \end{align*}

\end{thm}

The quantity $\Hs(e,e')$ is the \emph{transfer impedance} through $e'$,
with a source at the end points of $e$. We extend the name \emph{transfer impedance} to the whole of $\Hs$, even
when arguments are possibly not edges, but vertices.

The derivation of Theorems~\ref{thm:matrix_forest} and \ref{thm:transimp_forest} from
their classical versions is detailed in Appendix~\ref{sec:app_derivation_classical}.

\subsection{Infinite volume measure}\label{sec:infinite_vol_meas}

From now on, suppose that $\Gs$ is an infinite isoradial graph, whose faces are covering
the whole plane, with conductances $\rho$ and masses $m^2$ of Equations~\eqref{equ:def_conductances} and \eqref{eq:def_mass},
 for some $k\in(0,1)$. In the next theorem, we prove that the natural infinite volume Gibbs measure 
 on the set $\F(\Gs)$ of all rooted spanning forests of $\Gs$
 is expressed using the 
 impedance transfer matrix $\Hs^k$ involving the Green function $G^{m(k)}$ on $\Gs$ of Theorem~\ref{thm:expression_Green}.  
\begin{thm}
\label{thm:infinite_vol_meas}
  Let $k\in(0,1)$. There exists a unique measure $\PPforest^k$ on rooted spanning
  forests
  of $\Gs$ such that for any distinct edges
  $e_1,\dotsc,e_j$, and any distinct vertices $x_1,\dotsc,x_k$ of $\Gs$:
  \begin{equation*}
    \PPforest^k(\{e_1,\dotsc,e_j,x_1,\dotsc,x_k\})=\det
    \left(
    \begin{array}{c|c}
      \bigl(\Hs^k(e_i,e_\ell)\bigr) & 
      \bigl(\Hs^k(e_i,x_\ell)\bigr) \\
      \hline
      \bigl(\Hs^k(x_i,e_\ell)\bigr) & 
      \bigl(\Hs^k(x_i,x_\ell)\bigr)
    \end{array}
    \right)
  \end{equation*}
  where $\Hs^k$ is the transfer impedance on $\Gs$.

  The measure $\PPforest^k$ is the weak limit of the sequence $(\PPforest^{k,(n)})$ on
  rooted spanning forests of any exhaustion $(\Gs_n)_{n\geq 1}$ of $\Gs$ by finite graphs. Under
  $\PPforest^k$, the connected components of the random rooted spanning forests
  are finite almost surely.
\end{thm}

\begin{proof}
  Let $(\Gs_n)_{n\geq 1}$ be an
  exhaustion of $\Gs$ by finite graphs.
  By Theorem~\ref{thm:transimp_forest}, the determinantal
  process on edges of $\Gs_n$ with kernel given by the transfer impedance matrix
  on $\Gs_n$ is a probability measure on rooted spanning forests of $\Gs_n$.
  Moreover, by
  Lemma~\ref{lem:convergence_Green} of Appendix~\ref{sec:rw}, the sequence of Green functions on
  $\Gs_n$ converges pointwise to the Green function $G^m$ of $\Gs$. Therefore
  $\Hs^k$ is the limit of the sequence of transfer impedance matrices on $\Gs_n$.
  
  The convergence of the kernel of a determinantal process implies the convergence of the
  finite dimensional laws, which are consistent, as limits of probability
  measures. By Kolmogorov's extension theorem, there exists a probability
  measure $\PPforest^k$ on the set of edges $\Gs$, which has those finite dimensional
  marginals. Moreover, this measure is unique, since $\Gs$ has countably
  many edges.
  
  The random spanning forest on $\Gs_n$ can be sampled by Wilson's algorithm, by creating the branches 
  from the loop erasure of killed random walks, with transition probabilities naturally 
  defined from the conductances and masses, see Section~\ref{sec:app_derivation_classical}.
  Since $(\Gs_n)$ is an exhaustion of $\Gs$, we can take
  the limit in the loop erasure procedure and also sample the random configuration
  from the Gibbs measure on $\Gs$ by Wilson's algorithm on $\Gs$ with the killed
  random walk $(X_j)_{j\geq 0}$ defined in Section~\ref{sec:rw}, in the same
  manner as it is done in \cite[Theorem~5.1]{BLPS} to construct the
  wired uniform spanning forests on infinite graphs.

  We now show that the support is the set of rooted spanning forests with
  finite size components.
  From the convergence of finite dimensional marginals, it is clear that the limiting objects
  are rooted spanning forests.
  But what could happen is that as $n$ goes to infinity,
  some tree components on $\Gs_n$ grow infinite, and the root of these components 
  either stay at finite distance, or are sent to infinity (and thus
  disappear). To prove the statement about the support of $\PPforest^k$, one has
  to rule out the presence with positive probability of an infinite component in
  $\Gs$, with or without a root.

  Fix a vertex $x_0$. For every $\ell\geq 1$, define $S_\ell$ to be the set of
  vertices of $\Gs$ at distance $2\ell$ from $x_0$.
  If there is in the random rooted spanning forest an infinite component $T$, then
  $T$ has to intersect infinitely many $S_\ell$ (in fact all except maybe a
  finite number of them). The root of $T$ is either at infinity, or at finite
  distance from $x_0$. There is thus an infinite number of $\ell$ for which
  there exists a vertex $x_\ell\in S_\ell\cap T\subset S_\ell$ at distance at least $\ell$
  of the root of $T$.
  However, from Wilson's algorithm, the path from $x_\ell$ to its root is the loop
  erasure of the killed random walk $(X_j)_{j\geq 0}$ starting from $x_\ell$. The distance to
  the root is thus not larger than the length of the trajectory of the random
  walk starting from $x_\ell$ before being absorbed. Since the random walk has a
  probability of being absorbed at each vertex which is uniformly bounded from
  below by a positive quantity, the length is dominated by a geometric variable.
  The probability that it is greater than $\ell$ is thus exponential small in $\ell$.
  Since there are $O(\ell)$ vertices on $S_\ell$, by Borel-Cantelli's lemma, we see
  that the probability that the infinite sequence $(x_\ell)$ exists is zero. In
  other words, with probability 1, there is no infinite component.
  \end{proof}


\begin{rem}
By Remark~\ref{rem:critical_Green}, when $k$ goes to $0$, the impedance transfer
matrix $\Hs^k$ converges to the impedance transfer matrix defined from Kenyon's
critical Green function~\cite{Kenyon3}, which is the
kernel of the determinantal process on edges corresponding to the spanning tree
measure $\PPtree=\PPforest^0$ with conductances $(\tan(\theta_e))_{e\in\Es}$. Therefore, as $k\rightarrow 0$, the measure
$\PPforest^k$ on spanning forests converges weakly to the measure $\PPtree$.
\end{rem}

Using the computations for the Green function of
Equations~\eqref{equ:explicit1} and~\eqref{equ:explicit2}, we can write down the probability
under $\PPforest^k$ of a single edge $e=xy$ to be in the random rooted spanning forest
\begin{align}
 \PPforest^k(\{e\})&=\sc(\theta_e)(G^m(x,x)+G^m(y,y)-2G^m(x,y))\nonumber\\
  &= \frac{2\sc(\theta_e)
  K'}{\pi}(k'-\dn(\theta_e))+2H(2\theta_e), \label{eq:proba_edge}
\end{align}
and that of a vertex $x$ to be a root:
\begin{equation}
  \label{eq:proba_vertex}
  \PPforest^k(\{x\})= m^2(x) G^m(x,x)=\frac{m^2(x) K'k'}{\pi}.
\end{equation}

\subsection{Free energy of rooted spanning forests on periodic isoradial graphs}
\label{sec:free_energy}

Suppose that the isoradial graph $\Gs$ is $\ZZ^2$-periodic and  
let $(\Gs_n)_{n\geq 1}$ be the natural exhaustion by toroidal graphs:
$\Gs_n=(\Vs_n,\Es_n):=\Gs/n\ZZ^2$. 
Since conductances and masses only depend on the elliptic modulus $k$,
we denote by $\Zforest^k(\Gs_n)$ the partition function of rooted spanning forests of $\Gs_n$.

Define the
\emph{free energy of rooted spanning forests}, denoted by $F^k_{\mathrm{forest}}$, to be minus the exponential growth rate of the rooted spanning forest partition functions
of the graphs $\Gs_n$. 
\begin{equation*}
F^k_{\mathrm{forest}}=-\lim_{n\rightarrow\infty}\frac{1}{n^2} \log \Zforest^k(\Gs_n).
\end{equation*}
Then, we obtain the following result.

\begin{thm}
\label{thm:free_energy}
For every $k\in(0,1)$, the free energy of the rooted spanning forest model on $\Gs$ admits the following
formula in terms of the angles of the isoradial embedding:
\begin{align}\label{equ:free_energy_thm}
  F^k_{\mathrm{forest}}&=
  -\vert \Vs_1 \vert \int_0^K 4H'(2\theta)\log\sc(\theta)\ud\theta-
  \sum_{e\in\Es_1}
  \int_{0}^{\theta_e}\frac{2H(2\theta)\sc'(\theta)}{\sc(\theta)}\ud\theta\\
  &=-\vert \Vs_1 \vert \int_0^K 4H'(2\theta)\log\sc(\theta)\ud\theta+
  \sum_{e\in\Es_1}\Bigl(-
  2H(2\theta_e)\log\sc(\theta_e)+\int_{0}^{\theta_e}4H'(2\theta)\log\sc(\theta)\ud\theta\Bigr),\nonumber
\end{align}
where $H$ is the function defined in Equation~\eqref{def:H}. 
\end{thm}



\begin{proof}

For every $n\geq 1$, let $\Delta^m_n$ be the massive Laplacian matrix of the
graph $\Gs_n$.

Using symmetries of the graph $\Gs_n$ under the group $(\ZZ/n\ZZ)^2$, the matrix $\Delta^m_n$ can be
block diagonalized, and by Theorem~\ref{thm:matrix_forest},
\begin{equation*}
  \frac{1}{n^2}\log \Zforest(\Gs_n,\rho,m^2) = 
  -\frac{1}{n^2}\log \det \Delta^m_n =
  -\frac{1}{n^2} \sum_{j,\ell=0}^{n-1}
  \log P_{\Delta^m}(e^{\frac{2i\pi j}{n}}, e^{\frac{2i\pi \ell}{n}}),
\end{equation*}
where $P_{\Delta^m}(z,w)=\det\Delta^m(z,w)$ is the characteristic polynomial, see
Section~\ref{sec:qperfunc}. Since it does not vanish on the torus $\TT^2$,
this quantity converges to
\begin{equation*}
F_{\mathrm{forest}}=-\iint_{|z|=|w|=1}\log \det\Delta^m(z,w)\frac{\ud z}{2\pi i z}\frac{\ud w}{2\pi i w}.
\end{equation*}

This formula is true for any biperiodic weighted
graphs, as long as the mass is strictly positive at one vertex at least. When
all the masses are zero, this expression is the free energy of the
spanning tree model of the graph.

When conductances become infinite, the free energy blows up. A relevant,
related quantity is the \emph{entropy} of the model:
\begin{equation*}
  S_{\text{forest}} = -F_{\text{forest}} -
    \sum_{e\in \Es_1} \PPforest(\{e\}) \log \rho(\theta_e)-
    \sum_{x\in \Vs_1} \PPforest(\{x\})\log m^2(x).
\end{equation*}

Note that in a rooted spanning forest, the number of roots plus the number of edges is
equal to the number of vertices. Therefore, 
\begin{equation*}
  \sum_{e\in \Es_1} \PPforest(\{e\}) +
  \sum_{x\in \Vs_1} \PPforest(\{x\})
  = \vert\Vs_1\vert.
\end{equation*}
As a consequence, if we multiply all conductances and squared masses by the same
factor $\lambda$, $F_{\text{forest}}$ gets an extra additive constant
$-|\Vs_1|\log\lambda$, whereas $S_{\text{forest}}$ stays unchanged. In particular,
it always gives a finite result.


To find the formula for the free energy, following ideas of~\cite{Kenyon3}, we
study  its variation as the embedding
of the graph is modified by tilting the train-tracks, see Section \ref{sec:NaturalOperations}.

Let us consider a smooth deformation of the isoradial graph $\Gs$, \emph{i.e.}, a
continuous family of isoradial graphs $(\Gs(t))_{t\in[0,1]}$ obtained by varying
the directions $\alpha_T(t)$ of the train-tracks smoothly with $t$, in such a way that
$\Gs(1)=\Gs$ and $\Gs(0)=\Gs^{\text{flat}}$, where $\Gs^{\text{flat}}$ is an isoradial graph whose edges
have half-angles equal to\footnote{%
  This is in contradiction with our hypothesis that all the angles
  $\overline\theta_e$ are bounded away from $0$ and $\frac{\pi}{2}$. We can still make sense
  of it. In particular, we can suppose that the condition of bounded angles is
  true for $\Gs(\varepsilon)$, as soon as $\varepsilon>0$.
} $0$ or $\frac{\pi}{2}$.
More precisely, every vertex of $\Gs^{\text{flat}}$ has two incident vertices
with angle $\overline{\theta}_e=\frac{\pi}{2}$ and infinite conductance (called the
\emph{short edges}), the other incident edges (called the \emph{long edges})
having $\overline{\theta}_e=0$, thus zero conductance. The short edges form 
nontrivial disjoint cycles on the fundamental domain $\Gs^{\text{flat}}_1$. At a
vertex $x$ of degree $n$ of $\Gs^{\text{flat}}$, the mass becomes:
\begin{equation*}
  m^2(x)=\lim_{%
    \substack{%
    \theta_i,\theta_j\to K \\ 
     \theta_k\to 0,\ k\neq i,j
   }
 } \sum_{\ell=1}^n (\Armbis(\theta_\ell)-\sc(\theta_\ell))  =0,
\end{equation*}
since the function $\Armbis-\sc$ vanishes at $0$ and $K$, see the proof of Proposition \ref{prop:neg_mass}.
Let us first compute the entropy $S_{\text{forest}}^{\text{flat}}$, when the
graph becomes flat,
by dividing all the conductances and masses by the largest
conductance. After this renormalization, all edge-weights and
masses on $\Gs^{\text{flat}}$ are zero, except for the short edges: the entropy
we want to compute is thus the entropy of the spanning tree model on the
degenerate periodic graph only made of copies of the short edges, forming infinite lines. Since the number of
spanning trees on a cycle does not grow exponentially with its size, the number
of spanning trees on $\Gs^{\text{flat}}_N$ does not grow exponentially with
$N^2$, and thus the entropy of the model on $\Gs^{\text{flat}}$ is equal to zero.

One could then follow the variation of the entropy along the deformation.
However, it is simpler to use a twisted definition of the entropy, which does not
really have a physical interpretation, but whose variation is easier to analyze.
Let us define:
\begin{align}
  \label{eq:def_Stilde}
  \widetilde{S}_{\text{forest}}&=
  -F_{\text{forest}} - \sum_{e\in\Es_1} 2H(2\theta_e) \log\rho(\theta_e) \\
  &= S_{\text{forest}} 
  + \sum_{e\in\Es_1}[\PPforest(\{e\})-2H(2\theta_e)]\log\sc(\theta_e)
  +\sum_{x\in\Vs_1}\PPforest(\{x\})\log m^2(x).\nonumber
\end{align}
As the graph becomes flat, $\widetilde{S}_{\text{forest}}$ tends to zero, since its
difference with $S_{\text{forest}}$ becomes negligible, as can be checked from
Equations~\eqref{eq:proba_edge} and \eqref{eq:proba_vertex}.

Denote by $F_{\text{forest}}(t)$ and
$\widetilde{S}_{\text{forest}}(t)$ the free energy and the twisted entropy for the
rooted spanning forest model on the graph $\Gs(t)$.
As the angles of the train-tracks are supposed to vary smoothly with $t$, one
can write:
\begin{align*}
\frac{\ud F_{\mathrm{forest}}(t)}{\ud t}&=-
 \iint_{|z|=|w|=1}  \frac{\ud}{\ud t} \log\det \Delta^m(z,w)\frac{\ud z}{2\pi i z}\frac{\ud w}{2\pi i w}
\end{align*}
\begin{align*}
\frac{\ud F_{\mathrm{forest}}(t)}{\ud t}&=-\iint_{|z|=|w|=1} \sum_{x,y\in\Vs_1} 
     \frac{\partial \log\det \Delta^m(z,w)}
          {\partial \Delta^m(z,w)_{x,y}} 
     \frac{\ud \Delta^m(z,w)_{x,y}}{\ud t}
  \frac{\ud z}{2\pi i z} \frac{\ud w}{2\pi i w} \\ 
  & = -\sum_{x,y\in\Vs_1} \iint_{|z|=|w|=1} (\Delta^m(z,w)^{-1})_{y,x}
    \frac{\ud\Delta^m(z,w)_{x,y}}{\ud t} 
  \frac{\ud z}{2\pi i z} \frac{\ud w}{2\pi i w},
\end{align*}
since for an invertible matrix $M=(M_{i,j})$, one has
$\frac{\partial \log\det M}{\partial M_{i,j}} = (M^{-1})_{j,i}$.

By definition of the massive Laplacian matrix $\Delta^m$,
the nonzero contributions of the entries of its Fourier transform
$\Delta^m(z,w)_{x,y}$ can be split into two categories:
\begin{itemize}
\item If $(x,y)$ defines a directed edge $e$ of $\Gs_1$,
  then $\Delta^m(z,w)_{x,y}$ has a term equal to $-\rho(\theta_{e})$, possibly
  multiplied by a nontrivial power of $z$ and $w$ if
  the lifts of $x$ and $y$ in $\Gs$ belong to different fundamental domains.
  In that case, if the contribution is $-\rho(\theta_e) z^i w^j$, then its
  derivative with respect to $t$ is $-\frac{\ud \rho(\theta_e)}{\ud t} z^i w^j$.
\item If $x=y$, there is also in $\Delta^m(z,w)_{x,x}$ a term $d(x)$, coming
  from the diagonal of $\Delta^m$.
\end{itemize}
Note that in some cases, in particular for graphs $\Gs$ with a small fundamental
domain, the two types of contributions can appear on the diagonal. However, if
that happens, the term $d(x)$ is the only one with no extra power of $z$ or
$w$. The other terms on the diagonal come by pair with opposite exponents for
$z$ and $w$, corresponding to the two possible directions of the edge crossing
$\widetilde\gamma_x$ and/or $\widetilde\gamma_y$.

From Equation \eqref{eq:periodic_Green}
and using also the symmetry of the Green function
we thus obtain:
\begin{align*}
\frac{\ud F_{\mathrm{forest}}(t)}{\ud t}
&=-\sum_{x\in\Vs_1} G^m(x,x)\frac{\ud(d(x))}{\ud t}
+2\sum_{e=xy\in\Es_1}G^m(x,y)\frac{\ud\rho(\theta_e)}{\ud t}.
\end{align*}

Along the deformation, the graph $\Gs(t)$ stays isoradial, so the formulas for
the conductances, the diagonal term of the massive Laplacian and the Green
function in terms of the elliptic functions hold.
Let us handle the first term. The diagonal term $d(x)$
by Equation~\eqref{eq:def_mass} is:
\begin{equation*}
d(x)=\sum_{e\sim x} \Armbis(\theta_e).
\end{equation*}
Moreover, by Equation~\eqref{eq:green_diag} of Appendix~\ref{app:comput_green},
$G^m(x,x) = \frac{k' K'}{\pi}$, which does not depend on $x$. We can
therefore rewrite the first term as
\begin{equation}
  \label{equ:free1}
  \sum_{x\in \Vs_1}\! G^m(x,x) \frac{\ud(d(x)) }{\ud t} = 
  \frac{k' K'}{\pi}
  \sum_{x\in \Vs_1}\! \sum_{e \sim x}\! \frac{\ud \Armbis(\theta_e)}{\ud t}
  = 2 \frac{k' K'}{\pi}
  \sum_{e \in \Es_1} \frac{\ud \Armbis(\theta_e)}{\ud t} 
  \stackrel{\eqref{eq:deriv_Abis}}{=}
  \sum_{e\in\Es_1}\! 2\frac{K'}{\pi} \frac{\dn^2(\theta_e)}{\cn^2(\theta_e)}
  \frac{\ud \theta_e}{\ud t}.
\end{equation}

We now handle the second term. 
By definition, $\rho(\theta_e)=\sc(\theta_e)$ and 
$\sc'=\dn\cdot \cn^{-2}$.
By Formula~\ref{pointc} of Lemma~\ref{lem:Gneighbor} proved in Appendix~\ref{app:comput_green}, 
$G^m(x,y)=\frac{K'\dn(\theta_e)}{\pi}-\frac{H(2\theta_e)}{\sc(\theta_e)}$. The
second term can therefore be rewritten as
\begin{equation}\label{equ:free2}
2\sum_{e=xy\in\Es_1}G^m(x,y)\frac{\ud\rho(\theta_e)}{\ud t}
= 2\sum_{e\in\Es_1}
  \left[
  \frac{K'\dn^2(\theta_e)}{\pi\cn^2(\theta_e)}
  - \frac{H(2\theta_e)\dn(\theta_e)}{\cn(\theta_e)\sn(\theta_e)}
\right]\frac{\ud \theta_e}{\ud t}.
\end{equation}

Combining Equations~\eqref{equ:free1} and \eqref{equ:free2}, we deduce
\begin{align}\label{eq:simp_f}
  \frac{\ud F_{\text{forest}}(t)}{\ud t} &= \sum_{e\in\Es_1}
  f(\theta_e)\frac{\ud \theta_e}{\ud t}, &
  \text{with}
  \quad
  f(\theta)
  &= -2\frac{H(2\theta)\dn(\theta)}{\cn(\theta)\sn(\theta)}
  = -2\frac{H(2\theta)\sc'(\theta)}{\sc(\theta)}.\\
  \intertext{Similarly,}
  \frac{\ud\widetilde{S}_{\text{forest}}(t)}{\ud t} &= \sum_{e\in \Es_1}
  \widetilde{s}(\theta_e)\frac{\ud\theta_e}{\ud t}, &
  \text{with}
  \quad
    \widetilde{s}(\theta) &= -f(\theta) - \frac{\ud}{\ud \theta}
    2H(2\theta)\log\sc(\theta) = -4H'(2\theta)\log\sc(\theta).\nonumber
  \end{align}

  Therefore to compute $\widetilde{S}_{\text{forest}}$ for the graph $\Gs$, it
  suffices to integrate
  $\frac{\ud \widetilde{S}_{\text{forest}}}{\ud t}$ along the deformation:
\begin{equation*}
  \widetilde{S}_{\text{forest}}=
  \widetilde{S}_{\text{forest}}(1)-\widetilde{S}_{\text{forest}}(0) 
  = \int_{0}^{1}\sum_{e\in\Es_1}
  \widetilde{s}(\theta_e(t))\frac{\ud\theta_e}{\ud t} \ud t 
  =\sum_{e\in \Es_1} \int_{\theta_e^{\text{flat}}}^{\theta_e}
  \widetilde{s}(\theta)
  \ud\theta.
\end{equation*}

Among the parameters $(\theta_e^{\text{flat}})_{e\in\Es_1}$, exactly $|\Vs_1|$
are equal to $K$, the others being $0$. Using moreover that $\widetilde{S}_{\text{forest}}(0)=\widetilde{S}_{\text{forest}}^{\text{flat}}=0$,
we then have:
\begin{equation*}
  \widetilde{S}_{\text{forest}}=
\sum_{e\in \Es_1} \int_{0}^{\theta_e}
\widetilde{s}(\theta)
\ud\theta - |\Vs_1|\int_0^K\widetilde{s}(\theta)\ud\theta.
\end{equation*}

Finally, one can compute the free energy from~\eqref{eq:def_Stilde}:
\begin{equation}
  F_{\text{forest}} 
  = -\widetilde{S}_\text{forest} - \sum_{e\in\Es_1} 2H(2\theta_e)\log\sc(\theta_e)
  =|\Vs_1|
    \int_{0}^{K}\widetilde{s}(\theta)\ud\theta
  +\sum_{e\in\Es_1} \int_{0}^{\theta_e}f(\theta)\ud\theta,
\label{free_energy_final}
\end{equation}
which is exactly the expression given in Theorem~\ref{thm:free_energy}. The
equality between the two expressions follows from an integration by parts.
\end{proof}

\begin{rem}
Formulas~\eqref{equ:free_energy_thm} of Theorem \ref{thm:free_energy} are a
continuous expression of $k$. When $k$ goes to zero,
$H'$ becomes constant, $\sc$ becomes $\tan$,
and the first integral becomes up to some multiplicative constant
\begin{equation*}
  \int_0^{\frac{\pi}{2}}\log\tan(\theta)\ud \theta,
\end{equation*}
which is zero by antisymmetry. 
Splitting $\log\tan\theta=\log\sin\theta -\log\cos\theta$ in the
remaining integral of the second formula
yields the following value:
when $k\to 0$:
\begin{equation}
\label{eq:crit_free_energy}
     F_{\textnormal{forest}}^0=-\sum_{e\in\Es_1} \frac{2}{\pi} (L(\theta_e)+L(\pi/2-\theta_e))+\frac{2\theta_e}{\pi}\log \tan \theta_e,
\end{equation}
where $L$ is the Lobachevsky function, \emph{i.e.}, $L(x)=-\int_{0}^x\log (2\sin t)\text{d}t$. 
This is up to a negative sign, the logarithm of the normalized determinant of the Laplacian operator of \cite{Kenyon3}. 
By slightly adapting the proof above, one sees that \eqref{eq:crit_free_energy} actually is the free energy of the spanning tree model 
on $\Gs$ with conductances $(\tan(\theta_e))_{e\in\Es}$.
\end{rem}

The next result proves a second order phase transition at $k=0$ in the rooted spanning forest model. 
This shows that the spanning tree model with conductances $(\tan(\theta_e))_{e\in\Es}$, 
corresponding to the Laplacian introduced in \cite{Kenyon3}, is a \emph{critical model}; thus giving full meaning to the terminology
\emph{critical} used in the paper \cite{Kenyon3}. Note that the conductances
and masses behave smoothly in the neighborhood of $k=0$, see Lemma \ref{lem:analytic_weight}.

\begin{thm}
\label{thm:free_energy_k0}
Let $F_{\textnormal{forest}}^0$ be the free energy of spanning trees with critical
conductances $(\tan(\theta_e))_{e\in\Es}$. The free
energy $F_{\textnormal{forest}}^k$ admits the following expansion around $k=0$:
\begin{equation*}
  F_{\textnormal{forest}}^k=F_{\textnormal{forest}}^0 - k^2\log k^{-1} \vert \Vs_1\vert  +O(k^2).
\end{equation*}
As a consequence the model of rooted spanning forests on $\Gs$ exhibits a phase transition of order two at $k=0$.
\end{thm}

\begin{proof}
We start from the terms involving $f$ in the second equality of 
Equation~\eqref{free_energy_final}, in which we perform
the change of variable from $\theta$ to $\overline\theta=\frac{\pi \theta}{2K}$:
\begin{equation*}
     F_e:=
     \int_{0}^{\theta_e}
     f(\theta)\,\ud\theta
     =-\frac{4K}{\pi}\int_{0}^{\overline{\theta}_e}H\Bigl(\frac{4K\overline\theta}{\pi}\Bigr)\frac{\dn}{\sn\cdot\cn}\Bigl(\frac{2K\overline\theta}{\pi}\Bigr)\,\text{d}\overline\theta.
\end{equation*}

We use the expansion of $H$ in terms of the \emph{nome} $q=e^{-\pi K'/K}$:
\begin{equation}
\label{eq:expansion_nome_H}
     H\Bigl(\frac{4K\overline\theta}{\pi}\Bigr)=\frac{\overline\theta}{\pi}+\frac{2K'}{K}
     \sum_{s=1}^\infty  \frac{q^s}{1-q^{2s}} \sin(2s\overline\theta)=\frac{\overline\theta}{\pi}-\frac{2}{\pi}\log q
     \sum_{s=1}^\infty  \frac{q^s}{1-q^{2s}} \sin(2s\overline\theta).
\end{equation}
In order to prove \eqref{eq:expansion_nome_H}, we use the expression of $H$ in terms of $\Erm$, see \eqref{eq:new_formulation_Armbis}, 
as well as the expansion of $\Erm$ in terms of the nome, which can be obtained from \cite[17.4.28 and 17.4.38]{AS}.
The following expansion near $k=0$ holds (see \cite[17.3.14 and 17.3.21]{AS})
\begin{equation*}
  q=\frac{k^2}{16}+\frac{k^4}{32}+O(k^6).
\end{equation*}
We obtain that
\begin{equation*}
     H\Bigl(\frac{4K\overline\theta}{\pi}\Bigr)=\frac{\overline\theta}{\pi}-k^2\log k \frac{\sin(2\overline\theta)}{4\pi} +O(k^2).
\end{equation*}
We now multiply by $\frac{\dn}{\sn\cdot\cn}(\frac{2K\overline\theta}{\pi})$, which is analytic in $k^2$ and admits the expansion 
$\frac{1}{\sin \overline\theta \cos\overline\theta}+O(k^2)$, see \cite[16.13.1--16.13.3]{AS}, and we integrate. In this way, we obtain
\begin{equation*}
     F_e = -\frac{2}{\pi} \int_{0}^{\overline{\theta}_e} \frac{\overline\theta}{\sin \overline\theta \cos \overline\theta}\text{d}{\overline\theta}+k^2\log k \frac{\overline{\theta}_e}{\pi}+O(k^2),
\end{equation*}
where we have made use of the standard identity $\sin 2\overline\theta=2\sin \overline\theta \cos \overline\theta$. 
The constant coefficient of $F_e$ is integrated by parts to get:
\begin{equation*}
     \frac{2}{\pi} \int_{0}^{\overline{\theta}_e}
     \frac{\overline\theta}{\sin \overline\theta \cos \overline\theta}\text{d}{\overline\theta} 
     = 
     \frac{2}{\pi}\Bigl(L(\overline{\theta}_e)+L\Bigl(\frac{\pi}{2}-\overline{\theta}_e\Bigr)\Bigr)
     -\frac{2\overline{\theta}_e}{\pi}\log \tan(\overline{\theta}_e),
\end{equation*}
with $L$ equal to the Lobachevsky function.
Similar computations as above give that $\int_{0}^{K}\widetilde{s}(\theta)\ud\theta$ admits the following
expansion when the parameter $k$ goes to $0$:
\begin{equation*}
  \int_{0}^{K}\widetilde{s}(\theta)\ud\theta= \frac{k^2\log k }{2}+O(k^2).
\end{equation*}

When summing all the contributions to the free energy, the 
constant coefficient is exactly $F_{\textnormal{forest}}^0$ from Equation~\eqref{eq:crit_free_energy},
and the coefficient in front of $k^2\log k$ is:
\begin{equation*}
  \frac{|\Vs_1|}{2} + \frac{1}{\pi}\sum_{e\in\Es_1}\overline\theta_e.
\end{equation*}
But since around every vertex of $\Gs_1$, the half-angles of the rhombi sum to
$\pi$, we have:
\begin{equation*}
  \frac{1}{\pi}\sum_{e\in\Es_1}\overline\theta_e =
  \frac{1}{2\pi}\sum_{x\in\Vs_1}\sum_{e\sim x} \overline\theta_e = \frac{|\Vs_1|}{2}.\qedhere
\end{equation*}
\end{proof}

\subsection{$Z$-invariance of the spanning forest model}\label{sec:Zinvariance}

Theorem~\ref{thm:expression_Green} proves an explicit, local expression for the massive
Green function of an isoradial graph with the choice of weights~\eqref{equ:cond_intro}.
From the point of view of statistical mechanics, this feature is expected from models defined on isoradial graphs that are \emph{$Z$-invariant}.
Although already present in the papers~\cite{Kennelly,Onsager},
the notion of $Z$-invariance has been fully developed by Baxter in the context of the integrable 8-vertex
model~\cite{Baxter:8V}, in connection with the Ising model and the $q$-Potts model~\cite{Baxter:Zinv}, and is 
directly related to the Yang-Baxter equations satisfied by the weights of integrable models~\cite{Perk:YB,Baxter:exactly,Kenyon6}. 


In this section, we define $Z$-invariance for rooted spanning forests, explain why one expects local expressions for probabilities, and
make explicit the Yang-Baxter equations. Then in Theorem~\ref{thm:Z_invariant_Lap}, using 3-dimensional consistency of the massive Laplacian 
(Proposition~\ref{prop:3Dconsistency}), we prove that 
with the choice of conductances and mass of Definition~\ref{def:massiveLap}, the model of rooted spanning
forests is indeed $Z$-invariant.


\subsubsection{Definition}\label{sec:deftriangleetoile}

Let $\Gsstar$ and $\Gstriang$ be finite isoradial graphs differing
by a star-triangle
transformation, as defined in Section~\ref{sec:NaturalOperations}. For
convenience of the reader, we repeat Figure~\ref{fig:star_triangle},
fixing notation for vertices and weights around the star/triangle. 
  
\begin{figure}[ht]
\begin{center}
\input{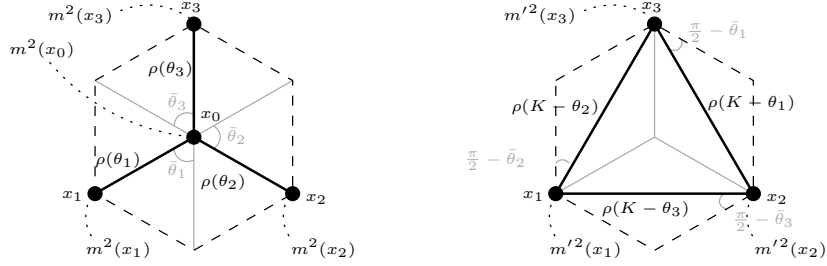}
\end{center}
\caption{\emph{Star-triangle transformation} and notation. If an isoradial graph $\Gsstar$
(left) has a \emph{star}, \emph{i.e.}, a vertex $x_0$ of degree $3$, it can be transformed into a new
isoradial graph $\Gstriang$ (right) having a \emph{triangle} connecting the three neighbors $x_1,x_2,x_3$ of
$x_0$, by shifting around the three rhombi of the underlying rhombus graph $\GR$, and vice-versa.
}
\end{figure}


$Z$-invariance imposes strong relations on the partition functions of $\Gsstar$ and $\Gstriang$.
They are more easily expressed using the bijection of Section \ref{subsec:rsf_ust}: instead of considering 
the rooted spanning forest partition functions of $\Gsstar$ and $\Gstriang$, we take
the spanning tree partition functions of $\Gsstar^{\rs}$ and $\Gstriang^{\rs}$. 

Let $\Gs'$ be the graph obtained from $\Gsstar^{\rs}$ by removing the vertex $x_0$, the edges $x_i x_0$, $x_i \rs$, $i\in\{1,2,3\}$ and 
$x_0\rs$. Note that $\Gs'$ is also obtained from $\Gstriang^{\rs}$ by removing the edges $x_i\rs$ and $x_i x_{i+1}$ (in cyclic notations), 
$i\in\{1,2,3\}$.


Denote by $\widetilde{\T}(\Gs')$ the set of edge-configurations of $\Gs'$, which
can be extended to spanning trees on $\Gstriang^\rs$ and $\Gsstar^\rs$.
For $\widetilde{\Ts}\in\widetilde{\T}(\Gs')$,
let $Z(\Gsstar|\widetilde{\Ts})$ (resp.\ $Z(\Gstriang|\widetilde{\Ts})$) be the
restricted spanning tree partition function of $\Gsstar^\rs$ (resp.\
$\Gstriang^\rs$) coinciding with $\widetilde{\Ts}$ outside the location of the
star-triangle transformation, \emph{i.e.} the sum of the weights of the local
configurations used to extend $\widetilde{\Ts}$ to
a full spanning tree of the whole graph $\Gsstar^\rs$ (resp. $\Gstriang^\rs$).


\begin{defi}
The rooted spanning forest model is \emph{$Z$-invariant}, if the conductances assigned to edges and masses assigned to vertices
are such that there exists a constant $\C$, such that 
for every  $\widetilde{\Ts}\in\widetilde{\T}(\Gs')$, we have:
\begin{equation*}
Z(\Gsstar|\widetilde{\Ts})=\C Z(\Gstriang|\widetilde{\Ts}).
\end{equation*}
\end{defi}

\begin{rem}\label{rem:equiv_proba}
Since the probability of an event can be written as the ratio of the
partition function restricted to the event 
and the full partition function, the condition of $Z$-invariance is equivalent to asking that this probability is not affected by star-triangle 
transformations performed away from the event. In particular, this suggests that formulas for probabilities should have the 
locality property.
%
\end{rem}

\subsubsection{Yang-Baxter equations of rooted spanning forests}\label{sec:ZinvEqu}

Actually, $Z(\Gsstar|\widetilde{\Ts})$ and $Z(\Gstriang|\widetilde{\Ts})$ only depend on the connection properties of $\widetilde{\Ts}$ outside of 
the star-triangle, so that we can partition $\widetilde{\T}(\Gs')$ according to whether the configuration $\widetilde{\Ts}$ satisfies:
\begin{itemize}
 \item $\Rs^{\{x_1,x_2,x_3\}}$: vertices $x_1,x_2,x_3$ are connected to $\rs$,
 \item $\Rs^{\{x_i,x_j\}}$: vertices $x_i,x_j$ are connected to $\rs$, $x_k$ is not; $i\neq j\neq k$,
 $\{i,j\}\subset \{1,2,3\}$,
 \item $\Rs^{\{x_i\}}$: the vertex $x_i$ is connected to $\rs$, $x_j,x_k$ are not; $i\in\{1,2,3\}$,
 \item $\Rs^{\emptyset}$: none of the vertices $x_1,x_2,x_3$ is connected to $\rs$.
\end{itemize}

Denote by $\Rs$ any condition above.
With a slight abuse of notation, if $\widetilde{\Ts}$ satisfies the condition $\Rs$,
we will write $Z(\Gsstar|\Rs)$ for $Z(\Gsstar|\widetilde{\Ts})$ and the same for
$\Gstriang$.
The model is thus $Z$-invariant
if and only if, there exists a constant $\C$, such that:
\begin{equation}
\label{equ:star_triangle_1}
\forall\,\text{ condition $\Rs$},\ Z(\Gsstar|\Rs)=\C\,Z(\Gstriang|\Rs) \quad \text{(Yang-Baxter equations).}
\end{equation}

Classicaly $Z$-invariance is proved by showing that the weights satisfy the Yang-Baxter equations.
In this paper we provide an alternative, shorter proof, see 
Theorem~\ref{thm:Z_invariant_Lap}\footnote{A direct proof showing that our choice of weights satisfy the Yang-Baxter equations can be found on the 
first arXiv version of this paper.}, which does not require 
making the equations explicit.
We nevertheless write them down for three reasons: first, the first equation allows to explicitly compute the constant 
$\C$; second, they are not present in the physics literature and might be of interest to this community; third, it is quite remarkable that such rather complicated looking
equations have a one-parameter family of solutions.

Writing the Yang-Baxter equations amounts to making explicit the contributions $Z(\Gsstar|\Rs)$ and
$Z(\Gstriang|\Rs)$ in the four cases above.

$\bullet$ Case $\Rs^{\{x_1,x_2,x_3\}}$ illustrated in Figure~\ref{fig:Zinvariance}.
  If all the $x_\ell$'s are connected to $\rs$, then in $\Gstriang^\rs$, one can add
  (exactly) one edge to connect $x_0$ to $\rs$ through one of the three vertices
  $x_\ell$, (and have a weight $\rho(\theta_\ell)$), or directly connect $x_0$
  to $\rs$ through an edge with weight $m^2(x_0)$. On $\Gstriang^\rs$, there is
  nothing to do, so the total weight is 1. This yields the following identities
 \begin{align}
 Z(\Gsstar|\Rs^{\{x_1,x_2,x_3\}})&=\sum_{\ell=1}^3 \rho(\theta_\ell)+m^2(x_0),\label{Z_inv:equ_1}\\
 Z(\Gstriang|\Rs^{\{x_1,x_2,x_3\}})&=1. \label{Z_inv:equ_2}
 \end{align}

\begin{figure}[h]
\begin{center}
\resizebox{\textwidth}{!}{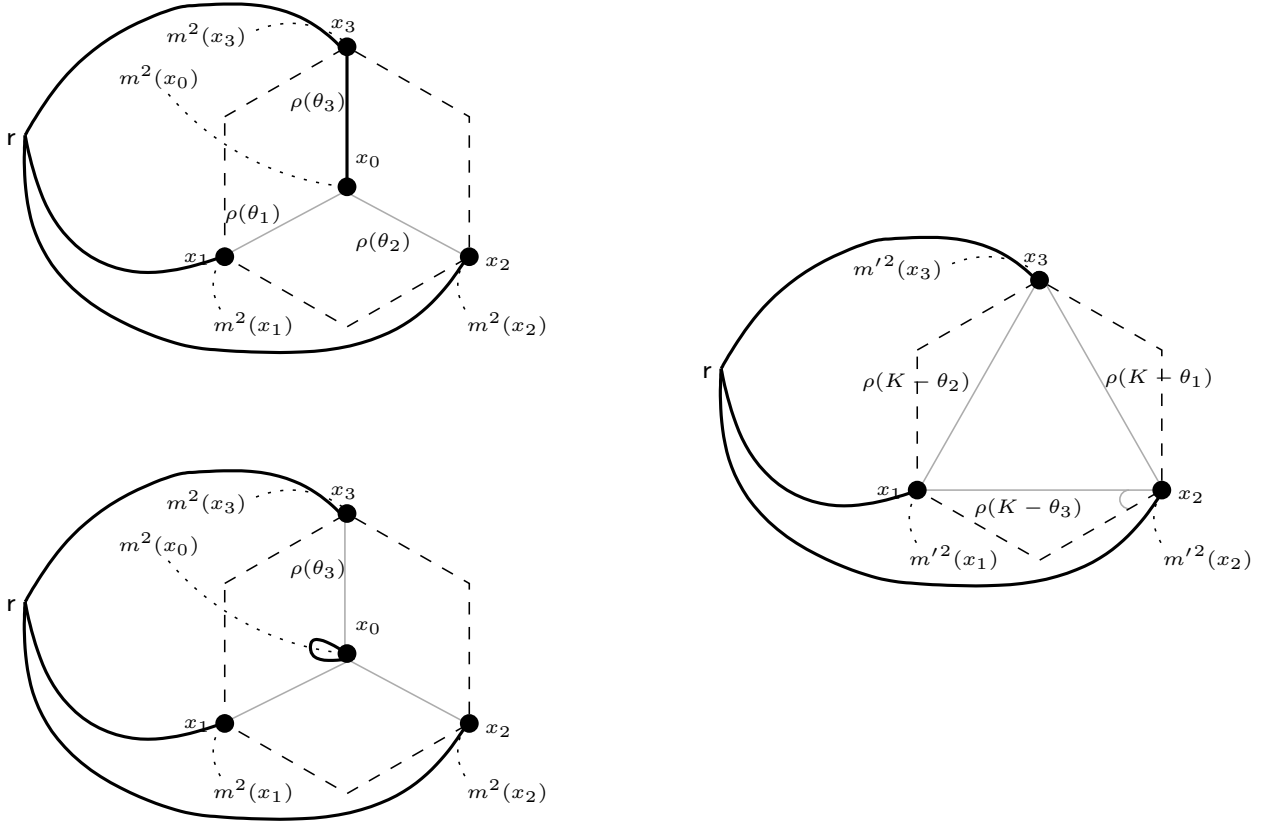}
\end{center}
\caption{Left (with the analog containing either $\rho(\theta_1)$ or $\rho(\theta_2)$) and center: possible configurations for 
$Z(\Gsstar|\Rs^{\{x_1,x_2,x_3\}})$. Right: possible configuration for $Z(\Gstriang|\Rs^{\{x_1,x_2,x_3\}})$.}\label{fig:Zinvariance}
\end{figure}

Similar considerations for the other three cases lead to the following
expressions of $Z(\Gsstar|\Rs)$ and $Z(\Gstriang|\Rs)$. The expressions are
longer because the number of possible situations increases.

$\bullet$ Case $\Rs^{\{x_i,x_j\}}$:
 \begin{align*}
 Z(\Gsstar|\Rs^{\{x_i,x_j\}})=&
 \ \rho(\theta_k)\Bigl[\sum_{\ell\neq k} \rho(\theta_\ell)\Bigr]
 +m^2(x_0)\rho(\theta_k)
 + m^2(x_k)\Bigl[\sum_{\ell=1}^3 \rho(\theta_\ell)+m^2(x_0)\Bigr],\\
 Z(\Gstriang|\Rs^{\{x_i,x_j\}})=&
 \ \sum_{\ell\neq k}\rho(K-\theta_\ell)+m'^2(x_k).
 \end{align*}
$\bullet$ Case $\Rs^{\{x_i\}}$:
  \begin{align*}
 Z(\Gsstar|\Rs^{\{x_i\}}) =&
 \ \prod_{\ell=1}^3\rho(\theta_\ell) 
 +m^2(x_0)\prod_{\ell\neq i}\rho(\theta_\ell)
 +\sum_{\ell\neq i}m^2(x_\ell)\rho(\theta_{{\{i,\ell\}^c}}) 
    \Bigl[\sum_{\ell'\in\{i,\ell\}}\rho(\theta_{\ell'})\Bigr]
\\
&+m^2(x_0)[m^2(x_k)\rho(\theta_j)+m^2(x_j)\rho(\theta_k)]
+ \Bigl[\prod_{\ell\neq i}m^2(x_\ell)\Bigr]
    \Bigl[\sum_{\ell=1}^3\rho(\theta_\ell)+m^2(x_0)\Bigr],
 \\
 Z(\Gstriang|\Rs^{\{x_i\}})=&
 \ \sum_{\ell=1}^3 \prod_{\ell'\neq\ell}\rho(K-\theta_{\ell'})
+\sum_{\ell\neq i}m'^2(x_\ell)
    \Big[\sum_{\ell'\in\{i,\ell\}}\rho(K-\theta_{\ell'})\Bigr]
+\prod_{\ell\neq i}m'^2(x_\ell).
 \end{align*}
 Above, $\{i,\ell\}^c$ denotes the complementary set of $\{i,\ell\}$, \emph{i.e.}, the unique index $k$ which is not $i$ and $\ell$.
 
$\bullet$ Case $\Rs^{\emptyset}$:
 \begin{align*}
 Z(\Gsstar|\Rs^{\emptyset})=&
 \ \Bigl[\sum_{i=0}^3 m^2(x_i)\Bigr]\Bigl[\prod_{i=1}^3\rho(\theta_i)\Bigr]
 +m^2(x_0)\sum_{i=1}^3 m^2(x_i)\prod_{\ell\neq i}\rho(\theta_{\ell})\\
 &+ \sum_{i=1}^3 \Bigl[\prod_{\ell\neq i }m^2(x_{\ell})\Bigr]\rho(\theta_i)
     \Bigl[\sum_{\ell\neq i}\rho(\theta_\ell)\Bigr]
 +m^2(x_0)\sum_{i=1}^3
     \Bigl[\prod_{\ell\neq i}m^2(x_{\ell})\Bigr]\rho(\theta_{i})\\
 &
 +\Bigl[\prod_{i=1}^3 m^2(x_i)\Bigr]
     \Bigl[\sum_{i=1}^3\rho(\theta_i)+m^2(x_0)\Bigr],\\
 Z(\Gstriang|\Rs^\emptyset)=&\ \Bigl[\sum_{i=1}^3 m'^2(x_i)\Bigr]\Bigl[\sum_{i=1}^3\prod_{\ell\neq i}
 \rho(K-\theta_{\ell}) \Bigr]
 +\sum_{i=1}^3\Bigl[\prod_{\ell\neq i}m'^2(x_{\ell})\Bigr]
 \Bigl[\sum_{\ell\neq i}\rho(K-\theta_{\ell})\Bigr]\\&+\prod_{i=1}^3 m'^2(x_i).
 \end{align*}
 
\begin{rem}
  When $k=0$, the equations drastically simplify since all the masses are 0. They are the Yang-Baxter equations
  of the spanning tree model, and reduce to 
  the so-called Kennelly's
  theorem~\cite{Kennelly}, linking the conductances so
  that the electric networks $\Gsstar$ and $\Gstriang$ are equivalent.
  When eliminating $\rho(K-\theta_i)$ from the equations, one is left with a
  single equation:
  \begin{equation*}
    \rho(\theta_1)+\rho(\theta_2)+\rho(\theta_3)=
    \rho(\theta_1)\rho(\theta_2)\rho(\theta_3),
    \qquad \theta_1+\theta_2+\theta_3=\pi,
  \end{equation*}
  which, when parametrized by taking $\rho(\theta)=\tan(\theta)$ is the
  triple tangent identity. This expression for $\rho(\theta)$
  coincides with the critical conductances for trees on isoradial graphs
  introduced in~\cite{Kenyon3}.
\end{rem} 
 
\subsubsection{$Z$-invariance of the rooted spanning forests model}

The next theorem proves that, with the choice of conductances and masses of Equations \eqref{equ:def_conductances} 
and \eqref{eq:def_mass}, the rooted spanning forest model is $Z$-invariant.

\begin{thm}
\label{thm:Z_invariant_Lap}
Let $k\in[0,1)$. Suppose that conductances assigned to edges, and masses assigned to vertices are given by Equations \eqref{equ:def_conductances} 
and \eqref{eq:def_mass}. 
Then, the model of rooted spanning forests is $Z$-invariant with constant $\C(k)=k'\sc(\theta_1)\sc(\theta_2)\sc(\theta_3)$.

%
\end{thm}

\begin{rem}
We conjecture that the conductances and masses from
Equations~\eqref{equ:def_conductances} and~\eqref{eq:def_mass} provide a complete
parametrization of the Yang-Baxter equations of rooted spanning forests.
\end{rem}
\begin{rem}
In the actual state of knowledge, $Z$-invariance
does not provide a way of finding local expressions, but it gives a
framework for choosing the parameters of the model. In some cases (not including ours), there are some
elements in that direction though in the work by \cite{AdlerBobenkoSuris}
through the link between 3-dimensional consistency of some classes of equations
on isoradial graphs, 
and existence of solutions of these equations with a product structure.
\end{rem}

\begin{proof}
Let us first suppose that $Z$-invariance is proved and compute the constant $\C(k)$.
From Equations~\eqref{Z_inv:equ_1} and \eqref{Z_inv:equ_2}, we know that $\C(k)=\sum_{\ell=1}^3 \rho(\theta_\ell)+m^2(x_0)$. 
Using Equation~\eqref{equ:m0}, we deduce that $\C(k)=k'\prod_{\ell=1}^3 \sc(\theta_\ell)$. 

Using Remark~\ref{rem:equiv_proba}, proving $Z$-invariance is equivalent to proving invariance of the probability measure under
star-triangle transformations. Using the transfer-impedance theorem~\ref{thm:transimp_forest} ~\cite{BurtonPemantle}
it suffices to show that the Green functions $G^m_{\Gstriang}(x,y)$ and $G^m_{\Gsstar}(x,y)$ are equal on all common vertices 
(that is different than $x_0$). Let us fix $x$, then $G^m_{\Gstriang}(x,\cdot)$ is harmonic everywhere except at $x$. By 
Proposition~\ref{prop:3Dconsistency}, there is a unique way of extending $G^m_{\Gstriang}(x,\cdot)$ to $\Gsstar$ in such a way
that its massive Laplacian takes the same value at every vertices other than $x_0$, and is equal to 0 at $x_0$. This new function is 
massive harmonic everywhere except at $x$, it is thus equal to 
$G^m_{\Gsstar}(x,\cdot)$. By construction it is equal to $G^m_{\Gstriang}(x,\cdot)$ on all common vertices thus concluding the proof.
%
%
\end{proof}

\appendix

\section{Useful identities involving elliptic functions}
\label{app:elliptic}

In this section we list required identities satisfied by elliptic functions. We also derive properties and identities satisfied by the
functions $\Arm$ and $H$ defined in Section~\ref{sec:EllipticFunctions}.

\subsection{Identities for Jacobi elliptic functions}

\paragraph{Change of argument.} 
Jacobi elliptic functions satisfy various addition formulas by quarter-periods and half-periods, among which:
\begin{align}
&\sc(u-K\vert k)=-\frac{1}{k'}\sc(u\vert k)^{-1}, \label{id:scK} \text{ (\cite[2.2.17--2.2.18]{La89})}\\
&\dn(u+K\vert k)=k' \dn(u\vert k)^{-1}, \label{id:dnK} \text{ (\cite[2.2.19]{La89})}\\
&\sc(u+2iK'\vert k)=-\sc(u+2K\vert k)=-\sc(u\vert k), \label{id:sc2iKp} \text{ (\cite[2.2.11--2.2.12]{La89})}\\
&\sc(u+iK'\vert k)={i}{\dn(u\vert k)^{-1}}, \label{id:sciKp} \text{ (\cite[2.2.17--2.2.18]{La89})}\\
&\sn(u-iK'\vert k)=\frac{1}{k}\sn(u\vert k)^{-1}, \label{id:sniKp} \text{ (\cite[2.2.11--2.2.17]{La89}).}
\end{align}
\paragraph{Jacobi imaginary transformation.} These transformations, which are proved in \cite[2.6.12]{La89}, refer to the substitution of $u$ by $iu$ in the argument of Jacobi elliptic functions:
\begin{equation}
     \sn(iu\vert k)=i\sc(u\vert k'),\qquad \cs(iu\vert k)=-i\ns(u\vert k'),\qquad
\dn(iu\vert k)=\dc(u\vert k').\label{equ:sncomplex}
\end{equation}

\paragraph{Derivatives of Jacobi functions.} 
These derivatives are computed in \cite[16.16]{AS}:
\begin{equation}
\label{eq:derivatives_Jacobi_functions}
\sn'=\cn\dn,\qquad
\cn'=-\sn\dn,\qquad
\dn'=-k^2\sn \cn.
\end{equation}
\paragraph{Ascending Landen transformation.} This allows to express the ratio $\frac{\sn \cdot \cn }{\dn}$ as an $\sn$ function, with a different elliptic modulus:
\begin{equation}
\label{eq:alt}
     \frac{\sn \cdot \cn }{\dn}(u\vert k)=\frac{\sn ((1+\mu)u\vert \ell)}{1+\mu},\qquad \text{with\ } \ell = \frac{2-k^2-2\sqrt{1-k^2}}{k^2} \text{ and } \mu=\frac{1-\ell}{1+\ell}.
\end{equation}
It is stated in \cite[3.9.19]{La89}. Furthermore, the values of $K(k),K'(k)$ are related to those of $K(\ell),K'(\ell)$ as follows (this can be noticed indirectly, by comparing the periods of the above functions):
\begin{equation}
\label{eq:K_Kp_after_Landen}
     K(k)=(1+\ell)K(\ell),\qquad K'(k)=\frac{K'(\ell)}{1+\mu}.
\end{equation}

\subsection{Identities for the functions $\Arm$ and $H$}
\label{app:ellipticAH}

Recall the definition of the function $\Arm(\cdot\vert k)$, see Equation \eqref{def:Abis} of Section~\ref{sec:EllipticFunctions},
\begin{equation*}
     \Arm(u\vert k)=\frac{1}{k'}\left(\Dc(u\vert k)+\frac{E-K}{K}u\right),\quad \text{where }\Dc(u\vert k)=\int_0^{u} \dc^2(v\vert k)\,\ud v.
\end{equation*}

\begin{lem}
\label{cor:Armbis}
The function $\Arm(\cdot\vert k)$ is odd and satisfies the following identities:
\begin{align}
\bullet\ &\Arm(K-u\vert k)=-\Arm(u\vert k)+\frac{1}{k'}\ns(u\vert k)\dc(u\vert k),\label{cor:Armbis:item2}\\
\bullet\ &\Arm(v-u\vert k)=\Arm(v\vert k)-\Arm(u\vert k)-k'\sc(u\vert k)\sc(v\vert k)\sc(v-u\vert k),\label{cor:Armbis:item3}\\
\bullet\ &\Arm(u+2K\vert k)=\Arm(u\vert k),\label{cor:Armbis:item1}\\
\bullet\ &\Arm(u+2iK'\vert k)=\Arm(u\vert k)+\frac{i\pi}{k'K},\label{cor:Armbis:item1bis}\\
\bullet\ &\frac{\textnormal{d}\Arm}{\textnormal{d}u}(u\vert k)=\frac{\dc^2(u\vert k)}{k'}-\frac{K-E}{k'K}.\label{eq:deriv_Abis}
\end{align}
\end{lem}
\begin{proof} 
Consider \emph{Jacobi epsilon function} $\Erm(\cdot\vert k)$, see also \cite[3.4.25]{La89},
\begin{equation}
\label{eq:jacob_epsi}
     \forall\,u\in\CC,\quad\Erm(u\vert k)=\int_0^{u} \dn^2(v\vert k)\,\ud v.
\end{equation}
Performing the change of variable $u\rightarrow iu$ in $\Erm(iu\vert k')$ and using Jacobi imaginary transformation \eqref{equ:sncomplex}, 
the function $\Arm$ can be expressed as
\begin{equation}
\label{eq:new_formulation_Armbis}
     \Arm(u\vert k)=-\frac{i}{k'}\Erm(iu\vert k')+\frac{E-K}{k'K}u.
\end{equation}
The function $\Arm$ is odd because $\Erm$ is. Moreover, Jacobi epsilon function satisfies the following:
\begin{align}
\bullet\ &\Erm(-u+iK'\vert k)=-\Erm(u\vert k)+i(K'-E')-\cs(u\vert k)\dn(u\vert k),\label{equ:Erm2}\ \text{ (\cite[3.6.17]{La89})} \\
\bullet\ &\Erm(v-u\vert k)=\Erm(v\vert k)-\Erm(u\vert k)+k^2\sn(u\vert k)\sn(v\vert k)\sn(v-u\vert k),\label{equ:Erm3}\ \text {(\cite[3.5.14]{La89})}\\
\bullet\ &\Erm(2K\vert k)=2E,\quad \Erm(2iK'\vert k)=2i(K'-E'),
\label{equ:Erm10}\ \text{ (\cite[3.6.22]{La89})}.
\end{align}
We first prove \eqref{cor:Armbis:item2}.
From Identities \eqref{eq:new_formulation_Armbis} and \eqref{equ:Erm2}, we have
\begin{align*}
 \Arm(-u+K\vert k)+\Arm(u\vert k)&=-\frac{i}{k'}(\Erm(-iu+iK\vert k')+\Erm(iu\vert k'))+\frac{E-K}{k'}\\&=-\frac{i}{k'}
(i(K-E)-\cs(iu\vert k')\dn(iu\vert k'))+\frac{E-K}{k'}\\
&=\frac{i}{k'}\cs(iu\vert k')\dn(iu\vert k').
\end{align*}
The proof is concluded using Jacobi imaginary transformation~\eqref{equ:sncomplex}.
We turn to the proof of \eqref{cor:Armbis:item3}. 
Using the definition of $\Arm(v-u\vert k)$ and Identity~\eqref{equ:Erm3} evaluated at $k'$, we have:
\begin{equation*}
\Arm(v-u\vert k)=\Arm(v\vert k)-\Arm(u\vert k)-ik'\sn(iu\vert k')\sn(iv\vert k')\sn(i(v-u)\vert k').
\end{equation*}
The proof is again concluded using Jacobi imaginary transformation~\eqref{equ:sncomplex}. 
We now move to the proof of~\eqref{cor:Armbis:item1} and~\eqref{cor:Armbis:item1bis}. 
From~\eqref{equ:Erm10} used with $k'$ instead of $k$, we have 
\begin{align*}
\Arm(2K\vert k)&=-\frac{i}{k'}\Erm(2iK\vert k')+2\frac{E-K}{k'}=-\frac{i}{k'}2i(K-E)+2\frac{E-K}{k'}=0,\\
\Arm(2iK'\vert k)&=
-\frac{i}{k'}(\Erm(-2K'\vert k'))+\frac{2iK'(E-K)}{k'K}=-\frac{i}{k'}(-2E')+\frac{2iK'(E-K)}{k'K}
=\frac{i\pi}{k'K},
\end{align*}
where the last equality is a consequence of Legendre's identity \eqref{eq:Legendre}.
The proof of~\eqref{cor:Armbis:item1} (resp.\ \eqref{cor:Armbis:item1bis}) is concluded using \eqref{cor:Armbis:item3} 
evaluated at $-u$ and $2K$ (resp.\ $-u$ and $2iK'$).

Finally, Equation \eqref{eq:deriv_Abis} readily follows from \eqref{eq:new_formulation_Armbis}.
\end{proof}


Recall the definition of the function $H(u\vert k)=  \frac{-ik'K'}{\pi}\Arm(\frac{iu}{2}\vert k')$, see \eqref{def:H}.

\begin{lem}
\label{lem:h}
The function $H(\cdot\vert k)$ satisfies the following properties:
\begin{align*}
\bullet\ &H(u+4K\vert k)=H(u\vert k)+1, \\
\bullet\ &H(u+4iK'\vert k)=H(u\vert k),\\
\bullet\ &\lim_{k\to 0}H(u\vert k)=\frac{u}{2\pi},\\
\bullet\ &\text{$H$ has a simple pole in the rectangle $[0,4K]+[0,4iK']$, at $2iK'$,
       with residue $\frac{2K'}{\pi}$.}
\end{align*}
\end{lem}

\begin{proof}
The first two properties of Lemma \ref{lem:h} immediately follow from using \eqref{cor:Armbis:item1} and \eqref{cor:Armbis:item1bis} in \eqref{def:H}. 
The limit of $H$ as $k\to0$ is computed using the alternative expression \eqref{eq:new_formulation_Armbis}:
\begin{equation*}
     H(u\vert k)=\frac{-ik'K'}{\pi}\Arm\Bigl(\frac{iu}{2}\Big\vert k'\Bigr)=\frac{K'}{\pi}\left(\Erm\Bigl(\frac{u}{2}\Big\vert k\Bigr)+
     \frac{E'-K'}{K'}\frac{u}{2}\right).
\end{equation*}
We conclude using the fact that $\lim_{k\to 0}\Erm(u\vert k)=u$, see \eqref{eq:jacob_epsi} and Section \ref{sec:EllipticFunctions}, 
together with $\lim_{k\to 0}E'(k)=1$.

Finally, $\dc$ has a simple pole at $K$ with residue $1$, hence by \eqref{eq:jacob_Dc} the same holds true for $\Dc$. 
Accordingly $\Dc(\frac{iu}{2}\vert k')$ has a pole of order $1$ at $-2iK'$ with residue $-2i$. Using the expression \eqref{def:Abis} as well 
as the fact that $H$ is odd completes the proof.
\end{proof}

\section{Explicit computations of the Green function}
\label{app:comput_green}

In this section, we explicitly compute values of the Green function along the diagonal and for incident vertices.
We use the explicit formula \eqref{eq:greenbis} of Theorem~\ref{thm:expression_Green} and the residue theorem. 
The second formula for incident vertices uses the symmetry of the Green function.
Note that it is not immediate that the two formulas 
\ref{pointb} and~\ref{pointc} are indeed equal.
The first is more useful in the proof of Theorem~\ref{thm:expression_Green}, 
the third is more attractive since it only involves the half-angle
$\theta$.

%

\begin{lem}\label{lem:Gneighbor}$\,$
\begin{enumerate}
\item\label{point1} Let $x$ be a vertex of $\Gs$. Then, the Green function on the diagonal at $x$ is equal to:
\begin{equation}
  \label{eq:green_diag}
  G^m(x,x) =   \frac{k' K'}{\pi}.
\end{equation}
\item\label{point2} 
Let $x$ and $y$ be neighboring vertices in $G$, endpoints of an edge $e$ in a
rhombus spanned by $e^{i\overline{\alpha}}$, and $e^{i\overline{\beta}}$, of half-angle
$\theta = \frac{\beta-\alpha}{2}\in (0,K)$ with
$y=x+e^{i\overline{\alpha}}+e^{i\overline{\beta}}$, then we have the following
expressions for 
the Green function evaluated at $(x,y)$:
\begin{enumerate}[label={\rm (\alph{*})},ref={\rm (\alph{*})}]
\item\label{pointa} $\displaystyle G^m(x,y)
= \frac{H(\alpha+2K)-H(\beta+2K)}{\sc(\theta)}+\frac{k' K'}{\pi}\expo_{(x,y)}(2iK')$,
\item\label{pointb}$\displaystyle G^m(x,y)=
\frac{H(\alpha)-H(\beta)}{\sc(\theta)}+
\frac{K'}{\pi}\dn\big(\frac{\alpha}{2}\bigr)\dn\bigl(\frac{\beta}{2}\bigr)$,
\item\label{pointc} $\displaystyle G^m(x,y)=
-\frac{H(2\theta)}{\sc(\theta)}+
\frac{K'}{\pi}\dn(\theta)$.
\end{enumerate}
\end{enumerate}
\end{lem}

\begin{proof}[Proof of Point~\ref{point1}.]
 Using expression \eqref{eq:green}, we have
 \begin{equation*}
   G^m(x,x) = \frac{k'}{4i\pi}\oint_{\Cs} 1\ud u,
 \end{equation*}
 where $\Cs$ is any contour winding once vertically on $\TT(k)$ (the
 contour can be anywhere, since the integrand has no pole).
 Take for $\Cs$ a vertical segment and parametrize it by the ordinate
 $w=\text{Im}(u)$. Using that the length of $\Cs$ is $4K'$ and that
 $\ud u=i\ud  w$, one readily gets
 \begin{equation*}
   G^m(x,x) = \frac{k'}{4i\pi} (i4K') =
   \frac{k' K'}{\pi}.\qedhere
 \end{equation*}
\end{proof}

\begin{proof}[Proof of Point~\ref{point2}.]
  We first prove~\ref{pointa}. Using expression \eqref{eq:greenbis} and
replacing the exponential function by its definition, we need to compute
\begin{equation*}
  G^m(x,y)= - \frac{k'^2}{4i\pi} \oint_{\gamma_{x,y}} H(u)\,\sc\Bigl(\frac{u-\alpha}{2}\Bigr)\,\sc\Bigl(\frac{u-\beta}{2}\Bigr)\ud u.
\end{equation*}
where $\gamma_{x,y}$ is a trivial contour containing the pole $2iK'$ of $H$ and the poles $\alpha+2K$, $\beta+2K$ of the 
exponential function. The residue of $\sc$ at $K$ is $-1/k'$, see \eqref{id:scK}
or \cite[Table~16.7]{AS}, from which we deduce that:
\begin{equation*}
  \res_{u=\alpha+K}\Bigl[\sc\Bigl(\frac{u-\alpha}{2}\Bigr)\Bigr]=-\frac{2}{k'}.
\end{equation*}
By the residue theorem, we thus have:
\begin{align*}
  G^m(x,y) &= - \frac{k'^2}{2}
  \Bigl(
   \underbrace{\ -\frac{2}{k'}[H(\alpha+2K)\sc(K-\theta)+H(\beta+2K)\sc(K+\theta)]}_{\text{residues at $\alpha+2K$ and $\beta+2K$}}+
  \underbrace{\ -\frac{2K'}{k'\pi}\expo_{x,y}(2iK')}_{\text{residue at $2iK'$}}.
  \Bigr)\\
  &\stackrel{\eqref{id:scK}}{=}
  \frac{H(\beta+2K)-H(\alpha+2K)}{\sc(\theta)}+\frac{k' K'}{\pi}\expo_{x,y}(2iK'),
\end{align*}
which concludes the proof of~\ref{pointa}.
Note that by Identity~\eqref{id:sciKp},
$\expo_{(x,y)}(2iK')=
\frac{k'}{\dn(\frac{\alpha}{2})\dn(\frac{\beta}{2})}$.

Expression~\ref{pointb} is obtained by symmetry of $G^m$, by exchanging the
role of $x$ and $y$, transforming $\alpha$ and $\beta$ into $\alpha+2K$ and
$\beta+2K$, respectively. Using Identity~\eqref{id:dnK}, one gets
\begin{equation*}
  \expo_{(y,x)}(2iK')=
\frac{k'}{\dn(\frac{\alpha+2K}{2})\dn(\frac{\beta+2K}{2})}=
\frac{\dn(\frac{\alpha}{2}) \dn(\frac{\beta}{2})}{k'}.
\end{equation*}

To obtain~\ref{pointc} we again use 
Equation~\eqref{eq:greenbis} but instead of the function $H$, we use the function $\widetilde{H}$:
\begin{equation*}
  \widetilde{H}(u)=H(u-\alpha).
\end{equation*}
Indeed, since $\widetilde{H}-H$ is an elliptic function, it satisfies the conditions required for Equation~\eqref{eq:greenbis} to
hold, see Remark~\ref{rem:critical_Green}.
After a change of variable $v=u-\alpha$ in the integral, we get
\begin{equation*}
  G^m(x,y)=-\frac{k'^2}{4i\pi}\oint_\gamma
  H(v)\,\sc\Bigl(\frac{v}{2}\Bigr)\,\sc\Bigl(\frac{v-2\theta}{2}\Bigr)\ud v,
\end{equation*}
which would be the integral expression giving~\ref{pointb} when $\alpha=0$, and
$\beta=2\theta$. Expression~\ref{pointc} is then obtained using the fact that
$H(0)=0$ and $\dn(0)=1$.
\end{proof}

\section{Identities for weights of the star-triangle transformation}
\label{app:stt}

In this section we prove identities for weights involved in the star-triangle transformation, used in 
Sections~\ref{sec:integrability} and~\ref{sec:Zinvariance}. We refer to Figure~\ref{fig:star_triangle} for notation. 

\begin{lem}
\label{lem:rewritingm}
We have the following identities for weights involved in the star-triangle transformation:
\begin{align}
k'\prod_{\ell=1}^3 \rho(\theta_\ell)&=m^2(x_0)+\sum_{\ell=1}^3 \rho(\theta_\ell),\label{equ:m0}\\
{m'}^2(x_k)-m^2(x_k)&=\rho(\theta_k)-\sum_{\ell\neq k}\rho(K-\theta_\ell)-k'\rho(K-\theta_i)\rho(K-\theta_j)\rho(\theta_k).\label{equ:m1_bis}
\end{align}
\end{lem}

\begin{proof}
Equation \eqref{equ:m0} 
is a consequence of \eqref{cor:Armbis:item3} and \eqref{cor:Armbis:item1} using that $\theta_1+\theta_2+\theta_3=2K$.
We now prove Equation~\eqref{equ:m1_bis}. Note that the star-triangle transformation implies that ${m'}^2(x_k)-m^2(x_k)$ only depends on 
three angles, that we call $\theta_i,\theta_j,\theta_k$. We have:
\begin{align*}
{m'}^2(x_k)&-m^2(x_k)=\\
&=\Arm(K-\theta_i)+
\Arm(K-\theta_j)-\sc(K-\theta_i)-\sc(K-\theta_j)-\Arm(\theta_k)+\sc(\theta_k)\\
&=\Arm(K-\theta_i)+
\Arm(K-\theta_j)-\Arm(2K-(\theta_i+\theta_j))-\sc(K-\theta_i)-\sc(K-\theta_j)+\sc(\theta_k)\\
&=-
k'\sc(K-\theta_i)\sc(K-\theta_j)\sc(2K-(\theta_i+\theta_j))-\sc(K-\theta_i)-\sc(K-\theta_j)+\sc(\theta_k),\\
&\quad\quad\text{by Point~\eqref{cor:Armbis:item3} of Lemma~\ref{cor:Armbis}}\\
&=-k'\sc(K-\theta_i)\sc(K-\theta_j)\sc(\theta_k)-\sc(K-\theta_i)-\sc(K-\theta_j)+\sc(\theta_k),
\end{align*}
thus ending the proof.
\end{proof}

\section{Random walks and rooted spanning forests}\label{app:stat_mec}

\label{app:statmech_finite}

In this Appendix, we collect some facts about rooted spanning forests, killed random walks and their link to 
network random walks and spanning trees, that are useful for Section \ref{sec:statmech}.

Suppose for the moment that $\Gs$ is a finite connected
(not necessarily isoradial) graph, with a massive
Laplacian $\Delta^m$. Equivalently, by Equation~\eqref{equ:operator_general},
$\Gs$ is endowed with positive conductances $(\rho(e))_{e\in\Es}$
and positive masses $(m^2(x))_{x\in\Vs}$.
Consider the graph $\Gs^\rs=(\Vs^\rs,\Es^\rs)$ obtained from $\Gs$
by adding a root vertex $\rs$ and joining every vertex of $\Gs$ to $\rs$, as in Section \ref{subsec:rsf_ust}. The graph $\Gs^\rs$ is weighted by the function
$\rho^m$, see~\eqref{equ:defrhom}.

%
%

\subsection{Massive harmonicity on $\Gs$ and harmonicity on $\Gs^\rs$}

\label{subsec:massive_harmo}

There is a natural (non-massive) Laplacian $\Delta_\rs$ on $\Gs^\rs$,
acting on functions $f$ defined on vertices of $\Gs^\rs$: 
\begin{equation*}
  \forall\ x\in\Vs^\rs,\quad \Delta_\rs f(x) = 
  \sum_{xy\in\Es^\rs} \rho^m(xy)[f(x)-f(y)].
\end{equation*}
Then the restriction $\Delta^{(\rs)}_{\rs}$ of the matrix of $\Delta_\rs$
to vertices of $\Gs$, obtained by removing the row and column corresponding to
$\rs$, is exactly the matrix $\Delta^m$.

Functions on vertices of $\Gs$ are in bijection with functions on vertices of $\Gs^\rs$ taking
value $0$ on $\rs$ (by extension/restriction). This bijection is compatible with
the Laplacians on $\Gs$ and $\Gs^\rs$: if $f$ is a function on $\Gs$ and
$\widetilde{f}$ is its extension to $\Gs^\rs$ such that $\widetilde{f}(\rs)=0$, then  
$\Delta^m f = \Delta_\rs \widetilde{f}$ on
$\Gs$.

The operator $\Delta^m$ is invertible, and its inverse is
$G^m$, the massive Green function of $\Gs$.
The matrix $\Delta_\rs$ is not invertible: its kernel is exactly the space of
constant functions on $\Gs^\rs$, but its restriction to functions on
$\Gs^{\rs}$ vanishing at $\rs$ is invertible, and its negated inverse is exactly
$\widetilde{G}^m$, the extension of $G^m$ to $\Gs^\rs$, taking the value 0 at $\rs$:
\begin{equation*}
  \forall\ x,y\in\Vs^{\rs},\quad \widetilde{G}^m(x,y)=\widetilde{G}^{m}(y,x)=
  \begin{cases}
    G^m(x,y) & \text{if $x$ and $y$ are vertices of $\Gs$,} \\
    0 & \text{if $x$ or $y$ is equal to $\rs$.}
  \end{cases}
\end{equation*}

\subsection{Random walks}
\label{subsec:rw_app}

The \emph{network random walk} $(Y_j)_{j\geq 0}$ on $\Gsr$ with initial state $x_0$ is defined by $Y_0=x_0$ and
jumps
\begin{equation*}
\forall\,x,y\in \Vs^\rs,\quad
P_{x,y}=\PP_{x_0}[Y_{j+1}=y|Y_j=x]=
\begin{cases}
\displaystyle\frac{\rho^m(xy)}{\rho^m(x)}&\text{ if $y\sim x$,}\\
0&\text{ otherwise},
\end{cases}
\end{equation*}
where
\begin{equation}
\label{eq:def_rho^m}
  \rho^m(x)=
  \sum_{y\in\Vs^\rs:y\sim x}\rho^m(xy)=
  \begin{cases}
    \sum\limits_{y\in\Vs:y\sim x}\rho(xy)+m^2(x)& \text{if $x\neq \rs$,} \\
    \sum\limits_{y\in\Vs} m^2(y) & \text{if $x=\rs$.}
  \end{cases}
\end{equation}
The Markov matrix $P=(P_{x,y})$ is related to the Laplacian $\Delta_\rs$ as follows: if
$A_{\rs}$ denotes the diagonal matrix whose entries are the diagonal entries of the Laplacian $\Delta_\rs$, then
\begin{equation}
\label{eq:link_matrix_P_Laplacian}
     P = I+ (A_\rs)^{-1}\Delta_\rs.
\end{equation}
This random walk is positive recurrent. The \emph{potential} $V_\rs(x,y)$ of this
random walk is defined as the
difference in expectation of the number of visits at $y$ starting from $x$ and
from $y$:
\begin{equation*}
  V_{\rs}(x,y)=
  \mathbb{E}_x\Biggl[\sum_{j=0}^{\infty}\mathbb{I}_{\{Y_j=y\}}\Biggr] -
  \mathbb{E}_y\Biggl[\sum_{j=0}^{\infty}\mathbb{I}_{\{Y_j=y\}}\Biggr]. 
\end{equation*}
Although both sums separately are infinite, the difference makes sense and is
finite, as can be seen by computing $V_{\rs}(x,y)$ with
a coupling of the random walks starting from $x$ and $y$, where they evolve
independently until they meet (in finite time a.s.), and stay
together afterward. 

Because $(Y_j)$ is (positive) recurrent, the time $\tau_\rs$ for $(Y_j)$ to
hit $\rs$ is finite a.s.
We can define the killed random walk
$(X_j)=(Y_{j\wedge(\tau_r-1)})$, absorbed at the
root $\rs$. The process $(X_j)$ visits only a finite number of vertices of
$\Gs$ before being absorbed: every vertex is thus transient. If $x$ and
$y$ are two vertices of $\Gs$, then we can define the potential of $(X_j)$,
$V^{m}(x,y)$, as the expected number of visits at $y$ of $(X_j)$
starting from $x$, before it gets absorbed. $V^m$ and $V_\rs$ are linked by the
formula below, which directly follows from the strong Markov property:
\begin{equation}
  \label{eq:massive_potential}
  \forall\ x,y\in\Gs,\quad
  V^{m}(x,y) = V_{\rs}(x,y)-V_{\rs}(\rs,y).
\end{equation}

As a matrix, $V^m$ is equal to $(I-Q^m)^{-1}$ where $Q^m$ is the substochastic
transition matrix for the killed process $(X_j)$. Given that
$Q^m=I-(A^m)^{-1}\cdot\Delta^m$, where $A^m$ is the diagonal
matrix extracted from $\Delta^m$, $V^m$ is related to the Green function by the
following formula:
\begin{equation}
\label{eq:potential_vs_green}
  V^m(x,y) = \frac{1}{A^m_{x,x}}(\Delta^m)^{-1}_{x,y} =
  \frac{G^m(x,y)}{\rho^m(x)}.
\end{equation}

Another quantity related to the potential is the \emph{transfer impedance
matrix} $\Hs$, whose rows and columns are indexed by oriented edges of the graph.
If $e=(x,y)$ and $e'=(x',y')$ are two directed edges of $\Gs^\rs$, 
the coefficient $\Hs(e,e')$ is the expected number of times that this
random walk $(Y_j)$, started at $x$ and stopped the first time it hits $x$,
crosses the edge $(x',y')$ minus the expected number of times that
it crosses the edge $(y',x')$:
\begin{equation*}
  \Hs(e,e')= 
  [V_{\rs}(x,x')-V_{\rs}(y,x')] P_{x',y'} -
  [V_{\rs}(x,y')-V_{\rs}(y,y')] P_{y',x'}.
\end{equation*}
The quantity $\Hs(e,e')/\rho(e')$ is symmetric in $e$ and $e'$, and is changed to
its opposite if the orientation of one edge is reversed.

When $e$ and $e'$ are in fact edges of $\Gs$,
by~\eqref{eq:massive_potential} and the definition of the
transition probabilities for the processes $(Y_j)$ and $(X_j)$, 
$V_{\rs}(x,x')-V_{\rs}(y,x')= V^{m}(x,x')-V^{m}(y,x')$ and
$P_{x',y'}=Q_{x',y'}=\rho(x'y')/\rho^m(x)$
(and similarly when exchanging the roles of $x'$ and $y'$). Therefore,
\begin{align}
  \label{eq:tranfer_impedance}
  \Hs(e,e')&= 
  [V^m(x,x')-V^m(y,x')] Q_{x',y'} -
  [V^m(x,y')-V^m(y,y')] Q_{y',x'} \nonumber \\
  &=\rho(x'y')[G^m(x,x')-G^m(y,x')-G^m(x,y')-G^m(y,y')].
\end{align}
If one of the vertices of $e$ or $e'$ is $\rs$, then the same formula holds if
we replace $G^m$ by $\widetilde{G}^{m}$, \emph{i.e.}, if we put to $0$ all the terms
involving the root $\rs$.

\subsection{Spanning forests on $\Gs$ and spanning trees on $\Gs^{\rs}$}\label{sec:app_derivation_classical}

Recall the definition of rooted spanning forests on $\Gs$ and spanning trees of
$\Gs^{\rs}$ from Section~\ref{subsec:rsf_ust}.
Kirchhoff's matrix-tree theorem~\cite{Kirchhoff} states that
spanning trees of $\Gsr$ are counted by the
determinant of $\Delta_{\rs}^{(\rs)}$, obtained from
$\Delta_{\rs}$ by deleting the row and column corresponding to 
 $\rs$:
\begin{thm}[\cite{Kirchhoff}]\label{thm:matrix_tree}
The spanning forest partition function of the graph $\Gs$ is equal to:
\begin{equation*}
     \Zforest(\Gs,\rho,m)=\det \Delta_{\rs}^{(\rs)}.
\end{equation*}
\end{thm}

Using the fact stated in Section~\ref{subsec:massive_harmo} that
$\Delta_r^{(r)}=\Delta^m$, we exactly obtain
Theorem~\ref{thm:matrix_forest}.

%
The explicit expression for the Boltzmann measure of spanning trees is due to
Burton and Pemantle~\cite{BurtonPemantle}.
Fix an arbitrary orientation of the edges of $\Gsr$.

\begin{thm}[\cite{BurtonPemantle}]
For any distinct edges $e_1,\dotsc,e_k$ of $\Gsr$, the probability that these
edges belong to a spanning tree of $\Gsr$ is:
$$
\PPtree(e_1,\dotsc,e_k)=\det (\Hs(e_i,e_j))_{1\leq i,j\leq k}.
$$
\end{thm}

Using the correspondence between edges (connected to $\rs$, or not)  in the spanning tree
of $\Gs^\rs$ and edges and roots for the corresponding rooted spanning forest of
$\Gs$, together with the expression of the transfer impedance matrix $\Hs$ in
terms of the massive Green function on $\Gs$ from
Equation~\eqref{eq:tranfer_impedance}, one exactly gets the statement of
Theorem~\ref{thm:transimp_forest}.

Due to the bijection between spanning trees on $\Gs^\rs$
and rooted spanning forests on $\Gs$, the latter can be generated by Wilson's
algorithm~\cite{Wilson} from the killed random walk $(X_j)$. Indeed, if we take $\rs$ as starting point of the spanning tree, 
and construct its branches by loop erasing the random walk $(Y_j)$, the obtained
trajectories are exactly loop erasures of $(X_j)$.

\subsection{Killed random walk on infinite graphs and convergence of the Green functions along exhaustions}
\label{sec:rw}

In this section  we define the killed random walk on an infinite graph $\Gs$, as well as its 
associated potential and Green function. We then prove  (Lemma 
\ref{lem:convergence_Green}) that the Green functions associated to an exhaustion 
$(\Gs_n)_{n\geq 1}$ of $\Gs$ converge pointwise to the Green function of $\Gs$. 
Lemma~\ref{lem:convergence_Green} is an important preliminary result to Theorem 
\ref{thm:infinite_vol_meas}.

In the case where $\Gs$ is infinite, it is not possible to consider the network
random walk $(Y_j)$ on $\Gs^\rs$, the graph obtained from $\Gs$ by adding the root $\rs$
connected to the other vertices, because the degree of $\rs$ is infinite and the
conductances associated to edges connected to $\rs$ are bounded from below by a
positive quantity, and are thus not summable. However, it is
possible to directly 
define the walk $(X_j)$, killed when it reaches
$\rs$. Its transition
probabilities are:
\begin{equation}
\label{eq:P_Green}
Q^m_{x,y}=\PP(X_{j+1}=y|X_j=x)=
 \left\{\begin{array}{cl}
  \displaystyle\frac{\rho(xy)}{\sum_{z\sim x} \rho(xz) + m^2(x)} &\text{if $y$
  and $x$ are neighbors,} \\ 
   0 & \text{otherwise,}
  \end{array}\right.
\end{equation}
and the probability of being absorbed at $x$ is
$\overline{Q^m_x}=\PP(X_{j+1}=\rs|X_j=x)=1-\sum_{xy\in\Es}Q^m_{x,y}$.

Under the condition that the conductances and masses are uniformly bounded away
from 0 and
infinity (which is the case on isoradial graphs, as soon as $k>0$ and the angles
of the rhombi are bounded away from 0 and $\frac{\pi}{2}$), the probability of
being absorbed at any given site
is bounded from below by some uniform positive quantity. The process $(X_j)$ is
thus absorbed in finite time, and vertices of $\Gs$ are transient. We will
assume that this condition is fulfilled.

There is the same link~\eqref{eq:link_matrix_P_Laplacian} as in Section~\ref{subsec:rw_app} 
between the substochastic matrix $Q^m=(Q^m_{x,y})$ and the Laplacian $\Delta^m$.

The {potential} $V^{m}$ of the discrete random walk $(X_j)$ is a
function on $\Gs\times\Gs$ defined at $(x,y)$ as the expected time spent at
vertex $y$ by the discrete random walk $(X_j)$ started at $x$ before being
absorbed (below, $\tau_\rs$ is defined as the first hitting time of $\rs$, as in Section \ref{subsec:rw_app}):
\begin{equation}
\label{eq:def_pot_krw}
  V^{m}(x,y)=\mathbb{E}_x\Biggl[\sum_{j=0}^{ \tau_\rs-1}\mathbb{I}_{\{y\}}(X_{j})\Biggr].
\end{equation}
In Section
\ref{subsec:cont} we give the standard interpretation of the Green function in terms of continuous time 
random processes.

We now come to the convergence of the Green functions along an exhaustion of the graph.
Let $(\Gs_n)_{n\geq 1}$ be an exhaustion of the infinite graph $\Gs$. Let $(Y^n_j)$ be the 
network random walk of $\Gs_n$ and $(X_j)$ be the killed random walk of $\Gs$. We introduce 
$\tau_\rs^n=\inf\{j>0:Y_j^n=\rs\}$ and $(X^n_j)=(Y^n_{j\wedge (\tau_\rs-1)})$,
the random walk on $\Gs_n$, killed 
at the vertex $\rs$. It is absorbed in finite time by $\rs$. Finally, 
$\tau_{\partial \Gs_n}=\inf\{j>0:X_j^n\notin \Gs_n\}=\inf\{j>0:X_j\notin \Gs_n\}$ 
(if the starting point belongs to $\Gs_n$) is the first exit time from the domain $\Gs_n$. 
\begin{lem}
  \label{lem:convergence_Green}
For any $x,y\in\Vs$, one has $\lim_{n\to\infty} G_n^m(x,y)=G^m(x,y)$.
\end{lem}

\begin{proof}
To use an interpretation with random walks,
we prove Lemma \ref{lem:convergence_Green} for the potential instead of the Green function; this is equivalent  
by~\eqref{eq:potential_vs_green}.
The potential function for the killed walk $(X^n_{j})$ is
\begin{equation*}
     V_n^m(x,y)=\mathbb{E}_x\Biggl[\sum_{j=0}^{\infty}\mathbb{I}_{\{y\}}(X^n_{j})\Biggr]
                       =\mathbb{E}_x\Biggl[\sum_{j=0}^{ \tau_\rs^n-1}\mathbb{I}_{\{y\}}(Y^n_{j})\Biggr].
\end{equation*}
The potential $V^m(x,y)$ for $(X_{j})$ is the same as 
above without the subscript $n$, see \eqref{eq:def_pot_krw}.
One has
\begin{equation*}
     V_n^m(x,y)=\mathbb{E}_x\Biggl[\sum_{j=0}^{ \tau_\rs^n-1}\mathbb{I}_{\{y\}}(X^n_{j});\tau_\rs^n<\tau_{\partial \Gs_n}\Biggr]
                        +\mathbb{E}_x\Biggl[\sum_{j=0}^{ \tau_\rs^n-1}\mathbb{I}_{\{y\}}(X^n_{j});\tau_\rs^n>\tau_{\partial \Gs_n}\Biggr].
\end{equation*}
In the first term we replace $X^n_{j}$ by $X_{j}$ (as $x\in\Gs_n$), $\tau_\rs^n$ by
$\tau_\rs$, and we use the monotone convergence theorem (as $n\to \infty$,
$\tau_{\partial \Gs_n}\to\infty$ monotonously). The first term goes to
$V^m(x,y)$. We now prove that the second term goes to $0$ as $n\to\infty$. It is
 less than $\mathbb{E}_x[\tau_\rs^n;\tau_\rs^n>\tau_{\partial \Gs_n}]$.
Conductances and masses are bounded away from 0 and
$\infty$, so $\tau_\rs^n$ is integrable and dominated by a geometric random
variable not depending on $n$.
We conclude since $\tau_{\partial \Gs_n}\to\infty$.
\end{proof}

\subsection{Laplacian operators and continuous time random processes}
\label{subsec:cont}

In this section we briefly recall the probabilistic interpretation of the Laplacian $\Delta^m$ on 
the infinite graph $\Gs$
(introduced in \eqref{equ:operator_general} of Section \ref{sec:defLap}). A similar interpretation
holds for Laplacian operators on other graphs (like on the finite graphs of Section \ref{subsec:rw_app}).

The Laplacian $\Delta^m$
is the generator of a continuous time Markov process $(X_t)$ on $\Gs$, augmented with an absorbing 
state (the \emph{root} $\rs$): when at $x$ at time $t$, the process waits an exponential time (with
parameter equal to the diagonal coefficient $d_x$), and then jumps to a neighbor of $x$ 
with probability \eqref{eq:P_Green}. For the same reasons as for $(X_j)$ and under the same hypotheses, 
the random process $(X_t)$ will be absorbed by the vertex $\rs$ in finite time. 

The matrix $Q^m$ in \eqref{eq:P_Green} is a substochastic matrix, corresponding
to the discrete time counterpart $(X_j)$ of $(X_t)$, just tracking the jumps.
The Green function 
$G^m(x,y)$ represents the total time spent at $y$ by the process $(X_t)$ started at $x$ at time $t=0$ before being absorbed.

\small
\bibliographystyle{alpha}
\bibliography{survey}

\end{document}